\documentclass{article}

\usepackage{amsmath}

\numberwithin{equation}{section} 

\counterwithin{figure}{section}
\counterwithin{table}{section}

\usepackage{amsxtra}

\usepackage{mathtools}

\usepackage{amssymb}

\usepackage[all,arrow]{xy}

\usepackage{geometry}
\usepackage{threeparttable}
\usepackage{diagbox}
\usepackage{paralist,bbding,pifont}

\usepackage[thmmarks]{ntheorem}
{
	\theoremstyle{nonumberplain}
	\theoremheaderfont{\bfseries}
	\theorembodyfont{\normalfont}
	\theoremsymbol{\mbox{$\Box$}}
	\newtheorem{pf}{Proof}
}
{
	\theoremstyle{plain}
	\theoremheaderfont{\bfseries}
	\theorembodyfont{\normalfont}
	\newtheorem{rmk}{Remark}[section]
}
{
	\theoremstyle{plain}
	\theoremheaderfont{\bfseries}
	\theorembodyfont{\normalfont}
	
}
{
	\theoremstyle{plain}
	\theoremheaderfont{\bfseries}
	\theorembodyfont{\normalfont}
	
}
{
	\theoremstyle{plain}
	\theoremheaderfont{\bfseries}
	\theorembodyfont{\normalfont}
	\newtheorem*{ackn}{Acknowledgements.}
}
\newtheorem{thm}{Theorem}[section]

\newtheorem{prop}{Proposition}[section]
\newtheorem{lem}{Lemma}[section]
{
	\theoremstyle{plain}
	\theoremheaderfont{\bfseries}
	\theorembodyfont{\normalfont}
	
}
{
	\theoremstyle{plain}
	\theoremheaderfont{\bfseries}
	\theorembodyfont{\normalfont}
	\newtheorem{cor}{Corollary}[section]
}
\newtheorem{Q}{Problem}

\usepackage{nicematrix,tikz}

\usepackage{amsmath,amssymb,tikz-cd}

\usepackage{dsfont}
\usetikzlibrary{matrix}
\usetikzlibrary{decorations.pathmorphing} 
\usetikzlibrary{positioning} 

\usepackage{tabularx,amsmath}
\usepackage{float}
\usepackage{booktabs}	
\usepackage{caption}   

\usepackage{hyperref}
\hypersetup{hypertex=true,
	colorlinks=true,
	linkcolor=red,
	anchorcolor=green,
	citecolor=blue}

\title{On the simplest simply connected rational homology $7$-spheres that are not $2$-connected}
\author{Fupeng Xu}
\date{}

\begin{document}
	
	\maketitle
		
	\begin{abstract}
		
		We give a complete classification of two families of simply connected $7$-manifolds: $\mathcal{G}^{3}(\mathrm{Wu})$-like manifolds and $\mathcal{G}^{3}_{p}\left(S^{5}\right)$-like manifolds for odd primes $p$. The former are non-spin with $H_{2}\cong H_{4}\cong \mathbb{Z}/2$ as their only nontrivial middle homology; the latter have $H_{2}\cong H_{4}\cong \mathbb{Z}/p$ as their sole nontrivial middle homology. These manifolds attain the minimal homological complexity among simply connected rational homology $7$-spheres that are not $2$-connected.
		
		We prove that Milnor's $\lambda$-invariant gives a bijection from the oriented diffeomorphism classes of $\mathcal{G}^{3}(\mathrm{Wu})$-like manifolds onto $\mathbb{Z}/7$, and each such manifold decomposes as the connected sum of a standard $\mathcal{G}^{3}(\mathrm{Wu})$-like manifold and a homotopy $7$-sphere. Analogously, the Eells-Kuiper $\mu$-invariant yields a bijection from the oriented diffeomorphism classes of $\mathcal{G}^{3}_{p}\left(S^{5}\right)$-like manifolds to $\mathbb{Z}/28$, with every manifold splitting as the connected sum of a standard $\mathcal{G}^{3}_{p}\left(S^{5}\right)$-like manifold and a homotopy $7$-sphere.
	\end{abstract}
	
	\tableofcontents
	
	\section{Introduction}

Unless otherwise stated, all manifolds are assumed to be smooth, connected, closed and oriented. All manifolds with boundary are assumed to be compact. All diffeomorphisms preserve orientations. $H_{q}(X)$ and $H^{q}(X)$ denote the (co)homology groups with coefficient $\mathbb{Z}$.

Rational homology spheres are closed manifolds whose rational homology groups coincide with those of the standard sphere. They occupy a distinguished position in topology: although they are homologically close to spheres, they often carry rich torsion phenomena encoded by linking pairings on their homology. As a result, they provide a natural testing ground for techniques from surgery theory, cobordism theory, and the study of manifold invariants in algebraic and differential topology.

In dimension $3$, rational homology spheres play a particularly central role. They arise naturally as Dehn surgery manifolds along knots and links in $S^3$ (\cite{Wal60,Lic62}), and form a fundamental class of objects in low-dimensional topology. Many powerful invariants, such as the Casson invariant and various Floer-theoretic invariants, are defined for or most naturally studied on rational homology $3$-spheres \cite{AMc90,Floer88}. Moreover, rational homology $3$-spheres frequently appear as boundaries of smooth or topological $4$-manifolds, where their properties interact deeply with intersection forms and gauge-theoretic techniques in four-dimensional topology \cite{Donaldson83}.

In dimensions greater than $4$, the situation becomes markedly different. From the viewpoint of surgery theory, the simply connected case forms the natural starting point. Such manifolds also arise naturally in the study of exotic spheres, smooth circle actions, and geometric structures such as positive curvature and Sasakian geometry (\cite{GZ00,BGN02}). In dimensions $5$ and $6$ we have a complete understanding of simply connected rational homology spheres due to the classification of simply connected manifolds (\cite{Smale62Spin5mfd,Barden65simplycnt5mfd,Wall1966Classification,Jupp1973,Zhubr00}). In dimension $7$ the topology of such manifolds becomes substantially richer. The complete classification of $2$-connected rational homology $7$-spheres is known due to the classification of $2$-connected manifolds (\cite{Kreck2018OnTC,CN19}), while a full classification of general simply connected rational homology $7$-spheres remains open. 

Motivated by these developments, it is natural to consider the following problem.
\begin{Q}
	Classify simply connected rational homology $7$-spheres.
\end{Q}

The purpose of this paper is to study the so-called \textbf{$\mathcal{G}^{3}(\mathrm{Wu})$-like manifolds} and \textbf{$\mathcal{G}^{3}_{p}\left(S^{5}\right)$-like manifolds}. These manifolds represent the simplest examples of simply connected rational homology $7$-spheres that are not $2$-connected, and we shall see later why they are the simplest. 
For a precise definition we begin with prototype manifolds which are obtained from gyrations. A \textbf{$k$-gyration} is a surgery on the product of a manifold and the $(k-1)$-sphere. It is originally defined by Gonz\'{a}lez Acu\~{n}a (\cite{GA75}) and appears implicitly in the study of intersections of quadrics and moment-angle manifolds by Gitler and L\'{o}pez de Medrano (\cite{GSLdM13}). In the special case $k=2$, Duan \cite{Duan2026} studies this operation in connection with free circle actions and refers to it as {suspension}. In addition, Chenery, Huang and Theriault (\cite{HT23,Huang2024,CT2025}) developed homotopy-theoretic analogues of gyrations and investigated their stabilization phenomena and homotopy invariance properties.

Let $\mathrm{Wu}=SU_{3}\big/SO_{3}$ denote the Wu manifold and consider $\mathcal{G}^{3}(\mathrm{Wu})$, the \textbf{$2$-gyration} of $\mathrm{Wu}$. Namely, first we form the trivial $S^{2}$-bundle $S^{2}\times\mathrm{Wu}$. The fiber inclusion $S^{2}\to S^{2}\times\mathrm{Wu}$ has trivial normal bundle, hence we can perform surgery along the fiber $S^{2}$. Since $\pi_{2}\left(SO_{5}\right)=0$, there is a unique framing while applying fiber surgery, and the resulting manifold is denoted by $\mathcal{G}^{3}(\mathrm{Wu})$.  See Figure \ref{Figure: The gyration on Wu manifold}.
\begin{figure}[h]
	\begin{center}
		\[
			\xymatrix{
				S^{2}\times D^{5}\ar@{^(-_>}[d] & S^{2}\times S^{4}\ar@{^(-_>}[d]\ar@{^(-_>}[r]\ar@{_(-_>}[l] & D^{3}\times S^{4}\ar@{^(-_>}[d]\\
				S^{2}\times\mathrm{Wu} & S^{2}\times\left(\mathrm{Wu}\setminus \mathring{D^{5}}\right)\ar@{_(-_>}[l]\ar@{^(-_>}[r] & \mathcal{G}^{3}(\mathrm{Wu})
			}
		\]
	\end{center}
	\caption{The gyration $\mathcal{G}^{3}(\mathrm{Wu})$}
	\label{Figure: The gyration on Wu manifold}
\end{figure}

It is routine to compute that $\mathcal{G}^{3}(\mathrm{Wu})$ is a simply connected $7$-manifold with only nontrivial middle homology groups $H_{2}\cong H_{4}\cong\mathbb{Z}\big/2$ and characteristic classes $w_{2}\neq0$, $w_3\neq0$, $w_4=w_{2}^{2}=0$, $p_1=0$ (Lemma \ref{Lemma: coh ring & char class of G_3(Wu)}). 

Next we introduce $\mathcal{G}^{3}_{p}\left(S^{5}\right)$. We require that $p$ is an odd prime. Then $\mathcal{G}^{3}_{p}\left(S^{5}\right)$ is the \textbf{$p$-twisted $2$-gyration} of $S^{5}$ constructed as follows. Form the trivial $S^{2}$-bundle $S^{2}\times S^{5}$, and $S^{2}\times S^{5}$ is simply connected with $\pi_{2}\left(S^{2}\times S^{5}\right)\cong\mathbb{Z}$. By Whitney's embedding theorem $p\in \pi_{2}\left(S^{2}\times S^{5}\right)$ is represented by a smooth embedding $S^{2}\xrightarrow{\iota_{p}} S^{2}\times S^{5}$, and different choices of smooth embeddings are smoothly isotopic. The normal bundle $\nu\left(\iota_{p}\right)$ is oriented and its classifying map lies in $\pi_{2}\left(BO_{5}\right)\cong\mathbb{Z}\big/2$, thereby being trivial. Hence we can perform surgery along the embedding $\iota_{p}$, and there is a unique framing while applying this surgery. The resulting manifold is then denoted by $\mathcal{G}^{3}_{p}\left(S^{5}\right)$. See Figure \ref{Figure: The gyration on S^5}, where with a slight abuse of notation we also denote by $S^{2}\times D^{5}\xrightarrow{\iota_{p}}S^{2}\times S^{5}$ the embedding of a tubular neighborhood of $\iota_{p}$.
\begin{figure}[H]
	\begin{center}
		\[
		\xymatrix{
			S^{2}\times D^{5}\ar@{^(-_>}[d]^{\iota_{p}} & S^{2}\times S^{4}\ar@{^(-_>}[d]\ar@{^(-_>}[r]\ar@{_(-_>}[l] & D^{3}\times S^{4}\ar@{^(-_>}[d]\\
			S^{2}\times S^{5} & \left(S^{2}\times S^{5}\right)\setminus \iota_{p}\left(S^{2}\times\mathring{D^{5}}\right) \ar@{_(-_>}[l]\ar@{^(-_>}[r] & \mathcal{G}^{3}_{p}\left(S^{5}\right)
		}
		\]
	\end{center}
	\caption{The gyration $\mathcal{G}^{3}_{p}\left(S^{5}\right)$}
	\label{Figure: The gyration on S^5}
\end{figure}
We emphasize that the notation $\mathcal{G}^{3}_{p}\left(S^{5}\right)$ differs from the standard convention $\mathcal{G}^{k}_{\tau}(M)$ (e.g. \cite{CT2025}), where the surgery annihilates the homology class represented by the product sphere itself, rather than a nontrivial multiple of it, and the subscript $\tau$ specifies the surgery framing. In contrast, our construction annihilates the class represented by the $p$-multiple of the sphere, and the framing is unique.

It is routine to compute that $\mathcal{G}^{3}_{p}\left(S^{5}\right)$ is a simply connected $7$-manifold with only nontrivial middle homology groups $H_{2}\cong H_{4}\cong\mathbb{Z}\big/p$. Moreover, the Stiefel-Whitney classes and Pontryagin classes of $\mathcal{G}^{3}_{p}\left(S^{5}\right)$ vanish, and the only nontrivial products in mod $p$ cohomology ring of $\mathcal{G}^{3}_{p}\left(S^{5}\right)$ are those detected by Poincar\'{e} duality (Lemma \ref{Lemma: coh ring & char class of G_3^p S^5}). 

$\mathcal{G}^{3}(\mathrm{Wu})$ exhibits the simplest simply connected rational homology $7$-sphere that is not $2$-connected, in the sense that $\mathcal{G}^{3}(\mathrm{Wu})$ has the minimal topological complexity among these manifolds: If $M$ is a rational homology $7$-sphere that is simply connected and not $2$-connected, then $\mathbb{Z}\big/2$ is the smallest possible second homotopy and homology group, and by Poincar\'{e} duality and universal coefficient theorem, the torsion subgroup of $H_{4}(M)$ is isomorphic to $H_{2}(M)$, and $\mathcal{G}^{3}(\mathrm{Wu})$ realizes the minimality. If we require the second homotopy group to be cyclic of odd prime order $p$, then $\mathcal{G}^{3}_{p}\left(S^{5}\right)$ realizes the minimal topological complexity among rational homology $7$-spheres with prescribed connectivity and second homotopy group.

A $\mathcal{G}^{3}(\mathrm{Wu})$-like manifold is a simply connected non-spin $7$-manifold $M$ whose only nontrivial middle homology groups are $H_{2}(M)\cong H_{4}(M)\cong\mathbb{Z}\big/2$. And a $\mathcal{G}^{3}_{p}\left(S^{5}\right)$-like manifold is a simply connected spin $7$-manifold $M$ whose only nontrivial middle homology groups are $H_{2}(M)\cong H_{4}(M)\cong\mathbb{Z}\big/p$. We adapt the convention that when we refer to $\mathcal{G}^{3}_{p}\left(S^{5}\right)$ or $\mathcal{G}^{3}_{p}\left(S^{5}\right)$-like manifolds, $p$ indicates an odd prime. Thus a $\mathcal{G}^{3}_{p}\left(S^{5}\right)$-like manifold is a $\mathbb{Z}\big/2$-homology $7$-sphere by universal coefficient theorem, and is automatically spin. It can be shown that a $\mathcal{G}^{3}(\mathrm{Wu})$-like manifold has the same characteristic classes and mod $2$ cohomology ring as $\mathcal{G}^{3}(\mathrm{Wu})$ (Lemma \ref{Lemma: equiv char of G_3(Wu)-like mfd}), and a $\mathcal{G}^{3}_{p}\left(S^{5}\right)$-like manifold has the same characteristic classes and mod $p$ cohomology ring as $\mathcal{G}^{3}_{p}\left(S^{5}\right)$ (Lemma \ref{Lemma: equiv char of G_3^p S^5-like mfd}).
\begin{Q}
	Classify all $\mathcal{G}^{3}(\mathrm{Wu})$-like manifolds.
\end{Q}
\begin{Q}
	Classify all $\mathcal{G}^{3}_{p}\left(S^{5}\right)$-like manifolds.
\end{Q}

$\mathcal{G}^{3}(\mathrm{Wu})$-like manifolds can be classified by Milnor's $\lambda$-invariant (\cite{Milnor56}), which is originally defined for homotopy $7$-spheres and reveals for the first time the existence of exotic smooth structure on the $7$-sphere $S^{7}$. 

Indeed the $\lambda$-invariant can be defined for any rational homology $7$-sphere. Let $M$ be a rational homology $7$-sphere. Since $\Omega_{7}^{SO}=0$, $M$ admits an oriented coboundary $V$ and we may assume that its signature $\sigma(V,M)$ vanishes. Since $M$ is a rational homology $7$-sphere, by the long exact sequence of rational cohomology groups associated to the pair $(V,M)$ we see that the homomorphism $H^{4}(V,M;\mathbb{Q})\to H^{4}(V;\mathbb{Q})$ induced by inclusion is an isomorphism, hence the first rational Pontryagin class $p_{1}^{\mathbb{Q}}(V)$ admits a unique lifting $\widetilde{p_{1}^{\mathbb{Q}}(V)}\in H^{4}(V,M;\mathbb{Q})$. Then we define the $\lambda$-invariant of $M$ as follows:
\begin{align*}
	\Lambda\left(V\right)&:=\left<\widetilde{p_{1}^{\mathbb{Q}}(V)}^{2},[V,M]_{\mathbb{Q}}\right>\in\mathbb{Q},\\
	\lambda(M)&:=\Lambda(V)\ \mathrm{mod}\ 7\in\mathbb{Q}\big/7\mathbb{Z}.
\end{align*}
Here if $R$ is a coefficient ring and $V$ is a compact $R$-oriented $n$-manifold $V$ with (possibly empty) boundary, the fundamental class is denoted by $[V,\partial V]_{R}\in H_{n}(V,\partial V;R)$, and when $R=\mathbb{Z}$ the subscript $R$ is omitted.

This invariant is well-defined, which can be reasoned by Hirzebruch's signature theorem as follows. Suppose $V'$ is another oriented coboundary of $M$ with vanishing signature. Form the closed $8$-manifold $X=V\cup_{M}\left(-V'\right)$, where $-V'$ has the same underlying manifold as $V'$ and is equipped with the opposite orientation, and we have
\begin{align*}
	\Lambda\left(V\right)-\Lambda\left(V'\right)&=\left<
	p_{1}^{\mathbb{Q}}(X)^{2},[X]_{\mathbb{Q}}\right>\\
	&=\left<p_{1}(X)^{2},[X]\right>\\
	&=45\sigma(X)-7\left<p_{2}(X),[X]\right>\\
	&=45\left(\sigma(V,M)-\sigma\left(V',M\right)\right)-7\left<p_{2}(X),[X]\right>\\
	&=-7\left<p_{2}(X),[X]\right>\in7\mathbb{Z}.
\end{align*}
Here the third equality follows from Hirzebruch's signature theorem, and the first and fourth equalities are results of Lemma \ref{Lemma: char nums additive wrt gluing via boundary}. Therefore, $\Lambda\left(V\right)$ and $\Lambda\left(V'\right)$ have the same image in $\mathbb{Q}\big/7\mathbb{Z}$.

If rational homology $7$-sphere $M$ has vanishing fourth cohomology group, the homomorphism $H^{4}(V,M)\to H^{4}(V)$ is epic and we consider the lifting $\widetilde{p_{1}(V)}\in H^{4}(V,M)$ of integral Pontryagin class $p_{1}(V)$. Then the $\lambda$-invariant of $M$ can also be defined using integral characteristic number as
\begin{align*}
	\Lambda\left(V\right)&:=\left<\widetilde{p_{1}(V)}^{2},[V,M]\right>\in\mathbb{Z},\\
	\lambda(M)&:=\Lambda(V)\ \mathrm{mod}\ 7\in\mathbb{Z}\big/7.
\end{align*}
Here the value of $\Lambda\left(V\right)$ does not depend on the choices of liftings (Lemma \ref{Lemma: Relative char nums}). In particular, the integral $\lambda$-invariant applies to $\mathcal{G}^{3}(\mathrm{Wu})$-like manifolds since they have vanishing fourth cohomology groups.

\begin{thm}\label{Theorem: main}
	\begin{compactenum}
		\item[\ ]
		\item Two $\mathcal{G}^{3}(\mathrm{Wu})$-like manifolds are diffeomorphic if and only if they have the same $\lambda$-invariant.
		\item The $\lambda$-invariant induces a bijection from the diffeomorphism classes of $\mathcal{G}^{3}(\mathrm{Wu})$-like manifolds to $\mathbb{Z}\big/7$. 
		\item For each $\mathcal{G}^{3}(\mathrm{Wu})$-like manifold $M$ there is a homotopy $7$-sphere $\Sigma$ such that $M\cong \mathcal{G}^{3}(\mathrm{Wu})\#\Sigma$.
	\end{compactenum}
\end{thm}
In particular, all $\mathcal{G}^{3}(\mathrm{Wu})$-like manifolds are homeomorphic and homotopy equivalent.

$\mathcal{G}^{3}_{p}\left(S^{5}\right)$-like manifolds can be classified by Eells-Kuiper's $\mu$-invariant (\cite{EKinv}), which is originally defined for certain smooth $(4n-1)$-dimensional spin manifolds including spin rational homology $7$-spheres and plays an important role in distinguishing exotic smooth structures on $S^{7}$ and $S^{11}$. 

The $\mu$-invariant of a spin rational homology $7$-sphere $M$ is defined similarly to the $\lambda$-invariant. Since $\Omega_{7}^{Spin}=0$, $M$ admits an oriented spin coboundary $V$ and we may assume that its signature $\sigma(V,M)$ vanishes. For a spin manifold $X$, its first spin Pontryagin class $\overline{p_{1}}(X)\in H^{4}(X)$ is defined and satisfies the relation $p_{1}(X)=2\overline{p_{1}}(X)$. Let us denote by \(\overline{p_{1}}^{\mathbb{Q}}(X)\) the rational first spin Pontryagin class, and the $\mu$-invariant of $M$ is defined as follows:
\begin{align*}
	\mathrm{M}\left(V\right)&:=\left<\widetilde{\overline{p_{1}}^{\mathbb{Q}}(V)}^{2},[V,M]_{\mathbb{Q}}\right>\in\mathbb{Q},\\
	\mu(M)&:=\mathrm{M}(V)\ \mathrm{mod}\ 224\in\mathbb{Q}\big/224\mathbb{Z}.
\end{align*}

This invariant is well-defined, which can be reasoned by Hirzebruch's integrality theorem of $\widehat{A}$-genus as follows. Suppose $V'$ is another oriented spin coboundary of $M$ with vanishing signature. Form the closed $8$-manifold $X=V\cup_{M}\left(-V'\right)$, where $-V'$ has the same underlying manifold as $V'$ and is equipped with the opposite orientation, and $X$ is an oriented spin manifold with
\begin{align*}
	\mathrm{M}\left(V\right)-\mathrm{M}\left(V'\right)&=\left<
	\overline{p_{1}}^{\mathbb{Q}}(X)^{2},[X]_{\mathbb{Q}}\right>\\
	&=\left<\overline{p_{1}}(X)^{2},[X]\right>\\
	&=224\widehat{A}(X)+\sigma(X)\\
	&=224\widehat{A}(X)+\sigma(V,M)-\sigma\left(V',M\right)\\
	&=224\widehat{A}(X)\in 224\mathbb{Z}.
\end{align*}
Here the third equality is the result of eliminating $p_{2}$-terms of $\widehat{A}$-genus and $L$-genus, and the last one follows from the Hirzebruch's integrality theorem of $\widehat{A}$-genus. Therefore, $\mathrm{M}\left(V\right)$ and $\mathrm{M}\left(V'\right)$ have the same image in $\mathbb{Q}\big/224\mathbb{Z}$.

Just as the previous case, if spin rational homology $7$-sphere $M$ has vanishing fourth cohomology group, then the invariant $\mathrm{M}$ of coboundary is an integer, hence $\mu(M)\in\mathbb{Z}\big/224$. 

\begin{thm}\label{Theorem: main, spin}
	Let $p$ be an odd prime.
	\begin{compactenum}
		\item Two $\mathcal{G}^{3}_{p}\left(S^{5}\right)$-like manifolds are diffeomorphic if and only if they have the same $\mu$-invariant.
		\item The $\mu$-invariant induces a bijection from the diffeomorphism classes of $\mathcal{G}^{3}_{p}\left(S^{5}\right)$-like manifolds to $\mathbb{Z}\big/28$. 
		\item For each $\mathcal{G}^{3}_{p}\left(S^{5}\right)$-like manifold $M$ there is a homotopy $7$-sphere $\Sigma$ such that $M\cong \mathcal{G}^{3}_{p}\left(S^{5}\right)\#\Sigma$.
	\end{compactenum}
\end{thm}
In particular, for a fixed odd prime $p$, all $\mathcal{G}^{3}_{p}\left(S^{5}\right)$-like manifolds are homeomorphic and homotopy equivalent.

\begin{rmk}
	It can be shown that two rational homology $7$-spheres are diffeomorphic only if they have the same $\lambda$-invariant, and two spin rational homology $7$-spheres are diffeomorphic only if they have the same $\mu$-invariant.
	Theorems \ref{Theorem: main} and \ref{Theorem: main, spin} implies that the converse is also true when we restrict to $\mathcal{G}^{3}(\mathrm{Wu})$-like manifolds and $\mathcal{G}^{3}_{p}\left(S^{5}\right)$-like manifolds.
\end{rmk}

\begin{rmk}
	The primary tool in this work is the modified surgery theory developed in \cite{SurgeryAndDuality}. The main technical difficulty lies in the analysis of the associated surgery obstruction.
	
	Existing classification results for simply connected $7$-manifolds are largely centered on cases where the second homotopy group is torsion-free (see, for instance, \cite{KS88,Kreck1991SomeNH,Kreck1991SomeNHCorrection,Wang2018CohCP2timesS3,Kreck2018OnTC}). In this setting, the structure of the surgery obstruction is sufficiently well understood, allowing for a relatively complete analysis. By contrast, the presence of torsion in the second homotopy group introduces additional complexities; the obstruction becomes sensitive to the torsion subgroup in a way that does not readily follow from the torsion-free arguments. To the best of the author’s knowledge, a systematic analysis in the presence of torsion in $\pi_2$ is not yet available.
	
	The purpose of this paper is to examine the simplest nontrivial torsion case. While we focus on this specific setting, it already exhibits features absent in the torsion-free regime and provides a first step toward understanding how torsion influences the surgery obstruction.
\end{rmk}

This paper is organized as follows. In Section \ref{Section: G_3(Wu) and G_3(Wu)-like mfd} we first compute the cohomology rings and characteristic classes of $\mathcal{G}^{3}(\mathrm{Wu})$ and $\mathcal{G}^{3}_{p}\left(S^{5}\right)$, then we determine the normal $2$-type, characteristic classes and cohomology rings of $\mathcal{G}^{3}(\mathrm{Wu})$-like manifolds and $\mathcal{G}^{3}_{p}\left(S^{5}\right)$-like manifolds. In Section \ref{Section: Identify the surgery obstructions} we identify the relevant surgery obstructions. In Section \ref{Section: Prove of main thm} we show when the obstructions are ``elementary'' and prove Statement 1 of Theorems \ref{Theorem: main} and \ref{Theorem: main, spin}. Then in Section \ref{Section: Compute s inv} we compute invariants $\lambda$ of $\mathcal{G}^{3}(\mathrm{Wu})$-like manifolds and invariants $\mu$ of $\mathcal{G}^{3}_{p}\left(S^{5}\right)$-like manifolds, proving Statements 2 and 3 of Theorems \ref{Theorem: main} and \ref{Theorem: main, spin}. In section \ref{Section: Compute bordism gps} we compute the bordism groups $\Omega_{q}^{Spin}\left(K\left(\mathbb{Z}\big/p,2\right)\right)$ for odd primes $p$ and $q=7,8$, which are required while analyzing surgery obstructions for $\mathcal{G}^{3}_{p}\left(S^{5}\right)$-like manifolds. In Section \ref{Section: Arf inv} we explore certain Arf type invariants. They play an important role in the analysis of surgery obstructions and are also of independent interest.

\begin{ackn}
	The author is grateful to Professor Matthias Kreck, Professor Yang Su, Professor Yi Jiang and Tenglin Hu for their valuable suggestions and comments. The author also thanks Professor Zhiyou Wu, Xurui Liu for discussions on Lemma \ref{Lemma: Existence of integral primitive isotropic elt} and Professor Zhilei Zhang, Dr. Jiahao Hu for discussions on Adams spectral sequence. The author’s research was supported by NSFC 12471069.
\end{ackn}
	
	\section{$\mathcal{G}^{3}(\mathrm{Wu})$-like manifolds and  $\mathcal{G}^{3}_{p}\left(S^{5}\right)$-like manifolds}\label{Section: G_3(Wu) and G_3(Wu)-like mfd}

In this section we first compute the cohomology rings and characteristic classes of $\mathcal{G}^{3}(\mathrm{Wu})$ and determine the normal $2$-type of $\mathcal{G}^{3}(\mathrm{Wu})$-like manifolds. Then we compute the cohomology rings and characteristic classes of $\mathcal{G}^{3}_{p}\left(S^{5}\right)$ and determine the normal $2$-type of $\mathcal{G}^{3}_{p}\left(S^{5}\right)$-like manifolds.

\subsection{$\mathcal{G}^{3}(\mathrm{Wu})$ and $\mathcal{G}^{3}(\mathrm{Wu})$-like manifolds}

\begin{lem}\label{Lemma: coh ring & char class of G_3(Wu)}
	$\mathcal{G}^{3}(\mathrm{Wu})$ is a closed simply-connected $7$-manifold. Its integral homology groups and cohomology groups with coefficients $\mathbb{Z}$ and $\mathbb{Z}\big/2$ are given as in Table \ref{Table: homology and coh gps of G_3(Wu)}.
	\begin{table}[h]
		\centering
		\caption{Homology and cohomology groups of $\mathcal{G}^{3}(\mathrm{Wu})$}
		\begin{tabular}{ccccccccc}
			\toprule[1pt]
			$q$ & $0$ & $1$ & $2$ & $3$ & $4$ & $5$ & $6$ & $7$\\
			\midrule[0.3pt]
			$H_{q}\left(\mathcal{G}^{3}(\mathrm{Wu})\right)$ & $\mathbb{Z}$ & $0$ & $\mathbb{Z}\big/2$ & $0$ & $\mathbb{Z}\big/2$ & $0$ & $0$ & $\mathbb{Z}$ \\
			$H^{q}\left(\mathcal{G}^{3}(\mathrm{Wu})\right)$ & $\mathbb{Z}$ & $0$ & $0$ & $\mathbb{Z}\big/2$ & $0$ & $\mathbb{Z}\big/2$ & $0$ & $\mathbb{Z}$ \\
			$H^{q}\left(\mathcal{G}^{3}(\mathrm{Wu});\mathbb{Z}\big/2\right)$ & $\mathbb{Z}\big/2$ & $0$ & $\mathbb{Z}\big/2$ & $\mathbb{Z}\big/2$ & $\mathbb{Z}\big/2$ & $\mathbb{Z}\big/2$ & $0$ & $\mathbb{Z}\big/2$ \\
			\bottomrule[1pt]
		\end{tabular}
		\label{Table: homology and coh gps of G_3(Wu)}
	\end{table}
	
		Moreover, $p_{1}\left(\mathcal{G}^{3}(\mathrm{Wu})\right)=0$, $w_{2}\left(\mathcal{G}^{3}(\mathrm{Wu})\right)\neq0$, $w_{3}\left(\mathcal{G}^{3}(\mathrm{Wu})\right)\neq0$, $w_{2}\left(\mathcal{G}^{3}(\mathrm{Wu})\right)^{2}=w_{4}\left(\mathcal{G}^{3}(\mathrm{Wu})\right)=0$, and the only non-trivial cup products in $H^{*}\left(\mathcal{G}^{3}(\mathrm{Wu});\mathbb{Z}\big/2\right)$ are those detected by Poincar\'{e} duality.
\end{lem}

\begin{pf}
	The Wu manifold is a simply connected $5$-manifold with the only non-trivial middle integral homology group $H^{2}(\mathrm{Wu})\cong\mathbb{Z}\big/2$, characteristic classes $w_{2}(\mathrm{Wu})$, $w_{3}(\mathrm{Wu})\neq0$, $p_{1}(\mathrm{Wu})=0$ and the only non-zero Stiefel-Whitney number $w_{2}w_{3}(\mathrm{Wu}):=\left<w_{2}(\mathrm{Wu})w_{3}(\mathrm{Wu}),[\mathrm{Wu}]\right>\neq0$.
	
	Then we apply van Kampen theorem and Mayer-Vietoris sequence to decompositions of $S^{2}\times\mathrm{Wu}$ and $\mathcal{G}^{3}(\mathrm{Wu})$ given in Figure \ref{Figure: The gyration on Wu manifold}. It is routine to show that $\mathcal{G}^{3}(\mathrm{Wu})$ is simply connected and to compute the homology groups, cohomology groups and characteristic classes as stated above.
	
	It remains to determine the ring structure of $H^{*}\left(\mathcal{G}^{3}(\mathrm{Wu});\mathbb{Z}\big/2\right)$. Let $\left(w_{3}\left(\mathcal{G}^{3}(\mathrm{Wu})\right)\right)^{*}\in H^{4}\left(\mathcal{G}^{3}(\mathrm{Wu});\mathbb{Z}\big/2\right)$ be the Poincar\'{e} dual class of $w_{3}\left(\mathcal{G}^{3}(\mathrm{Wu})\right)$ 
	and define $\left(w_{2}\left(\mathcal{G}^{3}(\mathrm{Wu})\right)\right)^{*}\in H^{5}\left(\mathcal{G}^{3}(\mathrm{Wu});\mathbb{Z}\big/2\right)$ likewise. Recall that all oriented $7$-manifolds bound and all their Stiefel-Whitney numbers vanish. In particular $w_{3}w_{4}\left(\mathcal{G}^{3}(\mathrm{Wu})\right)=w_{2}^{2}w_{3}\left(\mathcal{G}^{3}(\mathrm{Wu})\right)=0$. Hence we must have $w_{2}\left(\mathcal{G}^{3}(\mathrm{Wu})\right)^{2}=w_{4}\left(\mathcal{G}^{3}(\mathrm{Wu})\right)=0$ and $w_{2}\left(\mathcal{G}^{3}(\mathrm{Wu})\right)w_{3}\left(\mathcal{G}^{3}(\mathrm{Wu})\right)=0$, and the only non-trivial cup products are those detected by Poincar\'{e} duality.
\end{pf}

\begin{lem}\label{Lemma: equiv char of G_3(Wu)-like mfd}
	Let $M$ be a $\mathcal{G}^{3}(\mathrm{Wu})$-like manifold. Then $p_{1}(M)=0$, $w_{3}(M)\neq0$, $w_{2}(M)^{2}=w_{4}(M)=0$, and the only non-trivial cup products in $H^{*}\left(M;\mathbb{Z}\big/2\right)$ are those detected by Poincar\'{e} duality.
\end{lem}
Hence a $\mathcal{G}^{3}(\mathrm{Wu})$-like manifold also admits the same mod $2$ cohomology ring and remaining characteristic classes as those of $\mathcal{G}^{3}(\mathrm{Wu})$. To prove this lemma we need the \textit{normal $2$-type} and \textit{normal $2$-smoothing} of $\mathcal{G}^{3}(\mathrm{Wu})$-like manifolds. We briefly introduce them here, and we refer to \cite[Section 2]{SurgeryAndDuality} for more details. 

Given a manifold $M$, the classifying map of its stable normal bundle $M\xrightarrow{\nu}BO$ and a fibration $\mathcal{B}\xrightarrow{\mathcal{F}}BO$, a lifting $M\xrightarrow{\widetilde{\nu}}\mathcal{B}$ along $\mathcal{F}$ is called a \textbf{$k$-smoothing in $(\mathcal{B},\mathcal{F})$} if it is a $(k+1)$-equivalence, $(\mathcal{B},\mathcal{F})$ is called \textbf{$k$-universal} if the fibre of $\mathcal{F}$ is connected and its homotopy groups vanish in dimensions greater than $k$. It follows from homotopy theory that for each manifold $M$ there is a $k$-universal fibration $\left(\mathcal{B}^{k},\mathcal{F}^{k}\right)$ such that $M$ admits a $k$-smoothing in $\left(\mathcal{B}^{k},\mathcal{F}^{k}\right)$, and the homotopy fiber of $\mathcal{F}^{k}$ is unique up to homotopy equivalence. The \textbf{normal $2$-type} of $M$ is defined as the fibre homotopy type of $\mathcal{F}^{k}$, or equivalently as the pair $\left(\mathcal{B}^{k},\mathcal{F}^{k}\right)$, and for simplicity a normal $k$-smoothing of $M$ in $\left(\mathcal{B}^{k},\mathcal{F}^{k}\right)$ would be abbreviated as a normal $k$-smoothing of $M$.

\begin{lem}\label{Lemma: normal 2-type of G_3(Wu)-like mfds}
	Let $M$ be a $\mathcal{G}^{3}(\mathrm{Wu})$-like manifold. The normal $2$-type and normal $2$-smoothing of $M$ are given as in Figure \ref{Figure: normal 2-type & normal 2-str of G_3(Wu)-like mfd}. 
	\begin{figure}[h]
		\begin{center}
			\[
				\xymatrix{
					& BSO\ar[d]^{B\phi}\\
					M\ar[ru]^{\widetilde{\nu}}\ar[r]^{\nu} & BO
				}
			\]
		\end{center}
		\caption{The normal $2$-type and normal $2$-smoothing of $\mathcal{G}^{3}(\mathrm{Wu})$-like manifold $M$}
		\label{Figure: normal 2-type & normal 2-str of G_3(Wu)-like mfd}
	\end{figure}
	Here $SO\xrightarrow{\phi}O$ is the canonical map, $BSO\xrightarrow{B\phi}BO$ is the induced map and $M\xrightarrow{\nu}BO$ (resp. $M\xrightarrow{\widetilde{\nu}}BSO$) classifies the stable normal bundle (resp. stable oriented normal bundle) of $M$.
\end{lem}

To prove Lemmata \ref{Lemma: equiv char of G_3(Wu)-like mfd} and \ref{Lemma: normal 2-type of G_3(Wu)-like mfds} we need more results about the integral homology and cohomology of $BSO$. They are collected in the following lemma, and they are also important while analyzing the surgery obstruction in Section \ref{Section: Identify the surgery obstructions}.

\begin{lem}\label{Lemma: integral coh of BSO}
	The cohomology groups of $BSO$ up to dimension $6$ and the isomorphism type of its homology up to dimension $5$ are given as in Table \ref{Table: homology and coh gps of BSO}.
	\begin{table}[H]
		\centering
		\caption{Homology and cohomology groups of $BSO$}
		\begin{tabular}{ccccccc}
			\toprule[1pt]
			$q$ & $1$ & $2$ & $3$ & $4$ & $5$ & $6$ \\
			\midrule[0.3pt]
			$H^{q}\left(BSO\right)$ & $0$ & $0$ & $\mathbb{Z}\big/2\left\{\delta w_{2}\right\}$ & $\mathbb{Z}\left\{p_{1}\right\}$ & $\mathbb{Z}\big/2\left\{\delta w_{4}\right\}$ & $\mathbb{Z}\big/2\left\{\left(\delta w_{2}\right)^{2}\right\}$\\
			$H_{q}\left(BSO\right)$ & $0$ & $\mathbb{Z}\big/2$ & $0$ & $\mathbb{Z}\oplus\mathbb{Z}\big/2$ & $\mathbb{Z}\big/2$ &\\
			\bottomrule[1pt]
		\end{tabular}
		\label{Table: homology and coh gps of BSO}
	\end{table}
	
	Here $w_{i}$ and $p_{i}$ are the universal characteristic classes, and $H^{q}\left(BSO;\mathbb{Z}\big/2\right)\xrightarrow{\delta}H^{q+1}\left(BSO\right)$ is the Bockstein homomorphism associated to the short exact sequence $0\to\mathbb{Z}\to\mathbb{Z}\to\mathbb{Z}\big/2\to0$ of coefficients.
\end{lem}

\begin{pf}
	It can be deduced from the fibration $S^{n}\to BSO_{n}\to BSO_{n+1}$ that the stabilization map $BSO_{n}\to BSO$ is $(n-1)$-connected. Hence when $n\geqslant7$, the integral cohomology groups of $BSO$ and $BSO_{n}$ are isomorphic up to dimension $6$, and it suffices to compute $H^{q}\left(BSO_{7}\right)$. 
	
	By \cite[Theorem 1.5]{Ed82}, up to degree $7$ we have
	$$H^{*\leqslant7}\left(BSO_{7}\right)\cong\frac{\mathbb{Z}\left[p_{1},\delta\left(w_{2}\right),\delta\left(w_{4}\right),\delta\left(w_{6}\right)\right]}{\left<2\delta\left(w_{2}\right),2\delta\left(w_{4}\right),2\delta\left(w_{6}\right)\right>}.$$
	Hence it is straightforward to deduce the cohomology groups of $BSO_{7}$ up to dimension $6$, and it is routine to apply universal coefficient theorem and deduce the homology groups of $BSO_{7}$ up to dimension $5$.
\end{pf}

\begin{pf}[of Lemma \ref{Lemma: normal 2-type of G_3(Wu)-like mfds}]
	It is clear that $B\phi$ is just the universal $2$-sheeted covering. Hence $B\phi$ is $0$-universal and thus also $2$-universal. It remains to justify that $\widetilde{\nu}$ is a $3$-equivalence, namely the induced homomorphism $\pi_{3}\left(\widetilde{\nu}\right)$ on the third homotopy groups is epic and $\pi_{i}\left(\widetilde{\nu}\right)$ are isomorphisms for $i\leqslant2$.
	
	$\pi_{1}\left(\widetilde{\nu}\right)$ is automatically an isomorphism as $M$ and $BSO$ are both simply connected. Since $M$ is non-spin, $w_{2}(M)=\widetilde{\nu}^{*}w_{2}\neq0$ and $H^{2}\left(BSO;\mathbb{Z}\big/2\right)\xrightarrow{\widetilde{\nu}^{*}}H^{2}\left(M;\mathbb{Z}\big/2\right)$ is an isomorphism, and it follows from Lemma \ref{Lemma: integral coh of BSO} and universal coefficient theorem that $H_{2}(M)\xrightarrow{\widetilde{\nu}_{*}}H_{2}(BSO)$ is also an isomorphism. Since $M$ and $BSO$ are both simply connected, by Hurewicz theorem $\pi_{2}\left(\widetilde{\nu}\right)$ is also an isomorphism. Finally, $\pi_{3}\left(\widetilde{\nu}\right)$ is automatically an epimorphism since $\pi_{3}(BSO)=\pi_{2}(SO)=\pi_{2}(Spin)=0$.
\end{pf}

\begin{pf}[of Lemma \ref{Lemma: equiv char of G_3(Wu)-like mfd}]
	Let $M$ be a $\mathcal{G}^{3}(\mathrm{Wu})$-like manifold. By the isomorphim type of cohomology groups it remains to show that $w_{3}(M)\neq0$ and $w_{2}(M)^{2}=w_{4}(M)=0$.
	
	It follows from Lemma \ref{Lemma: normal 2-type of G_3(Wu)-like mfds} that the classifying map of stable oriented normal bundle $M\xrightarrow{\widetilde{\nu}}BSO$ is $3$-connected. Hence according to the long exact sequence of cohomology groups with coefficient $\mathbb{Z}\big/2$ associated to the pair $(BSO,M)$, the homomorphism $H^{3}\left(BSO;\mathbb{Z}\big/2\right)\xrightarrow{\widetilde{\nu}^{*}}H^{3}\left(M;\mathbb{Z}\big/2\right)$ is monic. It is known that 
	$$H^{*}\left(BSO;\mathbb{Z}\big/2\right)=\mathbb{Z}\big/2\left[w_{i}:i\geqslant2\right],$$
	in particular $H^{3}\left(BSO;\mathbb{Z}\big/2\right)=\mathbb{Z}\big/2\left\{w_{3}\right\}$. By assumption we also have $H^{3}\left(M;\mathbb{Z}\big/2\right)\cong\mathbb{Z}\big/2$. Hence the monomorphism $H^{3}\left(BSO;\mathbb{Z}\big/2\right)\xrightarrow{\widetilde{\nu}^{*}}H^{3}\left(M;\mathbb{Z}\big/2\right)$ is actually an isomorphism and $\widetilde{\nu}^{*}\left(w_{3}\right)=w_{3}(M)\neq0$.
	
	Again $M$ is null-bordant, and all its Stiefel-Whitney numbers vanish. Hence the same argument as in the proof of Lemma \ref{Lemma: coh ring & char class of G_3(Wu)} shows that $w_{2}(M)^{2}=w_{4}(M)=0$ and the only non-trivial cup products of cohomology ring with coefficient $\mathbb{Z}\big/2$ are those detected by Poincar\'{e} duality.
\end{pf}

\subsection{$\mathcal{G}^{3}_{p}\left(S^{5}\right)$ and $\mathcal{G}^{3}_{p}\left(S^{5}\right)$-like manifolds}

\begin{lem}\label{Lemma: coh ring & char class of G_3^p S^5}
	$\mathcal{G}^{3}_{p}\left(S^{5}\right)$ is a closed simply-connected stably parallelizable $7$-manifold. In particular, $p_{1}\left(\mathcal{G}^{3}_{p}\left(S^{5}\right)\right)=0$, $w_{2}\left(\mathcal{G}^{3}_{p}\left(S^{5}\right)\right)=0$. Its integral homology groups and cohomology groups with coefficients $\mathbb{Z}$ and $\mathbb{Z}\big/p$ are given as in Table \ref{Table: homology and coh gps of G_3^p S^5}.
	\begin{table}[h]
		\centering
		\caption{Homology and cohomology groups of $\mathcal{G}^{3}_{p}\left(S^{5}\right)$}
		\begin{tabular}{ccccccccc}
			\toprule[1pt]
			$q$ & $0$ & $1$ & $2$ & $3$ & $4$ & $5$ & $6$ & $7$\\
			\midrule[0.3pt]
			$H_{q}\left(\mathcal{G}^{3}_{p}\left(S^{5}\right)\right)$ & $\mathbb{Z}$ & $0$ & $\mathbb{Z}\big/p$ & $0$ & $\mathbb{Z}\big/p$ & $0$ & $0$ & $\mathbb{Z}$ \\
			$H^{q}\left(\mathcal{G}^{3}_{p}\left(S^{5}\right)\right)$ & $\mathbb{Z}$ & $0$ & $0$ & $\mathbb{Z}\big/p$ & $0$ & $\mathbb{Z}\big/p$ & $0$ & $\mathbb{Z}$ \\
			$H^{q}\left(\mathcal{G}^{3}_{p}\left(S^{5}\right);\mathbb{Z}\big/p\right)$ & $\mathbb{Z}\big/p$ & $0$ & $\mathbb{Z}\big/p$ & $\mathbb{Z}\big/p$ & $\mathbb{Z}\big/p$ & $\mathbb{Z}\big/p$ & $0$ & $\mathbb{Z}\big/p$ \\
			\bottomrule[1pt]
		\end{tabular}
		\label{Table: homology and coh gps of G_3^p S^5}
	\end{table}
	
	Moreover, the only non-trivial cup products in $H^{*}\left(\mathcal{G}^{3}_{p}\left(S^{5}\right);\mathbb{Z}\big/p\right)$ are those detected by Poincar\'{e} duality.
\end{lem}

\begin{pf}
	By definition, $\mathcal{G}^{3}_{p}\left(S^{5}\right)$ is obtained from a surgery on $S^{2}\times S^{5}$ that is stably parallelizable. Hence by \cite[Lemma 5.4]{KervaireMilnor63} the resulting manifold $\mathcal{G}^{3}_{p}\left(S^{5}\right)$ is also stably parallelizable. As a consequence, its characteristic classes vanish. In particular $\mathcal{G}^{3}_{p}\left(S^{5}\right)$ is orientable and spinnable.
	
	Then we apply van Kampen theorem and Mayer-Vietoris sequence to decompositions of $S^{2}\times S^{5}$ and $\mathcal{G}^{3}_{p}\left(S^{5}\right)$ given in Figure \ref{Figure: The gyration on S^5}. It is routine to show that $\mathcal{G}^{3}_{p}\left(S^{5}\right)$ is simply connected and to compute the homology, cohomology groups as stated above.
	
	It remains to determine the ring structure of $H^{*}\left(\mathcal{G}^{3}_{p}\left(S^{5}\right);\mathbb{Z}\big/p\right)$. Let $x$ be a generator of $H^{2}\left(\mathcal{G}^{3}_{p}\left(S^{5}\right);\mathbb{Z}\big/p\right)\cong\mathbb{Z}\big/p$. Consider the long exact sequence of cohomology groups of $\mathcal{G}^{3}_{p}\left(S^{5}\right)$ associated to the short exact sequence $\mathbb{Z}\rightarrowtail\mathbb{Z}\twoheadrightarrow\mathbb{Z}\big/p$ of coefficient rings. Here and henceforth a tailed arrow means a monomorphism and a two-head arrow is an epimorphism. The boundary homomorphism $H^{2}\left(\mathcal{G}^{3}_{p}\left(S^{5}\right);\mathbb{Z}\big/p\right)\xrightarrow{\delta}H^{3}\left(\mathcal{G}^{3}_{p}\left(S^{5}\right)\right)$ and the mod $p$ homomorphism $H^{3}\left(\mathcal{G}^{3}_{p}\left(S^{5}\right)\right)\xrightarrow{\rho_{p}}H^{3}\left(\mathcal{G}^{3}_{p}\left(S^{5}\right);\mathbb{Z}\big/p\right)$ are isomorphisms. Since the mod $p$ Bockstein homomorphism $\beta$ is the composition $\beta=\rho_{p}\circ\delta$, $H^{2}\left(\mathcal{G}^{3}_{p}\left(S^{5}\right);\mathbb{Z}\big/p\right)\xrightarrow{\beta}H^{3}\left(\mathcal{G}^{3}_{p}\left(S^{5}\right);\mathbb{Z}\big/p\right)$ is an isomorphism, and $\beta x$ is a generator of $H^{3}\left(\mathcal{G}^{3}_{p}\left(S^{5}\right);\mathbb{Z}\big/p\right)\cong\mathbb{Z}\big/p$.
	
	We claim that $x^{2}\beta x=0$. Note that the pair $\left(\mathcal{G}^{3}_{p}\left(S^{5}\right),x\right)$ determines an element of $\Omega_{7}^{Spin}\left(K\left(\mathbb{Z}\big/p,2\right)\right)$, and the assignment $\left(M,z\right)\mapsto\left<z^{2}\beta z,[M]_{p}\right>$ induces a homomorphism $\Omega_{7}^{Spin}\left(K\left(\mathbb{Z}\big/p,2\right)\right)\to\mathbb{Z}\big/p$. Here $[M]_{p}$ denotes the mod $p$ fundamental class (see also Lemma \ref{Lemma: char nums additive wrt gluing via boundary}). It follows from Propositions \ref{Proposition: Omega_q^Spin(Z/p,2), p>=5} and \ref{Proposition: Omega_q^Spin(Z/3,2)} that $\Omega_{7}^{Spin}\left(K\left(\mathbb{Z}\big/p,2\right)\right)=0$, hence $\left<z^{2}\beta z,[M]_{p}\right>=0$ for any closed spin $7$-manifold $M$ and any class $z\in H^{2}\left(M;\mathbb{Z}\big/p\right)$. As a consequence, $x^{2}\beta x=0$, any non-zero multiple of $x\beta x$ is not the Poincar\'e dual of $x$, which forces $x\beta x=0$. Likewise we have $x^{2}=0$. Therefore, the only non-trivial cup products in $H^{*}\left(\mathcal{G}^{3}_{p}\left(S^{5}\right);\mathbb{Z}\big/p\right)$ are those detected by Poincar\'{e} duality.
\end{pf}

\begin{lem}\label{Lemma: equiv char of G_3^p S^5-like mfd}
	Let $M$ be a $\mathcal{G}^{3}_{p}\left(S^{5}\right)$-like manifold. Then $p_{1}(M)=0$, $w_{2}(M)=0$ and the only non-trivial cup products in $H^{*}\left(M;\mathbb{Z}\big/p\right)$ are those detected by Poincar\'{e} duality.
\end{lem}
Hence a $\mathcal{G}^{3}_{p}\left(S^{5}\right)$-like manifold also admits the same mod $p$ cohomology ring and characteristic classes as those of $\mathcal{G}^{3}_{p}\left(S^{5}\right)$. 

\begin{pf}
	According to the homology and cohomology groups of the $\mathcal{G}^{3}_{p}\left(S^{5}\right)$-like manifold $M$, we have $H^{4}(M)=0$ and $M$ is a $\mathbb{Z}\big/2$-homology $7$-sphere, hence $p_{1}(M)=0$ and $M$ has vanishing Stiefel-Whitney classes. Following exactly the same argument of proof of Lemma \ref{Lemma: coh ring & char class of G_3^p S^5}, we can show that the only non-trivial cup products in $H^{*}\left(M;\mathbb{Z}\big/p\right)$ are those detected by Poincar\'{e} duality
\end{pf}

\begin{lem}\label{Lemma: normal 2-type of G_3^p S^5-like mfds}
	Let $M$ be a $\mathcal{G}^{3}_{p}\left(S^{5}\right)$-like manifold. The normal $2$-type and normal $2$-smoothing of $M$ are given as in Figure \ref{Figure: normal 2-type & normal 2-str of G_3^p S^5-like mfd}. 
	\begin{figure}[H]
		\begin{center}
			\[
			\xymatrix{
				& BSpin\times K\left(\mathbb{Z}\big/p,2\right)\ar[d]^{\mathrm{pr}_{1}}\\
				& BSpin\ar[d]^{B\phi}\\
				M\ar[ruu]^{{\varphi}=\left(\widetilde{\nu},\widehat{\varphi}\right)}\ar[ru]^{\widetilde{\nu}}\ar[r]^{\nu} & BO
			}
			\]
		\end{center}
		\caption{The normal $2$-type and normal $2$-smoothing of $\mathcal{G}^{3}_{p}\left(S^{5}\right)$-like manifold $M$}
		\label{Figure: normal 2-type & normal 2-str of G_3^p S^5-like mfd}
	\end{figure}
	Here $Spin\xrightarrow{\phi}O$ is the canonical map, $BSpin\xrightarrow{B\phi}BO$ is the induced map, $M\xrightarrow{\widetilde{\nu}}BSpin$ assigns the orientation and unique spin structure of $M$, and $M\xrightarrow{\widehat{\varphi}}K\left(\mathbb{Z}\big/p,2\right)$ induces an isomorphism of second homotopy groups.
\end{lem}

\begin{pf}
	Since $M$ is a $\mathcal{G}^{3}_{p}\left(S^{5}\right)$-like manifold, $M$ is simply-connected and by Hurewicz's theorem $\pi_{2}(M)\cong H_{2}(M)\cong\mathbb{Z}\big/p$. Since $BSpin$ is $3$-connected and $K\left(\mathbb{Z}\big/p,2\right)$ is the Eilenberg-MacLane space, it is straightforward to verify that $\left(BSpin\times K\left(\mathbb{Z}\big/p,2\right),B\phi\circ\mathrm{pr}_{1}\right)$ is $2$-universal and ${\varphi}$ is a normal $2$-smoothing.
\end{pf}

\begin{rmk}
	Equivalently, a normal $2$-smoothing of a $\mathcal{G}^{3}_{p}\left(S^{5}\right)$-like manifold assigns the orientation, the unique spin structure and a non-zero cohomology class $x=x_{{\varphi}}\in H^{2}\left(M;\mathbb{Z}\big/p\right)$.
	Unlike $\mathcal{G}^{3}(\mathrm{Wu})$-like manifolds, a $\mathcal{G}^{3}_{p}\left(S^{5}\right)$-like manifold admits more than one normal $2$-smoothings, which is in bijection with non-zero classes of $H^{2}\left(M;\mathbb{Z}\big/p\right)$.
\end{rmk}

	\section{Identify the surgery obstruction}\label{Section: Identify the surgery obstructions}
Let $M_{0}$, $M_{1}$ be a pair of $\mathcal{G}^{3}(\mathrm{Wu})$-like manifolds (resp. two $\mathcal{G}^{3}_{p}\left(S^{5}\right)$-like manifolds). Then they have the same normal $2$-type $(\mathcal{B},\mathcal{F})=(BSO,B\phi)$ (resp. $\left(BSpin\times K\left(\mathbb{Z}\big/p,2\right),B\phi\circ\mathrm{pr}_{1}\right)$), and a $(\mathcal{B},\mathcal{F})$-structure on an orientable manifold is simply an assignment of orientation (resp. an assignment of orientation, the unique spin structure and a nonzero cohomology class). Since $\Omega_{7}(\mathcal{B},\mathcal{F})=\Omega_{7}^{SO}=0$ (resp. $\Omega_{7}^{Spin}\left(K\left(\mathbb{Z}\big/p,2\right)\right)=0$, Propositions \ref{Proposition: Omega_q^Spin(Z/p,2), p>=5}, \ref{Proposition: Omega_q^Spin(Z/3,2)}), any two $7$-dimensional $(\mathcal{B},\mathcal{F})$-manifolds are $(\mathcal{B},\mathcal{F})$-bordant, hence there is a $(\mathcal{B},\mathcal{F})$-bordism $(W,\psi)$ between $\left(M_{i},\varphi_{i}\right)$. Here in the case of $\mathcal{G}^{3}(\mathrm{Wu})$-like manifolds, reference maps ${\psi}$ and ${\varphi_{i}}$ assign orientations of $W$ and $M_{i}$ respectively. In the case of $\mathcal{G}^{3}_{p}\left(S^{5}\right)$-like manifolds, we denote ${\psi}=\left(\widetilde{\nu_{W}},\widehat{\psi}\right)$, ${\varphi_{i}}=\left(\widetilde{\nu_{i}},\widehat{\varphi_{i}}\right)$, such that $\widetilde{\nu_{W}}$, $\widetilde{\nu_{i}}$ assign orientations and spin structures of $W$ and $M_{i}$, $\widehat{\psi}$ assign a class $x_{{\psi}}\in H^{2}\left(W;\mathbb{Z}\big/p\right)$ and $\widehat{\varphi_{i}}$ assign a non-zero class $x_{{\varphi_{i}}}\in H^{2}\left(M_{i};\mathbb{Z}\big/p\right)$.

According to modified surgery theory, $M_{0}$, $M_{1}$ are diffeomorphic if and only if there is a bordism $(W,\psi)$ between $\left(M_{i},\varphi_{i}\right)$ such that the surgery obstruction $\theta(W,\psi)\in l_{8}(\{e\})$ is elementary (\cite[Theorem 3]{SurgeryAndDuality}). By \cite[Proposition 4]{SurgeryAndDuality} we may assume $W\xrightarrow{\psi}\mathcal{B}$ is a $4$-equivalence, and it follows from \cite[Section 5]{SurgeryAndDuality} that in this case $\theta(W,\psi)$ is represented by the following data
\[\left(H_{4}\left(W,M_{0}\right)\xleftarrow{f_{0}}\pi^{\mathcal{B}}_{4}(W)\xrightarrow{f_{1}}H_{4}\left(W,M_{1}\right),\beta_{W}\right),\]
where 
\begin{compactenum}
	\item $\pi_{4}^{\mathcal{B}}(W)=\ker\left(\pi_{4}(\psi)\right)$;
	\item $f_{i}$ is the composition 
	$$\pi_{4}^{\mathcal{B}}(W)\to H_{4}^{\mathcal{B}}(W)\to H_{4}(W)\to H_{4}\left(W,M_{i}\right),$$
	in which $H_{4}^{\mathcal{B}}(W)=\ker\left(H_{4}(\psi)\right)$, $\pi_{4}^{\mathcal{B}}(W)\to H_{4}^{\mathcal{B}}(W)$ is the map induced from Hurewicz homomorphism, $H_{4}^{\mathcal{B}}(W)\to H_{4}(W)$ is the subgroup inclusion and $H_{4}(W)\to H_{4}\left(W,M_{i}\right)$ is the map induced from the inclusion $W=(W,\varnothing)\to\left(W,M_{i}\right)$;
	\item $\beta_{W}$ is the homological intersection pairing on $H_{4}\left(W,M_{0}\right)\times H_{4}\left(W,M_{1}\right)$;
	\item two representatives represent the same element if they become isomorphic after a direct sum with certain copies of hyperbolic forms.
\end{compactenum}
 
In this section we identify $\theta(W,\psi)$ and determine the isomorphism type of necessary groups, leaving the analysis of whether this obstruction is elementary to the next section.
We begin with some notations. If $A$ is an abelian group, let $T(A)$ denote the torsion subgroup of $A$ and let $F(A)=A\big/T(A)$ denote the associated torsion-free group. Notation $0$ may indicate a zero group, a zero homomorphism or a zero bilinear form.

We state the results of $\mathcal{G}^{3}(\mathrm{Wu})$-like manifolds and $\mathcal{G}^{3}_{p}\left(S^{5}\right)$-like manifolds separately. In the proof of conclusions concerning $\mathcal{G}^{3}_{p}\left(S^{5}\right)$-like manifolds, we require the results about certain (generalized) homology groups of $K\left(\mathbb{Z}\big/p,2\right)$, which can be found in Section \ref{Section: Compute bordism gps}.

\begin{prop}\label{Proposition: identify surgery obstruction}
	In the case of $\mathcal{G}^{3}(\mathrm{Wu})$-like manifolds,
	\begin{compactenum}
		\item $H_{4}(W)\cong\mathbb{Z}^{r}\oplus\mathbb{Z}\big/2$ $(r\geqslant2)$. $H_{4}(W)\xrightarrow{H_{4}(\psi)}H_{4}(\mathcal{B})$ is epic and $\ker\left(H_{4}(\psi)\right)\cong\mathbb{Z}^{r-1}\oplus\mathbb{Z}\big/2$. The induced homomorphism $F\left(\ker\left(H_{4}(\psi)\right)\right)\xrightarrow{\iota}F\left(H_{4}(W)\right)$ is monic and has cokernel $H_{4}(\mathcal{B})\cong\mathbb{Z}\oplus\mathbb{Z}\big/2$.
		\item Under Poincar\'{e}-Lefschetz duality, the intersection pairing $H^{4}\left(W,M_{1}\right)\times H^{4}\left(W,M_{0}\right)\to\mathbb{Z}$ induces a symmetric unimodular bilinear form $\beta$ on $F\left(H_{4}(W)\right)$.
		\item $\theta(W,\psi)$ can be represented by
		\begin{eqnarray}\label{Equation: representative of theta(W,psi)}
			\left(F\left(H_{4}(W)\right)\xlongleftarrow{\iota}F\left(\ker\left(H_{4}(\psi)\right)\right)\xlongrightarrow{\iota}F\left(H_{4}(W)\right),\beta\right)\oplus\left(0\xlongleftarrow{0}\left(\mathbb{Z}/2\right)^{2}\xlongrightarrow{0}0,0\right).
		\end{eqnarray}
	\end{compactenum}
\end{prop}

\begin{prop}\label{Proposition: identify surgery obstruction, G_3^p S^5}
	In the case of $\mathcal{G}^{3}_{p}\left(S^{5}\right)$-like manifolds,
	\begin{compactenum}
		\item $H_{4}(W)\cong\mathbb{Z}^{r}\oplus\mathbb{Z}\big/p$ $(r\geqslant2)$. $H_{4}(W)\xrightarrow{H_{4}(\psi)}H_{4}(\mathcal{B})$ is epic and $\ker\left(H_{4}(\psi)\right)\cong\mathbb{Z}^{r-1}\oplus\mathbb{Z}\big/p$. The induced homomorphism $F\left(\ker\left(H_{4}(\psi)\right)\right)\xrightarrow{\iota}F\left(H_{4}(W)\right)$ is monic and has cokernel $H_{4}(\mathcal{B})\cong\mathbb{Z}\oplus\mathbb{Z}\big/p$.
		\item Under Poincar\'{e}-Lefschetz duality, the intersection pairing $H^{4}\left(W,M_{1}\right)\times H^{4}\left(W,M_{0}\right)\to\mathbb{Z}$ induces a symmetric unimodular bilinear form $\beta$ on $F\left(H_{4}(W)\right)$.
		\item $\theta(W,\psi)$ can be represented by
		\begin{eqnarray}\label{Equation: representative of theta(W,psi), G_3^p S^5}
			\left(F\left(H_{4}(W)\right)\xlongleftarrow{\iota}F\left(\ker\left(H_{4}(\psi)\right)\right)\xlongrightarrow{\iota}F\left(H_{4}(W)\right),\beta\right)\oplus\left(0\xlongleftarrow{0}\mathbb{Z}/2p\xlongrightarrow{0}0,0\right).
		\end{eqnarray}
	\end{compactenum}
\end{prop}

\begin{pf}[of Proposition \ref{Proposition: identify surgery obstruction}]
	Consider the long exact sequences of homotopy and homology groups associated to the pair $\left(\mathcal{B}, W\right)$. They are related by the Hurewicz homomorphisms and we obtain the commutative diagram shown as in Figure \ref{Figure: LEL of Hurewicz homom ass to (B,W)}.
\begin{figure}[H]
	\begin{center}
		$$
			\xymatrix@R-2pt{
			\pi_{5}(\mathcal{B})\ar[r]\ar[d] & \pi_{5}(\mathcal{B},W)\ar[r]\ar[d] & \pi_{4}(W)\ar[r]^{\pi_{4}(\psi)}\ar[d] & \pi_{4}(\mathcal{B}) \ar[r]\ar[d] & \pi_{4}(\mathcal{B},W)\ar[d]\\
			H_{5}(\mathcal{B})\ar[r] & H_{5}(\mathcal{B},W)\ar[r] & H_{4}(W)\ar[r]^{H_{4}(\psi)} & H_{4}(\mathcal{B})\ar[r] & H_{4}(\mathcal{B},W)
			}
		$$
	\end{center}
	\caption{The long exact ladder of Hurewicz homomorphisms associated to $\left(\mathcal{B}, W\right)$}
	\label{Figure: LEL of Hurewicz homom ass to (B,W)}
\end{figure}
Since $\psi$ is an $4$-equivalence and $\mathcal{B}$ is simply connected, $\pi_{4}(\mathcal{B},W)=H_{4}(\mathcal{B},W)=0$ and the Hurewicz homomorphism $\pi_{5}(\mathcal{B},W)\to H_{5}(\mathcal{B},W)$ is an isomorphism. Moreover, $\pi_{5}(\mathcal{B})=\pi_{5}(BSO)=\pi_{4}(SO)=\pi_{4}(Spin)=0$, hence $\pi_{4}^{\mathcal{B}}(W)\cong\pi_{5}(\mathcal{B},W)\cong H_{5}(\mathcal{B},W)$ (see also Figure \ref{Figure: LEL of Hurewicz homom ass to (B,W), groups identified}).
\begin{figure}[h]
	\begin{center}
		\[
			\xymatrix@R-2pt{
				0\ar[r]\ar[d] & \pi_{5}(\mathcal{B},W)\ar[r]\ar[d]^{\cong} & \pi_{4}(W)\ar[r]^(0.57){\pi_{4}(\psi)}\ar[d] & \mathbb{Z} \ar[r]\ar[d] & 0\ar[d]\\
				\mathbb{Z}\big/2\ar[r] & H_{5}(\mathcal{B},W)\ar[r] & H_{4}(W)\ar[r]^(0.48){H_{4}(\psi)} & \mathbb{Z}\oplus\mathbb{Z}\big/2\ar[r] & 0
			}
		\]
	\end{center}
	\caption{The long exact ladder of Hurewicz homomorphisms associated to $\left(\mathcal{B}, W\right)$ with groups identified, $\mathcal{B}=BSO$}
	\label{Figure: LEL of Hurewicz homom ass to (B,W), groups identified}
\end{figure} 

To identify $\theta(W,\psi)$ it remains to determine the homomorphisms $\pi_{4}^{\mathcal{B}}(W)\xrightarrow{f_{i}}H_{4}\left(W,M_{i}\right)$. Under these identifications above, $\pi_{4}^{\mathcal{B}}(W)\xrightarrow{f_{i}}H_{4}\left(W,M_{i}\right)$ is equivalent to the map $H_{5}(\mathcal{B},W)\xrightarrow{\partial_{i}}H_{4}\left(W,M_{i}\right)$, where $\partial_{i}$ is the connecting homomorphism in the long exact sequence of relative homology groups associated to the tuple $\left(\mathcal{B},W,M_{i}\right)$.

Therefore, $\theta(W,\psi)$ is represented by the following diagram
$$\left(H_{4}\left(W,M_{0}\right)\xleftarrow{\partial_{0}}H_{5}(\mathcal{B},W)\xrightarrow{\partial_{1}}H_{4}\left(W,M_{1}\right),{\beta}_{W}\right),$$
where $\left.{\beta}_{W}\right|_{H_{5}(\mathcal{B},W)}$ is even by \cite[Proposition 6]{SurgeryAndDuality}.

To further study $\partial_{i}$, we consider the long exact braid of (relative) homology groups associated to the triple $\left(\mathcal{B},W,M_{i}\right)$ (Figure \ref{Figure: long exact braid ass to B,W,M_i}). 
\begin{figure}[H]
	\begin{center}
		$$
		\xymatrix@C=-1.2ex{
			&
			& H_{5}\left(W,M_{i}\right) \ar@/^25pt/[rr] \ar[rd]
			&
			& \ H_{4}\left(M_{i}\right)\  \ar@/^25pt/[rr]_{H_{4}\left(\varphi_{i}\right)} \ar[rd]
			&
			& \ H_{4}(\mathcal{B})\  \ar@/^25pt/[rr] \ar[rd]
			&
			& H_{4}(\mathcal{B},W) \ar[rd]
			&
			\\
			&\ \ H_{5}(W)\ \ \ar[ru] \ar[rd]
			&
			& H_{5}\left(\mathcal{B},M_{i}\right) \ar[ru] \ar[rd]
			&
			& \ \ \ H_{4}(W)\ \ \ \ar[ru]^{H_{4}(\psi)} \ar[rd]
			&
			& H_{4}\left(\mathcal{B},M_{i}\right) \ar[ru] \ar[rd]
			&
			& \ \ \ H_{3}(W)\ \ \ 
			\\
			\ H_{5}\left(M_{i}\right)\  \ar[ru] \ar@/_25pt/[rr]
			&
			& H_{5}(\mathcal{B}) \ar[ru] \ar@/_25pt/[rr]
			&
			& H_{5}(\mathcal{B}, W) \ar[ru] \ar@/_25pt/[rr]^{\partial_{i}}
			&
			& \hskip-4pt H_{4}\left(W,M_{i}\right) \hskip-4pt \ar[ru] \ar@/_25pt/[rr]
			&
			& \ H_{3}\left(M_{i}\right)\  \ar[ru]
			&
		}
		$$
	\end{center}
	\caption{The long exact braid of (relative) homology groups associated to $\left(\mathcal{B}, W, M_{i}\right)$}
	\label{Figure: long exact braid ass to B,W,M_i}
\end{figure}
Some groups in the diagram are identified as follows:
\begin{compactenum}
	\item $H_{5}(W)=0$. By Poincar\'{e}-Lefschetz duality and universal coefficient theorem $H_{5}(W)\cong H^{3}(W,\partial W)\cong H_{3}(W,\partial W)\spcheck\oplus\mathrm{Ext}\left(H_{2}(W,\partial W),\mathbb{Z}\right)$, and by the long exact sequence of relative homology groups associated to the triple $(W,\partial W,M_{i})$ we have $H_{3}(W,\partial W)\cong H_{2}(\partial W,M_{i})\cong H_{2}\left(M_{1-i}\right)\cong\mathbb{Z}\big/2$ and $H_{2}(W,\partial W)=0$.
	\item $H_{3}(W)=0$. Since $\psi$ is a $4$-equivalence and $H_{3}(\mathcal{B})=0$.
	\item $H_{5}\left(M_{i}\right)=H_{3}\left(M_{i}\right)=0$ and $H_{4}\left(M_{i}\right)\cong\mathbb{Z}\big/2$. This follows from Lemma \ref{Lemma: coh ring & char class of G_3(Wu)}.
	\item $H_{5}(\mathcal{B})\cong\mathbb{Z}\big/2$ and $H_{4}(\mathcal{B})\cong\mathbb{Z}\oplus\mathbb{Z}\big/2$. This follows from Lemma \ref{Lemma: integral coh of BSO}.
	\item $H_{5}\left(W,M_{i}\right)\cong H^{3}\left(W,M_{1-i}\right)=0$ and $H_{4}\left(W,M_{i}\right)\cong\mathbb{Z}^{r}$. Recall that $W\xrightarrow{\psi}\mathcal{B}$ is a $4$-equivalence and $M_{i}\xrightarrow{\varphi_{i}}\mathcal{B}$ are $3$-equivalences, hence the boundary inclusions $M_{i}\xrightarrow{\iota_{i}}W$ are also $3$-equivalence. Now the claimed isomorphisms follow from Hurewicz theorem and Poincar\'{e}-Lefschetz duality.
	\item $H_{5}(W,\partial W)\cong H^{3}(W)\cong H^{3}(\mathcal{B})\cong\mathbb{Z}\big/2$. The first isomorphism follows from Poincar\'{e}-Lefschetz duality and the second one is because $\psi$ is a $4$-equivalence.
\end{compactenum}
With these groups identified, we obtain Figure \ref{Figure: long exact braid ass to B,W,M_i, groups identified}. We also identify certain homomorphisms in the partial long exact braid, in the sense that whether they are trivial, monic, epic or isomorphic. 
\begin{figure}[H]
	\begin{center}
		$$
		\xymatrix@C=-1.2ex{
			&
			& \ \ \ \ \ \ 0\ \ \ \ \ \ \ar@/^25pt/[rr] \ar[rd]
			&
			& \ \mathbb{Z}\big/2\  \ar@/^25pt/[rr]_{H_{4}\left(\varphi_{i}\right)} \ar@{>->}[rd]
			&
			& \ \mathbb{Z}\oplus\mathbb{Z}\big/2\  \ar@/^25pt/[rr] \ar@{->>}[rd]
			&
			& \ \ \ \ \ \ 0\ \ \ \ \ \  \ar[rd]
			&
			\\
			&\ \ \ \ \ \ \ 0\ \ \ \ \ \ \ \ar[ru] \ar[rd]
			&
			& \ H_{5}\left(\mathcal{B},M_{i}\right)\ \ar[ru] \ar@{>->}[rd]
			&
			& \ \ \ \ H_{4}(W)\ \ \ \ \ar@{->>}[ru]^{H_{4}(\psi)} \ar@{->>}[rd]
			&
			& \ H_{4}\left(\mathcal{B},M_{i}\right) \ \ar[ru] \ar[rd]
			&
			& \ \ \ \ \ \ \ 0\ \ \ \ \ \ \ 
			\\
			\ \ \ \ \ 0\ \ \ \ \ \ar[ru] \ar@/_25pt/[rr]
			&
			& \ \ \mathbb{Z}\big/2\ \ \ar@{>->}[ru] \ar@{>->}@/_25pt/[rr]
			&
			& H_{5}(\mathcal{B}, W) \ar[ru] \ar@/_25pt/[rr]^{\partial_{i}}
			&
			& \ \ \mathbb{Z}^{r}\ \  \ar@{->>}[ru] \ar@/_25pt/[rr]
			&
			& \ \ \ \ \ 0\ \ \ \ \  \ar[ru] 
			&
		}
		$$
	\end{center}
	\caption{The long exact braid of (relative) homology groups associated to $\left(\mathcal{B}, W, M_{i}\right)$ with groups and homomorphisms identified, $\mathcal{B}=BSO$}
	\label{Figure: long exact braid ass to B,W,M_i, groups identified}
\end{figure}
It can also be read from Figure \ref{Figure: long exact braid ass to B,W,M_i, groups identified} that $H_{4}(W)\cong\mathbb{Z}^{r}\oplus\mathbb{Z}\big/2$, and it remains to determine the homomorphism $H_{4}\left(M_{i}\right)\xrightarrow{H_{4}(\varphi_{i})}H_{4}(\mathcal{B})$ for further identification of remaining groups.

Next we consider the long exact sequence of homology groups associated to the pair $\left(\mathcal{B},M_{i}\right)$ (Figure \ref{Figure: LES ass to (B,M_i)}) and determine $H_{q}\left(\mathcal{B},M_{i}\right)$, $q=4$, $5$. For convenience we denote by $h^{q}$ and $h_{q}$ the mod $2$ cohomology and homology until the end of this proof. Vertical arrows are homomorphisms of mod $2$ reduction.
\begin{figure}[h]
	\begin{center}
		$$
		\xymatrix{
			H^{3}\left(\mathcal{B},M_{i}\right) \ar[r]
			& H^{3}\left(\mathcal{B}\right) \ar[r]^(0.47){\varphi_{i}^{*}}
			& H^{3}\left(M_{i}\right) \ar[r]
			& H^{4}\left(\mathcal{B},M_{i}\right) \ar[r]
			& H^{4}\left(\mathcal{B}\right) \ar[r]
			& H^{4}\left(M_{i}\right)
		}
		$$
		$$
		\xymatrix@R-2pt{
			H^{4}(\mathcal{B}) \ar[r] \ar[d]
			& H^{4}\left(M_{i}\right) \ar[r] \ar[d]
			& H^{5}\left(\mathcal{B},M_{i}\right) \ar[r] \ar[d]
			& H^{5}(\mathcal{B}) \ar[r] \ar[d]
			& H^{5}\left(M_{i}\right) \ar[r] \ar[d]
			& H^{6}\left(\mathcal{B},M_{i}\right) \ar[r] \ar[d]
			& H^{6}(\mathcal{B}) \ar[d]\\
			h^{4}(\mathcal{B}) \ar[r]^(0.47){\varphi_{i}^{*}}
			& h^{4}\left(M_{i}\right) \ar[r]
			& h^{5}\left(\mathcal{B},M_{i}\right) \ar[r]
			& h^{5}(\mathcal{B}) \ar[r]^(0.47){\varphi_{i}^{*}}
			& h^{5}\left(M_{i}\right) \ar[r]
			& h^{6}\left(\mathcal{B},M_{i}\right) \ar[r]
			& h^{6}(\mathcal{B})
		}
		$$
	\end{center}
	\caption{The long exact sequences of cohomology groups associated to $\left(\mathcal{B},M_{i}\right)$}
	\label{Figure: LES ass to (B,M_i)}
\end{figure}

The cohomology groups of $\mathcal{B}$ and $M_{i}$ are known, and we shall determine certain homomorphisms.
\begin{compactenum}
	\item $h^{q}(\mathcal{B})\xrightarrow{\varphi_{i}^{*}}h^{q}\left(M_{i}\right)$ are trivial homomorphisms for $q=4$, $5$. This is because $\varphi_{i}^{*}$ assigns universal Stiefel-Whitney classes to the exact characteristic classes of $M_{i}$, and we have seen before that polynomials degree $4$ and $5$ in Stiefel-Whitney classes for $M_{i}$ vanish.
	\item $H^{5}\left(M_{i}\right)\xrightarrow{\rho_{2}}h^{5}\left(M_{i}\right)$ is an isomorphism. This follows from the long exact sequence of cohomology groups of $M_{i}$ induced from the short exact sequence $\mathbb{Z}\rightarrowtail\mathbb{Z}\twoheadrightarrow\mathbb{Z}\big/2$ of coefficient rings.
\end{compactenum}
\begin{figure}[h]
	\begin{center}
		\[
			\xymatrix{
				0 \ar[r]
				& \mathbb{Z}\big/2 \ar[r]^{\cong}
				& \mathbb{Z}\big/2 \ar[r]
				& H^{4}\left(\mathcal{B},M_{i}\right) \ar[r]^(0.66){\cong}
				& \mathbb{Z} \ar[r]
				& 0
			}
		\]
		\[
			\xymatrix@R-2pt{
				\mathbb{Z} \ar[r] \ar[d]
				& 0 \ar[r] \ar[d]
				& H^{5}\left(\mathcal{B},M_{i}\right) \ar[r] \ar[d]
				& \mathbb{Z}\big/2 \ar[r] \ar@{>->}[d]
				& \mathbb{Z}\big/2 \ar[r] \ar[d]^{\cong}
				& H^{6}\left(\mathcal{B},M_{i}\right) \ar@{->>}[r] \ar[d]
				& \mathbb{Z}\big/2 \ar@{>->}[d]\\
				\left(\mathbb{Z}\big/2\right)^{2} \ar[r]^(0.55){0}
				& \mathbb{Z}\big/2 \ar[r]
				& h^{5}\left(\mathcal{B},M_{i}\right) \ar[r]
				& \left(\mathbb{Z}\big/2\right)^{2} \ar[r]^(0.55){0}
				& \mathbb{Z}\big/2 \ar[r]
				& h^{6}\left(\mathcal{B},M_{i}\right) \ar@{->>}[r]
				& \left(\mathbb{Z}\big/2\right)^{4}
			}
		\]
	\end{center}
	\caption{The long exact sequences associated to $\left(\mathcal{B},M_{i}\right)$ with groups and homomorphisms identified}
	\label{Figure: LES ass to (B,M_i), groups identified}
\end{figure}

Therefore, we obtain Figure \ref{Figure: LES ass to (B,M_i), groups identified}, from which we further deduce that
$$H^{4}\left(\mathcal{B},M_{i}\right)\cong\mathbb{Z},\ H^{5}\left(\mathcal{B},M_{i}\right)\cong\mathbb{Z}\big/2,\ H^{6}\left(\mathcal{B},M_{i}\right)\cong\left(\mathbb{Z}\big/2\right)^{2}.$$ 
Hence by universal coefficient theorem we have $$H_{4}\left(\mathcal{B},M_{i}\right)\cong\mathbb{Z}\oplus\mathbb{Z}\big/2,\ H_{5}\left(\mathcal{B},M_{i}\right)\cong\left(\mathbb{Z}\big/2\right)^{2}.$$
Accordingly, we determine most homomorphisms in the partial long exact braid, knowing whether they are trivial, monic, epic or isomorphic.
In particular, since $H_{4}\left(W,M_{i}\right)\cong\mathbb{Z}^{r}$ surjects onto $H_{4}\left(\mathcal{B},M_{i}\right)\cong\mathbb{Z}\oplus\mathbb{Z}\big/2$, we must have $r\geqslant2$.

\begin{compactenum}
	\item We have a short exact sequence $H_{5}(\mathcal{B})\rightarrowtail H_{5}\left(\mathcal{B},M_{i}\right)\twoheadrightarrow H_{4}\left(M_{i}\right)$. The exactness at $H_{5}\left(\mathcal{B},M_{i}\right)$ is clear. The injectivity of $H_{5}(\mathcal{B})\to H_{5}\left(\mathcal{B},M_{i}\right)$ follows from that of the composition $H_{5}(\mathcal{B})\to H_{5}\left(\mathcal{B},M_{i}\right)\to H_{5}(\mathcal{B},W)$. Return to the long exact braid, and we see the cokernel of inclusion $H_{5}(\mathcal{B})\rightarrowtail H_{5}\left(\mathcal{B},M_{i}\right)$ injects into $H_{4}\left(M_{i}\right)$. Since both $H_{4}\left(M_{i}\right)$ and the cokernel are isomorphic to $\mathbb{Z}\big/2$, they are isomorphic and we obtain the claimed short exact sequence.
	\item $H_{4}\left(M_{i}\right)\xrightarrow{H_{4}\left(\varphi_{i}\right)}H_{4}(\mathcal{B})$ is zero and $H_{4}(\mathcal{B})\to H_{4}\left(\mathcal{B},M_{i}\right)$ is an isomorphism. This is a direct corollary of the last statement and the exact braid.
	\item $H_{4}(W)\cong\mathbb{Z}^{r}\oplus\mathbb{Z}\big/2$ and $\ker\left(H_{4}(\psi)\right)\cong\mathbb{Z}^{r-1}\oplus\mathbb{Z}\big/2$. We have an short exact sequence $\mathbb{Z}\big/2\rightarrowtail H_{4}(W)\twoheadrightarrow\mathbb{Z}^{r}$, which splits as $\mathbb{Z}^{r}$ is free. Since $H_{4}\left(\varphi_{i}\right)=0$, $T\left(H_{4}(W)\right)\cong\mathbb{Z}\big/2$ is contained in $\ker\left(H_{4}(\psi)\right)$ and $H_{4}(\psi)$ induces an epimorphism $F\left(H_{4}(W)\right)\xrightarrow{\overline{H_{4}(\psi)}}\mathbb{Z}^{r}\twoheadrightarrow\mathbb{Z}\oplus\mathbb{Z}\big/2$. Hence $\ker\left(\overline{H_{4}(\psi)}\right)\cong\mathbb{Z}^{r-1}$ and $\ker\left({H_{4}(\psi)}\right)\cong\mathbb{Z}^{r-1}\oplus\mathbb{Z}\big/2$.
	\item $H_{5}(\mathcal{B},W)\cong\mathbb{Z}^{r-1}\oplus\left(\mathbb{Z}\big/2\right)^{2}$. From the long exact braid we obtain a short exact sequence
	$$\left(\mathbb{Z}\big/2\right)^{2}\rightarrowtail H_{5}(\mathcal{B},W)\twoheadrightarrow\ker\left(\mathbb{Z}^{r}\twoheadrightarrow\mathbb{Z}\oplus\mathbb{Z}\big/2\right),$$
	where the ending term is isomorphic to $\mathbb{Z}^{r-1}$.
	\item $F\left(H_{5}(\mathcal{B},W)\right)=F\left(\ker\left(H_{4}(\psi)\right)\right)$.
\end{compactenum}
Therefore, $H_{5}(\mathcal{B},W)\xrightarrow{\partial_{i}}H_{4}\left(W,M_{i}\right)$ is identified with $F\left(\ker\left(H_{4}(\psi)\right)\right)\oplus\left(\mathbb{Z}\big/2\right)^{2}\xrightarrow{(\iota,0)}F\left(H_{4}(W)\right)$ and $\mathrm{coker}\ \iota\cong H_{4}(\mathcal{B})\cong\mathbb{Z}\oplus\mathbb{Z}\big/2$.

Finally, we have the following commutative diagram of cohomological and homological intersection forms (Figure \ref{Figure: The intersection pairings}), where $W\xrightarrow{i_{\varepsilon}}\left(W,M_{\varepsilon}\right)\xrightarrow{j_{\varepsilon}}\left(W,\partial W\right)$ denote the inclusions and $H^{q}(X;R)\times H_{q}(X;R)\xrightarrow{\left<\cdot,\cdot\right>}R$ denotes the Kronecker pairing between cohomology and homology groups. Figure \ref{Figure: The intersection pairings} justifies that the symmetric bilinear form $\beta$ on $F\left(H_{4}(W)\right)$ is unimodular.
\begin{figure}[H]
	\begin{center}
		\[
			\xymatrix@R-2pt{ 
				F\left(H_{4}(W)\right) \hskip-50pt 
				& \times \hskip-50pt 
				& F\left(H_{4}(W)\right) \ar@{-->}^(0.6){{\beta}}[rrr] 
				& & & \mathbb{Z} \ar@{=}[d]\\
				H_{4}(W)\ar@<-25pt>@{->>}[u]\ar@<20pt>@/_40pt/[ddd]_{i_{0*}} \hskip-50pt 
				& \times \hskip-50pt 
				& H_{4}(W) \ar^(0.6){\widetilde{\beta}}[rrr]\ar@{->>}[u]\ar@/^40pt/[ddd]^{i_{1*}} 
				& & &\mathbb{Z} \ar@{=}[d]\\
				H^{4}\left(W,\partial W\right)\ar@<-25pt>[u]^{\mathrm{PD}}_{\cong}\ar@<25pt>_{j_{1}^{*}}[d] \hskip-50pt 
				& \times \hskip-50pt 
				& H^{4}\left(W,\partial W\right) \ar_{j_{0}^{*}}[d] \ar[rrr]^(0.6){\left<\cdot\cup\cdot,[W,\partial W]\right>} \ar[u]^{\mathrm{PD}}_{\cong} 
				& & & \mathbb{Z} \ar@{=}[d]\\ 
				H^{4}\left(W,M_{1}\right) \ar@<25pt>[d]_{\mathrm{PD}}^{\cong} \hskip-50pt 
				& \times \hskip-50pt 
				& H^{4}\left(W,M_{0}\right) \ar[rrr]^(0.6){\left<\cdot\cup\cdot,[W,\partial W]\right>}\ar[d]_{\mathrm{PD}}^{\cong} 
				& & & \mathbb{Z}\ar@{=}[d]\\ 
				H_{4}\left(W,M_{0}\right) \hskip-50pt 
				& \times \hskip-50pt 
				& H_{4}\left(W,M_{1}\right) \ar[rrr]^(0.6){{\beta}_{W}} 				
				& & & \mathbb{Z} 	
			}
		\]
	\end{center}
	\caption{The intersection pairings}
	\label{Figure: The intersection pairings}
\end{figure}

As a consequence, we obtain the isomorphism
\begin{align*}
	& \left(H_{4}\left(W,M_{0}\right)\xleftarrow{\partial_{0}}H_{5}(\mathcal{B},W)\xrightarrow{\partial_{1}}H_{4}\left(W,M_{1}\right),{\beta}_{W}\right)\\
	\cong& \left(F\left(H_{4}(W)\right)\xlongleftarrow{\iota}F\left(\ker\left(H_{4}(\psi)\right)\right)\xlongrightarrow{\iota}F\left(H_{4}(W)\right),\beta\right)\oplus\left(0\xlongleftarrow{0}\left(\mathbb{Z}/2\right)^{2}\xlongrightarrow{0}0,0\right),
\end{align*}
thereby completing the proof of Proposition \ref{Proposition: identify surgery obstruction}.
\end{pf}

Before proving Proposition \ref{Proposition: identify surgery obstruction, G_3^p S^5}, we recall some homology groups of $BSpin$ and $BSpin\times\left(\mathbb{Z}\big/p,2\right)$.

\begin{lem}\label{Lemma: homology of BSpin and BSpin times K(Z/p,2)}
	For an odd prime $p$, the cohomology groups of $BSpin$ up to dimension $6$ and homology groups of $BSpin$, $\mathcal{B}=BSpin\times K\left(\mathbb{Z}\big/p,2\right)$ up to dimension $5$ are given as follows.
	\begin{table}[h]
		\centering
		\caption{$BSpin$ and $\mathcal{B}=BSpin\times K\left(\mathbb{Z}\big/p,2\right)$}
		\begin{tabular}{cccccccc}
			\toprule[1pt]
			$q$ & $0$ & $1$ & $2$ & $3$ & $4$ & $5$ & $6$ \\
			\midrule[0.3pt]
			$H^{q}\left(BSpin\right)$ & $\mathbb{Z}$ & $0$ & $0$ & $0$ & $\mathbb{Z}$ & $0$ & $\mathbb{Z}\big/2$\\
			$H_{q}\left(BSpin\right)$ & $\mathbb{Z}$ & $0$ & $0$ & $0$ & $\mathbb{Z}$ & $\mathbb{Z}\big/2$ & \\
			$H_{q}\left(\mathcal{B}\right)$ & $\mathbb{Z}$ & $0$ & $\mathbb{Z}\big/p$ & $0$ & $\mathbb{Z}\oplus\mathbb{Z}\big/p$ & $\mathbb{Z}\big/2$ &\\
			\bottomrule[1pt]
		\end{tabular}
		\label{Table: homology gps of BSpin & BSpin times K(Z/p,2)}
	\end{table}
\end{lem}
\begin{pf}
	The cohomology groups of $BSpin$ up to dimension $6$ can be deduced from \cite[Section 6, Theorem D']{duan2019characteristicclassesweylinvariants}. Then we apply universal coefficient theorem and obtain homology groups of $BSpin$ up to dimension $5$. Lemma \ref{Lemma: H_{q}(Z/p,2)} gives the lower homology groups of $K\left(\mathbb{Z}\big/p,2\right)$, and we apply K\"unneth formula to compute homology groups of $\mathcal{B}=BSpin\times K\left(\mathbb{Z}\big/p,2\right)$ up to dimension $5$.
\end{pf}

\begin{pf}[of Proposition \ref{Proposition: identify surgery obstruction, G_3^p S^5}]
	The proof is similar to that of Proposition \ref{Proposition: identify surgery obstruction}. Hence we omit the common arguments and exhibit what is different in the proof. Again we begin with the long exact ladder (Figure \ref{Figure: LEL of Hurewicz homom ass to (B,W)}). Now $\mathcal{B}=BSpin\times\left(\mathbb{Z}\big/p,2\right)$, $BSpin$ is $3$-connected with $\pi_{4}(Spin)\cong\mathbb{Z}$ and $\pi_{5}(BSpin)=0$. The homology groups of $\mathcal{B}$ is listed in Lemma \ref{Lemma: homology of BSpin and BSpin times K(Z/p,2)}. Hence we obtain Figure \ref{Figure: LEL of Hurewicz homom ass to (B,W), groups identified, spin case}.
	\begin{figure}[H]
		\begin{center}
			\[
			\xymatrix@R-2pt{
				0\ar[r]\ar[d] & \pi_{5}(\mathcal{B},W)\ar[r]\ar[d]^{\cong} & \pi_{4}(W)\ar[r]^(0.57){\pi_{4}(\psi)}\ar[d] & \mathbb{Z} \ar[r]\ar[d] & 0\ar[d]\\
				\mathbb{Z}\big/2\ar[r] & H_{5}(\mathcal{B},W)\ar[r] & H_{4}(W)\ar[r]^(0.48){H_{4}(\psi)} & \mathbb{Z}\oplus\mathbb{Z}\big/p\ar[r] & 0
			}
			\]
		\end{center}
		\caption{The long exact ladder of Hurewicz homomorphisms associated to $\left(\mathcal{B}, W\right)$ with groups identified, $\mathcal{B}=BSpin\times K\left(\mathbb{Z}\big/p,2\right)$}
		\label{Figure: LEL of Hurewicz homom ass to (B,W), groups identified, spin case}
	\end{figure} 
	As a consequence, the surgery obstruction $\theta(W,\psi)$ is again represented by the diagram
	$$\left(H_{4}\left(W,M_{0}\right)\xleftarrow{\partial_{0}}H_{5}(\mathcal{B},W)\xrightarrow{\partial_{1}}H_{4}\left(W,M_{1}\right),{\beta}_{W}\right).$$
	To further study $\partial_{i}$, we consider the long exact braid of (relative) homology groups associated to the triple $\left(\mathcal{B},W,M_{i}\right)$ (Figure \ref{Figure: long exact braid ass to B,W,M_i}). Note that some arguments in the proof of Proposition \ref{Proposition: identify surgery obstruction} still apply here, and we identify some groups in the partial long exact braid as follows:
	\begin{compactenum}
		\item $H_{5}(W)=H_{3}(W)=0$. 
		\item $H_{5}\left(M_{i}\right)=H_{3}\left(M_{i}\right)=0$ and $H_{4}\left(M_{i}\right)\cong\mathbb{Z}\big/p$.
		\item $H_{5}(\mathcal{B})\cong\mathbb{Z}\big/2$ and $H_{4}(\mathcal{B})\cong\mathbb{Z}\oplus\mathbb{Z}\big/p$.
		\item $H_{5}\left(W,M_{i}\right)\cong H^{3}\left(W,M_{1-i}\right)=0$ and $H_{4}\left(W,M_{i}\right)\cong\mathbb{Z}^{r}$ for some $r\geqslant0$.
		\item $H_{4}(\mathcal{B},W)=0$.
	\end{compactenum}
	Accordingly, we can identify certain homomorphisms in the partial long exact braid, in the sense that whether they are trivial, monomorphic, epimorphic or isomorphic, obtaining Figure \ref{Figure: long exact braid ass to B,W,M_i, groups identified, B=BSpin times K(Z/p,2)}.
	\begin{figure}[H]
		\begin{center}
			$$
			\xymatrix@C=-1.2ex{
				&
				& \ \ \ \ \ \ 0\ \ \ \ \ \ \ar@/^25pt/[rr] \ar[rd]
				&
				& \ \mathbb{Z}\big/p\  \ar@/^25pt/[rr]_{H_{4}\left(\varphi_{i}\right)} \ar@{>->}[rd]
				&
				& \ \mathbb{Z}\oplus\mathbb{Z}\big/p\  \ar@/^25pt/[rr] \ar@{->>}[rd]
				&
				& \ \ \ \ \ \ 0\ \ \ \ \ \  \ar[rd]
				&
				\\
				&\ \ \ \ \ \  0 \ \ \ \ \ \ \ar[ru] \ar[rd]
				&
				&  H_{5}\left(\mathcal{B},M_{i}\right)  \ar[ru] \ar@{>->}[rd]
				&
				& \ \ \ H_{4}(W)\ \ \ \ar@{->>}[ru]^{H_{4}(\psi)} \ar@{->>}[rd]
				&
				&  H_{4}\left(\mathcal{B},M_{i}\right) \ar[ru] \ar[rd]
				&
				& \ \ \ \ \ \ 0\ \ \ \ \ \ 
				\\
				\ \ \ \ \ \ 0\ \ \ \ \ \ \ar@{->>}[ru] \ar@/_25pt/[rr]
				&
				& \ \ \mathbb{Z}\big/2\ \ \ar@{>->}[ru] \ar@{>->}@/_25pt/[rr]
				&
				& H_{5}(\mathcal{B}, W) \ar[ru] \ar@/_25pt/[rr]^{\partial_{i}}
				&
				& \ \ \mathbb{Z}^{r}\ \  \ar@{->>}[ru] \ar@/_25pt/[rr]
				&
				& \ \ \ \ \ 0\ \ \ \ \  \ar[ru] 
				&
			}
			$$
		\end{center}
		\caption{The long exact braid of (relative) homology groups associated to $\left(\mathcal{B}, W, M_{i}\right)$ with groups and homomorphisms identified, $\mathcal{B}=BSpin\times K\left(\mathbb{Z}\big/p,2\right)$}
		\label{Figure: long exact braid ass to B,W,M_i, groups identified, B=BSpin times K(Z/p,2)}
	\end{figure}
	
	We also read from Figure \ref{Figure: long exact braid ass to B,W,M_i, groups identified, B=BSpin times K(Z/p,2)} that $H_{4}(W)\cong\mathbb{Z}^{r}\oplus\mathbb{Z}\big/p$, and it remains to determine the homomorphism $H_{4}\left(M_{i}\right)\xrightarrow{H_{4}(\varphi_{i})}H_{4}(\mathcal{B})$ for further identification of remaining groups. According to the product structure, $H_{4}(\varphi_{i})$ is identified with
	\[H_{4}\left(M_{i}\right)\xrightarrow{\left(H_{4}\left(\nu_{i}\right),H_{4}\left(\widehat{\varphi_{i}}\right)\right)}H_{4}(BSpin)\oplus H_{4}\left(\mathbb{Z}\big/p,2\right).\]
	Here and henceforth we adapt the following convention as Zhubr did in \cite{Zhubr75}. If $E$ is a (co)homology theory, $G$ is an abelian group and $n$ is a positive integer, then $E\left(K\left(G,n\right)\right)$ is abbreviated as $E(G,n)$. Since $H_{4}\left(M_{i}\right)\cong\mathbb{Z}\big/p$ and $H_{4}(BSpin)\cong\mathbb{Z}$, we have $H_{4}\left(\widetilde{\nu_{i}}\right)=0$ and it suffices to study $H_{4}\left(M_{i}\right)\xrightarrow{H_{4}\left(\widehat{\varphi_{i}}\right)}H_{4}\left(\mathbb{Z}\big/p,2\right).$ We will show that $H_{4}\left(\widehat{\varphi_{i}}\right)=0$ and thus $H_{4}\left({\varphi_{i}}\right)=0$.
	
	For convenience we denote by $h^{q}$ and $h_{q}$ the mod $p$ cohomology and homology till the end of this proof. Since $H_{4}\left(M_{i}\right)$ and $H_{4}\left(\mathbb{Z}\big/p,2\right)$ are both isomorphic to $\mathbb{Z}\big/p$, the mod $p$ homomorphism $H_{4}\left(M_{i}\right)\xrightarrow{\rho_{p}}h_{4}\left(M_{i}\right)$ and $H_{4}\left(\mathbb{Z}\big/p,2\right)\xrightarrow{\rho_{p}}h_{4}\left(\mathbb{Z}\big/p,2\right)$ are both isomorphism, which can be deduced from the long exact sequence of homology groups associated to the short exact sequence $\mathbb{Z}\rightarrowtail\mathbb{Z}\twoheadrightarrow\mathbb{Z}\big/p$ of coefficient rings. Hence it is sufficient to determine $h_{4}\left(\widehat{\varphi_{i}}\right)$, which is equivalent to study $h^{4}\left(\widehat{\varphi_{i}}\right)$ by universal coefficient theorem for field coefficients. See also Figure \ref{Figure: mod p reductions are isom}. 
	\begin{figure}[H]
		\begin{center}
			\[
				\xymatrix@R-2pt{
					H_{4}\left(M_{i}\right)\ar[d]_{H_{4}\left(\widehat{\varphi_{i}}\right)}\ar[r]^{\rho_{p}}_{\cong} & h_{4}\left(M_{i}\right)\ar[d]_{h_{4}\left(\widehat{\varphi_{i}}\right)} & h^{4}\left(M_{i}\right)\\
					H_{4}\left(\mathbb{Z}\big/p,2\right)\ar[r]^{\rho_{p}}_{\cong} & h_{4}\left(\mathbb{Z}\big/p,2\right) & h^{4}\left(\mathbb{Z}\big/p,2\right)\ar[u]_{h^{4}\left(\widehat{\varphi_{i}}\right)}
				}
			\]
		\end{center}
		\caption{The mod $p$ reductions are isomorphisms}
		\label{Figure: mod p reductions are isom}
	\end{figure}
	By assumption $\widehat{\varphi_{i}}$ is a $3$-equivalence, hence $h^{2}\left(\widehat{\varphi_{i}}\right)$ is an isomorphism. Let $\iota_{p}$ denote the generator of $h^{2}\left(\mathbb{Z}\big/p,2\right)\cong\mathbb{Z}\big/p$ that corresponds to $\left[\mathrm{id}\right]\in\left[K\left(\mathbb{Z}\big/p,2\right),K\left(\mathbb{Z}\big/p,2\right)\right]$ under the Brown representability isomorphism. Then $\widehat{\varphi_{i}}^{*}\iota_{p}$ generates $h^{2}\left(M_{i}\right)$ and $\iota_{p}^{2}$ generates $h^{4}\left(\mathbb{Z}\big/p,2\right)\cong\mathbb{Z}\big/p$ (see \cite[\S 4.L, p.~500]{HatcherAT}). Then it follows from Lemma \ref{Lemma: coh ring & char class of G_3^p S^5} that $\widehat{\varphi_{i}}^{*}\left(\iota_{p}^{2}\right)=\left(\widehat{\varphi_{i}}^{*}\iota_{p}\right)^{2}=0$. 
	
	As a consequence, $h^{4}\left(\widehat{\varphi_{i}}\right)=H_{4}\left(\widehat{\varphi_{i}}\right)=H_{4}\left({\varphi_{i}}\right)=0$. Now the remaining groups and homomorphisms are determined as follows.
	\begin{compactenum}
		\item $H_{5}\left(\mathcal{B},M_{i}\right)\cong\mathbb{Z}\big/2p$, $H_{4}\left(\mathcal{B},M_{i}\right)\cong\mathbb{Z}\oplus\mathbb{Z}\big/p$.
		\item In the identification $H_{4}\left(W,M_{i}\right)\cong\mathbb{Z}^{r}$, we must have $r\geqslant2$. Since $H_{4}\left(W,M_{i}\right)\to H_{4}\left(\mathcal{B},M_{i}\right)$ is epic and $H_{4}\left(\mathcal{B},M_{i}\right)\cong\mathbb{Z}\oplus\mathbb{Z}\big/p$.
		\item $F\left(H_{5}(\mathcal{B},W)\right)=F\left(\ker\left(H_{4}(\psi)\right)\right)$.
	\end{compactenum} 
	Therefore, $H_{5}(\mathcal{B},W)\xrightarrow{\partial_{i}}H_{4}\left(W,M_{i}\right)$ is identified with $F\left(\ker\left(H_{4}(\psi)\right)\right)\oplus\mathbb{Z}\big/2p\xrightarrow{(\iota,0)}F\left(H_{4}(W)\right)$ and $\mathrm{coker}\ \iota\cong H_{4}(\mathcal{B})\cong\mathbb{Z}\oplus\mathbb{Z}\big/p$, and the argument of Figure \ref{Figure: The intersection pairings} still applies. We obtain the isomorphism
	\begin{align*}
		& \left(H_{4}\left(W,M_{0}\right)\xleftarrow{\partial_{0}}H_{5}(\mathcal{B},W)\xrightarrow{\partial_{1}}H_{4}\left(W,M_{1}\right),{\beta}_{W}\right)\\
		\cong& \left(F\left(H_{4}(W)\right)\xlongleftarrow{\iota}F\left(\ker\left(H_{4}(\psi)\right)\right)\xlongrightarrow{\iota}F\left(H_{4}(W)\right),\beta\right)\oplus\left(0\xlongleftarrow{0}\mathbb{Z}/2p\xlongrightarrow{0}0,0\right),
	\end{align*}
	thereby completing the proof of Proposition \ref{Proposition: identify surgery obstruction, G_3^p S^5}.
\end{pf}
	
	\section{Classify $\mathcal{G}^{3}(\mathrm{Wu})$-like manifolds and $\mathcal{G}^{3}_{p}\left(S^{5}\right)$-like manifolds}\label{Section: Prove of main thm}

In this section we further study the obstruction $\theta\left(W,\psi\right)$. We determine when it is elementary and prove Statement 1 of Theorems \ref{Theorem: main} and \ref{Theorem: main, spin}. Then we show that invariant $\lambda$ is a diffeomorphism invariant of $\mathcal{G}^{3}(\mathrm{Wu})$-like manifolds and invariant $\mu$ is a diffeomorphism invariant of $\mathcal{G}^{3}_{p}\left(S^{5}\right)$-like manifolds. These two families of manifolds are treated separately in Subsections \ref{Subsection: Classify G_3 Wu like mfds} and \ref{Subsection: Classify G_3^p S^5 like mfds} respectively. 

For this purpose we first recall the definition of an element being elementary, then we further identify the groups and homomorphisms occurring in the representative of $\theta\left(W,\psi\right)$. Besides, we also need certain Arf type invariant, and a detailed discussion of definition and properties are postponed to Section \ref{Section: Arf inv}.

By \cite[Section 5]{SurgeryAndDuality}, an element $\theta\in l_{8}\left(\left\{e\right\}\right)$ is elementary if it has a representation $\left(\mathcal{V}^{0}\xleftarrow{f_{0}}\mathcal{V}\xrightarrow{f_{1}}\mathcal{V}^{1},\beta\right)$, such that $\mathcal{V}$ admits a free subgroup $\mathcal{U}$ satisfying the following conditions:
\begin{compactenum}
	\item[(e1)] $\mathcal{U}\subset \mathcal{U}^{\perp}$;
	\item[(e2)] $\mathcal{U}$ maps injectively into $\mathcal{V}^{i}$ and the image $\mathcal{U}_{i}:=f_{i}(\mathcal{U})$ is a directed summand;
	\item[(e3)] $\beta$ induces an isomorphism $\mathcal{U}_{0}\to\left(\mathcal{V}^{1}\big/\mathcal{U}_{1}\right)\spcheck$.
\end{compactenum}

Since the torsion part does not affect bilinear form which takes value in $\mathbb{Z}$, it follows from Proposition \ref{Proposition: identify surgery obstruction}, Statement 3 that $\theta\left(W,\psi\right)$ is elementary if and only if
$$\theta_{0}\left(W,\psi\right):=\left[\left(F\left(H_{4}(W)\right)\xlongleftarrow{\iota}F\left(\ker\left(H_{4}(\psi)\right)\right)\xlongrightarrow{\iota}F\left(H_{4}(W)\right),\beta\right)\right]$$
is elementary.

Before giving the proof we introduce some notations.
If a topology space $X$ has a structure of CW complex that admits finite a $q$-skeleton for any $n\in\mathbb{N}$ (e.g., $\mathcal{B}=BSO$, $BSpin\times K\left(\mathbb{Z}\big/p,2\right)$ or a compact smooth manifold) and $\mathbb{F}$ is a field (e.g., $\mathbb{Z}\big/p$ for prime number $p$), then $H_{q}(X;\mathbb{F})$ and $H^{q}(X;\mathbb{F})$ are vector spaces of the same finite dimension over $\mathbb{F}$ and we have $H^{q}(X;\mathbb{F})\cong H_{q}(X;\mathbb{F})\spcheck$, $H_{q}(X;\mathbb{F})\cong H^{q}(X;\mathbb{F})\spcheck$. If $\left\{e_{1},\cdots,e_{s}\right\}$ is a basis of $H^{q}(X;\mathbb{F})$, then $\left\{e_{1}\spcheck,\cdots,e_{s}\spcheck\right\}$ denotes the dual basis of $H_{q}(X;\mathbb{F})$.

\subsection{Classify $\mathcal{G}^{3}(\mathrm{Wu})$-like manifolds}\label{Subsection: Classify G_3 Wu like mfds}

Throughout this subsection, $\mathcal{B}=BSO$, and $h^{q}$, $h_{q}$ denote the mod $2$ (co)homology.
\begin{lem}\label{Lemma: further identify H_4(B)}
	$H_{4}(\mathcal{B})$ admits the unique generator set $\left(p_{1}\spcheck,\alpha\right)$ such that
	\begin{compactenum}
		\item $H_{4}(\mathcal{B})=\mathbb{Z}\left\{p_{1}\spcheck\right\}\oplus\mathbb{Z}\big/2\left\{\alpha\right\}$;
		\item $\left<p_{1},p_{1}\spcheck\right>=1$;
		\item $\rho_{2}\left(p_{1}\spcheck\right)=\left(w_{2}^{2}\right)\spcheck$, $\rho_{2}(\alpha)=w_{4}\spcheck$.
	\end{compactenum}
\end{lem}
As a straightforward corollary, we have
\begin{cor}
	$H_{4}(\psi)$ is given by
	\begin{equation}\label{Equation: effect of H_4(psi)}
		(H_{4}(\psi))(x)=\left<p_{1}(W),x\right>p_{1}\spcheck+\left<w_{4}(W),\rho_{2}(x)\right>\alpha,\quad\forall x\in H_{4}(W).
	\end{equation}
\end{cor}

\begin{pf}[of Lemma \ref{Lemma: further identify H_4(B)}]
	The uniqueness is clear, and we focus on the proof of existence. Let $\alpha\in H_{4}(\mathcal{B})$ be the unique non-trivial torsion element and let $\gamma\in H_{4}(\mathcal{B})$ be an element that generates a $\mathbb{Z}$-summand.
	Since $H^{4}(\mathcal{B})\cong H_{4}(\mathcal{B})\spcheck$ and $H^{4}(\mathcal{B})=\mathbb{Z}\left\{p_{1}\right\}$, we must have $\left<p_{1},\alpha\right>=0$ and $\left<p_{1},\gamma\right>=1$. 
	
	Now we consider the effect of mod $2$ homomorphism $\rho_{2}$. The long exact sequence of homology groups of $\mathcal{B}$ induced from the short exact sequence $\mathbb{Z}\rightarrowtail\mathbb{Z}\twoheadrightarrow\mathbb{Z}\big/2$ of coefficients implies that $H_{4}(\mathcal{B})\xrightarrow{\rho_{2}}h_{4}(\mathcal{B})$ is epic. Hence we may set
	$$\rho_{2}\begin{pmatrix}
		\alpha & \gamma
	\end{pmatrix}=\begin{pmatrix}
	\left(w_{2}^{2}\right)\spcheck & w_{4}\spcheck
	\end{pmatrix}g,\ g=\begin{pmatrix}
	A & C \\ B & D
	\end{pmatrix}\in GL\left(2,\mathbb{Z}\big/2\right).$$
	The homomorphism $\rho_{2}$ also preserves Kronecker pairing, i.e. the diagram shown in Figure \ref{Figure: rho_2 & Kronecker pairing} commutes.
	\begin{figure}[h]
		\begin{center}
			\[
				\xymatrix@R-2pt{
					H^{4}(\mathcal{B}) \ar@<35pt>[d]^{\rho_{2}} \hskip-70pt
					& \times \hskip-50pt
					& H_{4}(\mathcal{B}) \ar[rr]^(0.55){\left<\cdot,\cdot\right>} \ar[d]^{\rho_{2}} 
					& & \mathbb{Z} \ar[d]^{\rho_{2}}\\
					h^{4}(\mathcal{B}) \hskip-70pt
					& \times \hskip-50pt
					& h_{4}(\mathcal{B}) \ar[rr]^(0.525){\left<\cdot,\cdot\right>} 
					& & \mathbb{Z}\big/2
				}
			\]
		\end{center}
		\caption{$\rho_{2}$ and Kronecker pairing}
		\label{Figure: rho_2 & Kronecker pairing}
	\end{figure}
	
	Then we apply $\rho_{2}$ to the identities $\left<p_{1},\alpha\right>=0$ and $\left<p_{1},\gamma\right>=1$. Since $\rho_{2}\left(p_{1}\right)=w_{2}^{2}$, we obtain $A=0$ and $C=1\in\mathbb{Z}\big/2$, and it follows from the non-degeneracy of $g$ that $B=1\in\mathbb{Z}\big/2$. Now we set $p_{1}\spcheck=\gamma+Dw_{4}\spcheck$ and the proof is completed.
\end{pf}

The parity of $\beta$ is also important in the analysis of $\theta_{0}(W,\psi)$. Recall that $\beta$ is even if $\beta(x,x)$ is even for any $x\in F\left(H_{4}(W)\right)$, and is odd if otherwise. We have the following result.
\begin{lem}\label{Lemma: (FH_4W, beta) is odd}
	$\left(F\left(H_{4}(W)\right),\beta\right)$ is odd.
\end{lem}

\begin{pf}
	By definition the Poincar\'{e}-Lefschetz duality induces the isomorphism
	$$\left(F\left(H_{4}(W)\right),\beta\right)\cong\left(F\left(H^{4}\left(W,\partial W\right)\right),\check{\beta}\right),$$ 
	where $\check{\beta}$ is induced from the cohomological intersection pairing $\left<\cdot\cup\cdot,[W,\partial W]\right>$ on $H^{4}(W,\partial W)$. It follows from \cite[Section 7]{Kervaire57} that relative Wu classes are also defined and have nice interaction with Stiefel-Whitney classes and Steenrod operations as in the closed case. Namely, given a compact $n$-manifold $X$ with boundary, there are unique classes $\mathrm{U}^{q}(X)\in h^{q}(X)$ such that 
	$$\left<\mathrm{Sq}^{q}(x),[X,\partial X]_{2}\right>=\left<x\cup \mathrm{U}^{q}(X),[X,\partial X]_{2}\right>,\ \forall x\in h^{n-q}(X,\partial X),$$
	where $[X,\partial X]_{2}:=[X,\partial X]_{\mathbb{Z}\big/2}\in h_{n}(X,\partial X)$ denotes the mod $2$ fundamental class, and if $\mathrm{U}(X)=\mathrm{U}^{0}(X)+\mathrm{U}^{1}(X)+\cdots$ denotes the total class, then we have 
	$$w(X)=\mathrm{Sq}\left(\mathrm{U}(X)\right).$$
	Hence for the given bordism $W$, we have
	\begin{eqnarray}\label{Equation: relative Wu class}
		\left<x\cup x,[W,\partial W]_{2}\right>=\left<x\cup \mathrm{U}^{4}(W),[W,\partial W]_{2}\right>,\ \forall x\in h^{4}(W,\partial W),
	\end{eqnarray}
	and it is routine to compute that $\mathrm{U}^{4}(W)=w_{4}(W)+w_{2}(W)^{2}$. Then we obtain
	\begin{align*}
		\left<x\cup x,[W,\partial W]\right>\ \mathrm{mod}\ 2&=\left<\left(\rho_{2}(x)\right)\cup\left(\rho_{2}(x)\right),[W,\partial W]_{2}\right>\\
		&=\left<\left(\rho_{2}(x)\right)\cup\left(w_{4}(W)+w_{2}(W)^{2}\right),[W,\partial W]_{2}\right>\\
		&=\left<w_{4}(W)+w_{2}(W)^{2},\rho_{2}\left(x\cap [W,\partial W]\right)\right>,\ \forall x\in H^{4}(W,\partial W).
	\end{align*}
	Therefore, $\check{\beta}$ is even if and only if 
	$$\left<w_{4}(W)+w_{2}(W)^{2},\rho_{2}(x)\right>=0,\ \forall x\in H_{4}(W).$$
	By our previous computational results, we have:
	\begin{compactenum}
		\item $w_{4}(W)+w_{2}(W)^{2}\neq0$. Since $W\xrightarrow{\psi}\mathcal{B}$ is a $4$-equivalence, $h^{4}(\mathcal{B})\xrightarrow{h^{4}(\psi)}h^{4}(W)$ is monic and $w_{4}(W)+w_{2}(W)^{2}=h^{4}(\psi)\left(w_{4}+w_{2}^{2}\right)\neq0$.
		\item $h_{4}(W)\cong h_{4}\left(M_{i}\right)\oplus\rho_{2}\left(H_{4}(W)\right)$, and each element $x\in h_{4}(W)$ can be uniquely written as $x=h_{4}\left(\iota_{i}\right)y+\rho_{2}(z)$ for certain $y\in h_{4}\left(M_{i}\right)$ and $z\in H_{4}(W)$. This follows from the long exact sequence of homology groups associated to $\left(W,M_{i}\right)$ (Figure \ref{Figure: split of h_4(W)}).
		\begin{figure}[h]
			\begin{center}
				\[
					\xymatrix@R-2pt{
						& H_{4}(W)\ar^{\mu_{2}}[d] 
						& 
						& 
						& \mathbb{Z}^{r}\oplus\mathbb{Z}\big/2 \ar[d]
						& \\
						H_{4}\left(M_{i}\right) \ar@{>->}^{H_{4}\left(\iota_{i}\right)}[r] \ar_{\rho_{2}}^{\cong}[d]
						& H_{4}(W) \ar@{->>}[r] \ar@{->>}^{\rho_{2}}[d]
						& H_{4}\left(W,M_{i}\right) \ar@{->>}^{\rho_{2}}[d]
						& \mathbb{Z}\big/2 \ar@{>->}[r] \ar_{\cong}[d]
						& \mathbb{Z}^{r}\oplus\mathbb{Z}\big/2 \ar@{->>}[r] \ar@{->>}[d]
						& \mathbb{Z}^{r} \ar@{->>}[d] \\
						h_{4}\left(M_{i}\right) \ar@{>->}^{h_{4}\left(\iota_{i}\right)}[r]
						& h_{4}(W) \ar@{->>}[r]
						& h_{4}\left(W,M_{i}\right) 
						& \mathbb{Z}\big/2 \ar@{>->}[r]
						& \left(\mathbb{Z}\big/2\right)^{r}\oplus\mathbb{Z}\big/2 \ar@{->>}[r]
						& \left(\mathbb{Z}\big/2\right)^{r}
					}
				\]
			\end{center}
			\caption{Decomposition of $h_{4}(W)$}
			\label{Figure: split of h_4(W)}
		\end{figure}
		\item $\left<w_{4}(W)+w_{2}(W)^{2},h_{4}\left(\iota_{i}\right)(y)\right>=0$, $\forall y\in h_{4}\left(M_{i}\right)$. By tubular neighborhood theorem we have
		$$\left<w_{4}(W)+w_{2}(W)^{2},h_{4}\left(\iota_{i}\right)(y)\right>=\left<w_{4}\left(M_{i}\right)+w_{2}\left(M_{i}\right)^{2},y\right>.$$
		While $M_{i}$ is a $\mathcal{G}^{3}(\mathrm{Wu})$-like manifold, $w_{4}\left(M_{i}\right)=w_{2}\left(M_{i}\right)^{2}=0$, hence the characteristic number mentioned above must vanish.
	\end{compactenum}
	Therefore, there is a class $x_{0}\in h_{4}(W)\cong h^{4}(W)\spcheck$ such that  $\left<w_{4}(W)+w_{2}(W)^{2},x_{0}\right>\neq0$ as $w_{4}(W)+w_{2}(W)^{2}\neq0$. Moreover, by the decomposition of $h_{4}(W)$ there is a class $z_{0}\in H_{4}(W)$ such that $x_{0}=\rho_{2}\left(z_{0}\right)$. Hence $\check{\beta}$ and thus $\beta$ must be odd.
\end{pf}

\begin{prop}\label{Propositioin: theta_0 elementary, equivalent condition}
	$\theta_{0}(W,\psi)$ is elementary if and only if the following conditions hold:
	\begin{compactenum}
		\item $\left(\mathcal{V},\beta\right)$ has signature $\sigma(\mathcal{V},\beta)=0$;
		\item $\beta(v,v)=0$;
		\item $g(v)=0$;
		\item $\Xi_{odd}\left(\mathcal{V},\beta,g\right)=0$.
	\end{compactenum}
\end{prop}
Here $\sigma(\mathcal{V},\beta)=0$ forces $\mathrm{rank}\ \mathcal{V}=2r$ to be even, and by Proposition \ref{Proposition: identify surgery obstruction} we have $r\geqslant1$. $\Xi_{odd}$ is defined for any tripple $\left(\mathcal{V},\beta,g\right)$, such that $\mathcal{V}\cong\mathbb{Z}^{2n}$ for some $n$, $\beta$ is an odd symmetric unimodular bilinear form over $\mathcal{V}$ with vanishing signature, and $\mathcal{V}\xrightarrow{g}\mathbb{Z}\big/2$ is a homomorphism. It can be shown that $\Xi_{odd}\left(\mathcal{V},\beta,g\right)=0$ if and only if $\left(\mathcal{V},\beta\right)$ admits a Lagrangian contained in $\ker g$. See \ref{Section: Arf inv} for a complete discussion.
	
\begin{pf}
	We begin with some notations. Denote $\mathcal{V}=F\left(H_{4}(W)\right)$. $H_{4}(\psi)$ determines an epimorphism $\mathcal{V}\xrightarrow{(f,g)}\mathbb{Z}\oplus\mathbb{Z}\big/2\cong H_{4}(\mathcal{B})$, where $f$, $g$ can be explicitly expressed as $f(x)=\left<p_{1}(W),x\right>$, $g(x)=\left<w_{4}(W),\rho_{2}(x)\right>$ and are both epic. 
	Then $F\left(\ker\left(H_{4}(\psi)\right)\right)=\ker(f,g)$ and we denote this group by $\mathcal{K}$. Let $\mathcal{K}\xrightarrow{\iota}\mathcal{V}$ be the inclusion of subgroup as before. Since $\beta$ is unimodular, there is a unique element $v=v_{f}\in \mathcal{V}$ such that $f(x)=\beta(x,v)$ for any $x\in \mathcal{V}$. Moreover, $v$ is primitive as $f$ is epic. Now we obtain a new expression of the representative for $\theta_{0}\left(W,\psi\right)$:
	$$\left(\mathcal{V}\xleftarrow{\iota}\mathcal{K}\xrightarrow{\iota}\mathcal{V},\beta\right).$$
	
	We consider stabilizations via hyperbolic forms as well. Let $\mathcal{H}^{k}=\mathbb{Z}^{2k}$ and let $h^{k}$ be the hyperbolic form on $\mathcal{H}^{k}$. For simplicity we also write ${H}^{k}=\left(\mathcal{H}^{k},h^{k}\right)$ and when $k=1$ we drop the superscripts. In particular ${H}^{k}$ is isomorphic to the orthogonal direct sum of $k$ copies of ${H}$. In the stabilization
	\begin{align*}
		&\left(\mathcal{V}\xleftarrow{\iota}\mathcal{K}\xrightarrow{\iota}\mathcal{V},\beta\right)\oplus{H}^{k}\\
		=&\left(\mathcal{V}\oplus \mathcal{H}^{k}\xleftarrow{(\iota,\mathrm{id})}\mathcal{K}\oplus \mathcal{H}^{k}\xrightarrow{(\iota,\mathrm{id})}\mathcal{V}\oplus \mathcal{H}^{k},\beta\oplus h^{k}\right),
	\end{align*}
	we set $\mathcal{V}_{k}=\mathcal{V}\oplus \mathcal{H}^{k}$, $\iota_{k}=(\iota,\mathrm{id})$, $\mathcal{K}_{k}=\mathcal{K}\oplus \mathcal{H}^{k}$, $\beta_{k}=\beta\oplus h^{k}$ for convenience. We also represent $\mathcal{K}_{k}$ in terms of $(f,g)$ as follows. Denote by $\mathcal{V}_{k}\xrightarrow{\mathrm{pr}_{\mathcal{V}}}\mathcal{V}$ the projection onto $\mathcal{V}$, and it is clear that 
	\begin{compactenum}
		\item $\left(\mathcal{V}_{k},\beta_{k}\right)\xrightarrow{\mathrm{pr}_{\mathcal{V}}}\left(\mathcal{V},\beta\right)$ is an epic isometry;
		\item $\mathcal{V}_{k}\xrightarrow{(f,g)\circ\mathrm{pr}_{\mathcal{V}}}\mathbb{Z}\oplus\mathbb{Z}\big/2$ is epic and can be written in terms of components $(f,g)\circ\mathrm{pr}_{\mathcal{V}}=\left(f_{k},g_{k}\right)$, $f_{k}=f\circ\mathrm{pr}_{\mathcal{V}}$, $g_{k}=g\circ\mathrm{pr}_{\mathcal{V}}=g\oplus0$;
		\item $\mathcal{K}_{k}=\ker\left(f_{k},g_{k}\right)$.
	\end{compactenum}
	Therefore, if we view $v=v_{f}\in \mathcal{V}_{k}$, it is still primitive and satisfies the property that $f_{k}\left(x'\right)=\beta_{k}\left(x',v\right)$ for all $x'\in \mathcal{V}_{k}$. It follows from isometry that $\beta_{k}(v,v)=\beta(v,v)$. Also note that stabilization via hyperbolic forms does not vary the parity, and $\beta_{k}$ is again odd.
	
	By definition and assumption, $\theta_{0}(W,\psi)$ is elementary if and only if there exists $k\in\mathbb{N}$ such that in the new representative
	$$
		\left(\mathcal{V}_{k}\xleftarrow{\iota_{k}}\mathcal{K}_{k}\xrightarrow{\iota_{k}}\mathcal{V}_{k},\beta_{k}\right),
	$$
	$\mathcal{V}_{k}$ admits a free subgroup $\mathcal{U}$ satisfying conditions (e1)$\sim$(e3). In our case $\mathcal{\mathcal{V}}^{0}=\mathcal{\mathcal{V}}^{1}=\mathcal{V}_{k}$, $\mathcal{\mathcal{V}}=\mathcal{K}_{k}$, $f_{0}=f_{1}=\iota_{k}$ is the inclusion and $\beta_{k}$ is a symmetric unimodular bilinear form on $\mathcal{\mathcal{V}}_{k}$. Hence condition (e3) implies condition (e1), and $\theta_{0}(W,\psi)$ is elementary if and only if there exists $k\in\mathbb{N}$ such that $\left(\mathcal{V}_{k},\beta_{k}\right)$ admits a Lagrangian contained in $\mathcal{K}_{k}=\ker\left(f_{k},g_{k}\right)$.
	
	Now we start the proof and begin with necessity. From the comparison of ranks and signatures we obtain $\mathrm{rank}\ \mathcal{V}=2r$ is even, $\mathrm{rank}\ \mathcal{U}=r+k$ and $\sigma(\mathcal{V},\beta)=0$. We claim that $v\in \mathcal{U}$. By assumption $\mathcal{U}\subset \mathcal{K}_{k}\subset\ker\left(f_{k}\right)$, hence $\beta_{k}(u,v)=f_{k}(v)=0$ for all $u\in \mathcal{U}$, namely $v\in \mathcal{U}^{\perp}=\mathcal{U}$. Now $v\in \mathcal{U}\subset \mathcal{K}_{k}=\mathcal{K}\oplus \mathcal{H}^{k}$ and by the characterization of $v$ it must contains in $\mathcal{K}=\ker(f,g)$. This justifies the necessity.
	
	Next we prove the sufficiency. Suppose all the conditions in Proposition \ref{Propositioin: theta_0 elementary, equivalent condition} are satisfied, and we shall show that there exists $k\in\mathbb{N}$ such that after $k$ times of stabilizations, $\left(\mathcal{V}_{k},\beta_{k}\right)$ admits a Lagrangian that is contained in $\mathcal{K}_{k}$. Now suppose $k\geqslant1$, and $\mathrm{rank}\ \mathcal{V}_{k}\geqslant2(r+k)\geqslant4$. Since $v$ is primitive and $\beta_{k}$ is unimodular, $(v)^{\perp}$ is a direct summand of $\mathcal{V}_{k}$ of rank $2(r+k)-1$. Moreover, since $\beta_{k}(v,v)=0$, we have $v\in (v)^{\perp}$ and $v$ is also primitive in $(v)^{\perp}$ and $\mathcal{K}_{k}$. Now we form $\overline{\mathcal{V}_{k}}:=(v)^{\perp}\big/\left<v\right>$, and $\overline{\mathcal{V}_{k}}\cong\mathbb{Z}^{2(r+k-1)}$. Furthermore, it is routine to verify that
	\begin{compactenum}
		\item $g_{k}$ induces a homomorphism $\overline{\mathcal{V}_{k}}\xrightarrow{\overline{g_{k}}}\mathbb{Z}\big/2$ and $\overline{g_{k}}$ is again epic. This can be deduced from $g_{k}(v)=g(v)=0$.
		\item $\beta_{k}$ induces a symmetric unimodular bilinear form $\overline{\beta_{k}}$ on $\overline{\mathcal{V}_{k}}$. This can be deduced from $\beta_{k}(v,v)=\beta(v,v)=0$.
		\item $\sigma\left(\overline{\mathcal{V}_{k}},\overline{\beta_{k}}\right)=0$. We can find another primitive element $w\in \mathcal{V}_{k}$ such that $\mathcal{H}_{0}:=\mathbb{Z}\left\{v,w\right\}$ is a direct summand of $\mathcal{V}_{k}$, the restriction of $\beta_{k}$ on $\mathcal{H}_{0}$ is hyperbolic and that $\left(\mathcal{V}_{k},\beta_{k}\right)$ admits the orthogonal decomposition
		$$\left(\mathcal{V}_{k},\beta_{k}\right)\cong\left(\mathcal{H}_{0},\left.\beta_{k}\right|_{\mathcal{H}_{0}}\right)\oplus\left(\mathcal{H}_{0},\left.\beta_{k}\right|_{\mathcal{H}_{0}}\right)^{\perp}.$$
		As a consequence, we have an isomorphic isometry $\sigma\left(\overline{\mathcal{V}_{k}},\overline{\beta_{k}}\right)\cong\left(\mathcal{H}_{0},\left.\beta_{k}\right|_{\mathcal{H}_{0}}\right)^{\perp}$ and the signature identity follows.
	\end{compactenum}
	Therefore, $\left(\mathcal{V}_{k},\beta_{k}\right)$ admits a Lagrangian that is contained in $\mathcal{K}_{k}=\ker\left(f_{k},g_{k}\right)$ if and only if $\left(\overline{\mathcal{V}_{k}},\overline{\beta_{k}}\right)$ admits a Lagrangian that is contained in $\overline{\mathcal{K}_{k}}:=\ker\left(\overline{g_{k}}\right)$. Here it can be shown that $\overline{\mathcal{K}_{k}}=\mathcal{K}_{k}\big/\left<v\right>$ and has index $2$ in $\overline{\mathcal{V}_{k}}$.
	It remains to determine when $\left(\overline{\mathcal{V}_{k}},\overline{\beta_{k}}\right)$ admits a Lagrangian that is contained in $\overline{\mathcal{K}_{k}}:=\ker\left(\overline{g_{k}}\right)$. By Proposition \ref{Proposition: sym unimodular odd bilinear (V, lambda) with vanishing signature admits Lagrangian contained in ker of epim onto Z/2} this is achieved if and only if $\Xi_{odd}\left(\overline{g_{k}}\right)=0$, and by Proposition \ref{Proposition: Arf inv of direct sum} $\Xi_{odd}\left(\overline{g_{k}}\right)=\Xi_{odd}\left({g_{k}}\right)=\Xi_{odd}(g)$.
	This completes the proof of Proposition \ref{Propositioin: theta_0 elementary, equivalent condition}.
\end{pf}

Now we transform abstract criterions for $\theta_{0}(W,\psi)$ being elementary in Proposition \ref{Propositioin: theta_0 elementary, equivalent condition} into computable characteristic numbers of $W$. Let us begin with a lemma on relative characteristic numbers.

\begin{lem}\label{Lemma: Relative char nums}
	Let $R$ be a coefficient ring and len $X$ be an $R$-oriented $n$-manifold with (possibly empty) boundary. Suppose $n=p+q$ for some positive integers $p$, $q$, and there are classes $x\in H^{p}(X;R)$, $y\in H^{q}(X;R)$ such that $x|_{\partial X}=0$ and $y|_{\partial X}=0$. Let $\widetilde{x}\in H^{p}(X,\partial X;R)$ and $\widetilde{y}\in H^{q}(X,\partial X;R)$ be any liftings of $x$ and $y$. Then 
	\[\left<\widetilde{x}\cup\widetilde{y},[X,\partial X]\right>\]
	does not depend on the choices of liftings and is denoted by $xy(X)$.
\end{lem}

\begin{pf}
	Let $x_{1},x_{2}\in H^{p}(X,\partial X;R)$ and $y_{1},y_{2}\in H^{q}(X,\partial X;R)$ be the liftings. We have the commutative diagram shown in Figure \ref{Figure: The relative cup prod}.
	\begin{figure}[H]
		\begin{center}
			\[
				\xymatrix@R-2pt{
					H^{p}(X;R) \hskip-50pt & \times \hskip-50pt & H^{q}(X,\partial X;R)\ar[r]^(0.45){\cdot\cup\cdot} & H^{p+q}(X,\partial X;R)\\
					H^{p}(X,\partial X;R)\ar@<-30pt>[u]^{j^{*}}\ar@<30pt>@{=}[d] \hskip-50pt & \times \hskip-50pt & H^{q}(X,\partial X;R)\ar@{=}[u]\ar[d]^{j^{*}}\ar[r]^(0.45){\cdot\cup\cdot} & H^{p+q}(X,\partial X;R)\ar@{=}[u]\ar@{=}[d]\\
					H^{p}(X,\partial X;R) \hskip-50pt & \times \hskip-50pt & H^{q}(X;R)\ar[r]^(0.42){\cdot\cup\cdot} & H^{p+q}(X,\partial X;R)
				}
			\]
			\caption{The relative cup products}
			\label{Figure: The relative cup prod}
		\end{center}
	\end{figure}
	Then we obtain $x_{\varepsilon}\cup y_{1}=x_{\varepsilon}\cup y_{2}=x_{\varepsilon}\cup y$ and $x_{1}\cup y_{\varepsilon}=x_{2}\cup y_{\varepsilon}=x\cup y_{\varepsilon}$ for $\varepsilon=1,2$. Therefore, $x_{\varepsilon}\cup y_{\varepsilon'}$ are identical for $\varepsilon,\varepsilon'=1,2$. This completes the proof.
\end{pf}

\begin{lem}\label{Lemma: transform abstract charaterization of obstruction into computable char numbers}
	We have $\beta(v,v)=\Lambda(W)$ and $g(v)=w_{2}^{2}w_{4}(W)$.
\end{lem}

\begin{pf}
	By definition, $v\in \mathcal{V}=F\left(H_{4}(W)\right)$ is uniquely characterized by
	$$\left<p_{1}(W),x\right>=\beta(x,v),\ \forall x\in F\left(H_{4}(W)\right).$$
	According to Poincar\'{e}-Lefschetz duality, there is a unique class $\widetilde{v}\in F\left(H^{4}(W,\partial W)\right)$ such that $v=\widetilde{v}\cap [W,\partial W]$, and we define $\widetilde{x}\in F\left(H^{4}(W,\partial W)\right)$ likewise. Then the equation above can be rewritten as
	$$\left<p_{1}(W)\cup\widetilde{x},[W,\partial W]\right>=\left<\widetilde{v}\cup\widetilde{x},[W,\partial W]\right>,\ \forall\widetilde{x}\in F\left(H^{4}(W,\partial W)\right).$$
	Denote by $W\xrightarrow{j}(W,\partial W)$ the inclusion, and the right hand side is also equal to $\left<j^{*}\widetilde{v}\cup\widetilde{x},[W,\partial W]\right>$. Hence $j^{*}\widetilde{v}$ and $p_{1}(W)$ represent the same element in $F\left(H^{4}(W)\right)=H^{4}(W)$, $\widetilde{v}$ is a lifting of $p_{1}(W)$ and
	$$\beta(v,v)=\left<\widetilde{v}\cup\widetilde{v},[W,\partial W]\right>=\Lambda(W).$$
	
	Now we determine the other characteristic number.
	\begin{align*}
		g(v)=\left<w_{4}(W),\rho_{2}(v)\right>&=\left<w_{4}(W),\rho_{2}\left(\widetilde{v}\cap[W,\partial W]\right)\right>\\
		&=\left<w_{4}(W),\left(\rho_{2}\left(\widetilde{v}\right)\right)\cap[W,\partial W]_{2}\right>\\
		&=\left<w_{4}(W)\cup\left(\rho_{2}\left(\widetilde{v}\right)\right),[W,\partial W]_{2}\right>\\
		&=w_{2}^{2}w_{4}(W).
	\end{align*}
	Here $[W,\partial W]_{2}$ denotes the mod $2$ fundamental class, and
	the last idendity is explained as follows. Since $j^{*}\left(\rho_{2}\left(\widetilde{v}\right)\right)=\rho_{2}\left(j^{*}\left(\widetilde{v}\right)\right)=\rho_{2}\left(p_{1}(W)\right)=w_{2}(W)^{2}$, $\rho_{2}\left(\widetilde{v}\right)$ is a lifting of $w_{2}(W)^{2}$.
\end{pf}

The following lemma combined with Lemma \ref{Lemma: transform abstract charaterization of obstruction into computable char numbers} implies that Proposition \ref{Propositioin: theta_0 elementary, equivalent condition}, Condition 3 is redundant.

\begin{lem}\label{Lemma: w_{2}^{2}w_{4} vanish for oriented 8mfd}
	If $X$ is an oriented closed $8$-manifold, or a compact oriented $8$-manifold whose boundary has vanishing characteristic classes $w_{2}^{2}$ and $w_{4}$,
	then $w_{2}^{2}w_{4}(X)=0$.
\end{lem}

\begin{pf}
	We begin the coboudary case. Since $w_{2}(\partial X)^{2}=0$, $w_{2}(X)^{2}$ admits a lifting $\widetilde{w_{2}(X)^{2}}\in h^{4}(X,\partial X)$. Apply the formula and proposition of relative $4$th Wu class, and we have
	\begin{align*}
		w_{2}^{2}w_{4}(X)&=\left<\widetilde{w_{2}(X)^{2}}\cup w_{4}(X),[X,\partial X]_{2}\right>\\
		&=\left<\widetilde{w_{2}(X)^{2}}\cup\mathrm{U}^{4}(X),[X,\partial X]_{2}\right>+\left<\widetilde{w_{2}(X)^{2}}\cup w_{2}(X)^{2},[X,\partial X]_{2}\right>\\
		&=\left<\widetilde{w_{2}(X)^{2}}\cup\widetilde{w_{2}(X)^{2}},[X,\partial X]_{2}\right>+\left<\widetilde{w_{2}(X)^{2}}\cup w_{2}(X)^{2},[X,\partial X]_{2}\right>\\
		&=2\left<\widetilde{w_{2}(X)^{2}}\cup w_{2}(X)^{2},[X,\partial X]_{2}\right>=0.
	\end{align*}
	The proof of the closed case is similar, where we use absolute classes $w_{2}(X)^{2}$, $v_{4}(X)$ and $[X]_{2}$.
\end{pf}

The invariant $\Xi_{odd}(g)$ can be equivalently expressed as known characteristic number.
\begin{lem}\label{Lemma: Arf_odd=p_1^2 mod 2}
	$\Xi_{odd}(g)=\Lambda(W)\ \mathrm{mod}\ 2$.
\end{lem}
In particular, $\Xi_{odd}(g)=0$ if $\Lambda(W)=0$.

\begin{pf}
	Recall that $\mathcal{V}=F\left(H_{4}(W)\right)$ and $\beta$ is the homological intersection form on $\mathcal{V}$ such that $\beta$ is unimodular and $\sigma\left(\mathcal{V},\beta\right)=0$. Let $\left\{e_{i},f_{i}:1\leqslant i\leqslant r\right\}$ be a standard orthonormal basis of $(\mathcal{V},\beta)$, and there are unique classes $\widetilde{e_{i}},\widetilde{f_{i}}\in F\left(H^{4}(W,\partial W)\right)$ such that $e_{i}=\widetilde{e_{i}}\cap[W,\partial W]$, $f_{i}=\widetilde{f_{i}}\cap[W,\partial W]$. By Proposition \ref{Proposition: sym unimodular odd bilinear (V, lambda) with vanishing signature admits Lagrangian contained in ker of epim onto Z/2} and Remark \ref{Remark: on Arf type inv},
	\begin{align*}
		\Xi_{odd}(g)&=\sum_{i=1}^{r}g\left(e_{i}\right)+\sum_{i=1}^{r}g\left(f_{i}\right)\\
		&=\sum_{i=1}^{r}\left<w_{4}(W)\cup\rho_{2}\left(\widetilde{e_{i}}\right),[W,\partial W]_{2}\right>+\sum_{i=1}^{r}\left<w_{4}(W)\cup\rho_{2}\left(\widetilde{f_{i}}\right),[W,\partial W]_{2}\right>.
	\end{align*}
	Recall the formula for relative Wu class $\mathrm{U}^{4}(W)=w_{4}(W)+w_{2}(W)^{2}$. Now we apply formula \eqref{Equation: relative Wu class} and obtain
	\begin{align*}
		\Xi_{odd}(g)&=\sum_{i=1}^{r}\left<U^{4}(W)\cup\rho_{2}\left(\widetilde{e_{i}}\right),[W,\partial W]_{2}\right>+\sum_{i=1}^{r}\left<U^{4}(W)\cup\rho_{2}\left(\widetilde{f_{i}}\right),[W,\partial W]_{2}\right>\\
		&+\sum_{i=1}^{r}\left<w_{2}(W)^{2}\cup\rho_{2}\left(\widetilde{e_{i}}\right),[W,\partial W]_{2}\right>+\sum_{i=1}^{r}\left<w_{2}(W)^{2}\cup\rho_{2}\left(\widetilde{f_{i}}\right),[W,\partial W]_{2}\right>\\
		&=\sum_{i=1}^{r}\left<\rho_{2}\left(\widetilde{e_{i}}\right)\cup\rho_{2}\left(\widetilde{e_{i}}\right),[W,\partial W]_{2}\right>+\sum_{i=1}^{r}\left<\rho_{2}\left(\widetilde{f_{i}}\right)\cup\rho_{2}\left(\widetilde{f_{i}}\right),[W,\partial W]_{2}\right>\\
		&+\sum_{i=1}^{r}\left<\rho_{2}\left(p_{1}(W)\right)\cup\rho_{2}\left(\widetilde{e_{i}}\right),[W,\partial W]_{2}\right>+\sum_{i=1}^{r}\left<\rho_{2}\left(p_{1}(W)\right)\cup\rho_{2}\left(\widetilde{f_{i}}\right),[W,\partial W]_{2}\right>\\
		&=\rho_{2}\left(\sum_{i=1}^{r}\left(\beta\left(e_{i},e_{i}\right)+\beta\left(f_{i},f_{i}\right)\right)\right)\\
		&+\rho_{2}\left(\sum_{i=1}^{r}\left<p_{1}(W)\cup\widetilde{e_{i}},[W,\partial W]\right>+\sum_{i=1}^{r}\left<p_{1}(W)\cup\widetilde{f_{i}},[W,\partial W]\right>\right)\\
		&=\rho_{2}\left(\sum_{i=1}^{r}\left<p_{1}(W)\cup\widetilde{e_{i}},[W,\partial W]\right>+\sum_{i=1}^{r}\left<p_{1}(W)\cup\widetilde{f_{i}},[W,\partial W]\right>\right).
	\end{align*}
	Now we set $a_{i}=\left<p_{1}(W),e_{i}\right>$, $b_{i}=\left<p_{1}(W),f_{i}\right>$, and we have
	\begin{align*}
		p_{1}(W)&=j^{*}\left(\sum_{i=1}^{r}a_{i}\widetilde{e_{i}}-\sum_{i=1}^{r}b_{i}\widetilde{f_{i}}\right),\\
		\Lambda(W)&=\sum_{i=1}^{r}a_{i}^{2}+\sum_{i=1}^{r}b_{i}^{2},\\
		\Xi_{odd}(g)&=\rho_{2}\left(\sum_{i}^{r}a_{i}-\sum_{i=1}^{r}b_{i}\right).
	\end{align*}
	As a consequence, we obtain $\Xi_{odd}(g)=\Lambda(W)\ \mathrm{mod}\ 2$.
\end{pf}

\begin{pf}[of Theorem \ref{Theorem: main}, Statement 1]
	We have explained in Introduction that $\lambda$-invariant is well-defined for $\mathcal{G}^{3}(\mathrm{Wu})$-like manifolds, and now we start to prove the necessity of Theorem \ref{Theorem: main}, Statement 1. Assume two $\mathcal{G}^{3}(\mathrm{Wu})$-like manifolds $M_{i}$ $(i=0,1)$ are diffeomorphic, and we shall prove that $\lambda\left(M_{0}\right)=\lambda\left(M_{1}\right)$. Suppose $M_{1}\xrightarrow{h}M_{0}$ is a diffeomorphism. Then by Smale's $h$-cobordism theorem there is a bordism $W$ from $M_{0}$ to $M_{1}$ and a diffeomorphism $W\xrightarrow{H}M_{0}\times I$ such that $H|_{M_{0}}=\mathrm{id}_{M_{0}}$ and $H|_{M_{1}}=h$. Let $V_{0}$ be a coboundary of $M_{0}$ with vanishing signature. Then $V_{1}:=\left(-V_{0}\right)\cup W$ is a coboundary of $M_{1}$ and we have
	\begin{compactenum}
		\item $V_{1}$ has vanishing signature,
		\item $\Lambda\left(V_{0}\right)=\Lambda\left(V_{1}\right)$.
	\end{compactenum}
	Therefore, $\lambda\left(M_{0}\right)=\lambda\left(M_{1}\right)$.
	
	It remains to show that two $\mathcal{G}^{3}(\mathrm{Wu})$-like manifolds are diffeomorphic provided that they have the same $\lambda$-invariant.
	First we recall certain additive formulae. Their proof are basic algebraic topology and would be omitted.
	\begin{lem}\label{Lemma: char nums additive wrt gluing via boundary}
		Let $R$ denote the coefficient ring. Let $W_{1}$, $W_{2}$ be two connected compact $R$-oriented $n$-manifolds ($n>1$) with boundaries $M_{1}$, $M_{2}$ respectively (empty or disconnected). Suppose we have $x_{i}\in H^{p}\left(W_{i};R\right)$, $y_{i}\in H^{q}\left(W_{i};R\right)$ such that $p$, $q>0$, $p+q=n$ and $x_{i}|_{M_{i}}=0$, $y_{i}|_{M_{i}}=0$. Let $\widetilde{x_{i}}\in H^{p}\left(W_{i},M_{i};R\right)$, $\widetilde{y_{i}}\in H^{q}\left(W_{i},M_{i};R\right)$ be the liftings of $x_{i}$, $y_{i}$ respectively.
		\begin{compactenum}
			\item Form $W=W_{1}\# W_{2}$ and denote $M=\partial W$. Then $x_{i}$, $y_{i}$ uniquely determine the classes $x\in H^{p}(W;R)$, $y\in H^{q}(W;R)$ under the canonical isomorphisms $H^{j}(W;R)\cong H^{j}\left(W_{1};R\right)\oplus H^{j}\left(W_{2};R\right)$ $(j=p,q)$ with $x|_{M}=0$, $y|_{M}=0$. Choose liftings $\widetilde{x}\in H^{p}(W,M;R)$, $\widetilde{y}\in H^{q}(W,M;R)$. Then we have the following identity, in which the value of each summand does not depend on choices of liftings:
			\[\left<\widetilde{x}\cup\widetilde{y},\left[W,M\right]_{R}\right>=\left<\widetilde{x_{1}}\cup\widetilde{y_{1}},\left[W_{1},M_{1}\right]_{R}\right>+\left<\widetilde{x_{2}}\cup\widetilde{y_{2}},\left[W_{2},M_{2}\right]_{R}\right>.\]
			\item When $M_{1}=M_{2}=M$, we glue $W_{i}$ along the common coboundary with compactible orientations and obtain the oriented closed manifold $X=W_{1}\cup_{M} W_{2}$. In the Mayer-Vietoris sequence 
			$$H^{j}(X;R)\to H^{j}\left(V_{1},M;R\right)\oplus H^{j}\left(V_{2},M;R\right),$$ 
			let $x\in H^{p}(X;R)$, $y\in H^{q}(X;R)$ be liftings of $\left(x_{1},x_{2}\right)$, $\left(y_{1},y_{2}\right)$. Then we have the following identity, in which the value of each summand does not depend on choices of liftings:
			\[\left<x\cup y,\left[W\right]_{R}\right>=\left<\widetilde{x_{1}}\cup\widetilde{y_{1}},\left[W_{1},M_{1}\right]_{R}\right>+\left<\widetilde{x_{2}}\cup\widetilde{y_{2}},\left[W_{2},M_{2}\right]_{R}\right>.\]
		\end{compactenum}
	\end{lem}
	Now suppose we have two $\mathcal{G}^{3}(\mathrm{Wu})$-like manifolds $M_{i}$ $(i=0,1)$ with $\lambda\left(M_{0}\right)=\lambda\left(M_{1}\right)$, saying $V_{i}$ is a coboundary of $M_{i}$ with vanishing signature and $X$ is a closed $8$-manifold with vanishing signature such that $\Lambda\left(V_{0}\right)-\Lambda\left(V_{1}\right)=\Lambda(X)$. Form $W_{0}:=V_{0}\# V_{1}\#(-X)$, and by Lemma \ref{Lemma: char nums additive wrt gluing via boundary}, Statement 2, $W_{0}$ is a bordism between $M_{0}$, $M_{1}$ with $\sigma\left(W_{0},\partial W_{0}\right)=0$ and $\Lambda\left(W_{0}\right)=0$. It is unclear whether the classifying map $W_{0}\xrightarrow{\psi_{0}}\mathcal{B}$ is $4$-connected. Suppose after appropriate sequence of surgeries on $W_{0}$ we obtain another bordism $W_{1}$ between $M_{i}$ with $W_{1}\xrightarrow{\psi_{1}}\mathcal{B}$ a $4$-equivalence. Form closed $8$-manifolds $X_{0}:=\left(-W_{0}\right)\cup_{M_{0}\sqcup M_{1}}W_{0}$, $X_{1}:=\left(-W_{0}\right)\cup_{M_{0}\sqcup M_{1}}W_{1}$, and $X_{1}$ is obtained from $X_{0}$ via exactly the same sequence of surgeries. In particular, $X_{0}$ and $X_{1}$ are bordant. Since signature and Pontryagin numbers are bordism invariants, we have
	$$\sigma\left(X_{1}\right)=\sigma\left(X_{0}\right),\ p_{1}^{2}\left(X_{1}\right)=p_{1}^{2}\left(X_{0}\right).$$
	Moreover, it follows from Lemma \ref{Lemma: char nums additive wrt gluing via boundary}, Statement 2 that 
	\begin{align*}
		\sigma\left(X_{i}\right)&=\sigma\left(W_{i},\partial W_{i}\right)-\sigma\left(W_{0},\partial W_{0}\right),\\
		p_{1}^{2}\left(X_{i}\right)&=\Lambda\left(W_{i}\right)-\Lambda\left(W_{0}\right).
	\end{align*}
	Therefore, $\sigma\left(W_{1},\partial W_{1}\right)=\sigma\left(W_{0},\partial W_{0}\right)=0$ and $\Lambda\left(W_{1}\right)=\Lambda\left(W_{0}\right)=0$. 
	By Propositions \ref{Proposition: identify surgery obstruction}, \ref{Propositioin: theta_0 elementary, equivalent condition} and Lemmata \ref{Lemma: transform abstract charaterization of obstruction into computable char numbers}$\sim$\ref{Lemma: Arf_odd=p_1^2 mod 2}, two $\mathcal{G}^{3}(\mathrm{Wu})$-like manifolds $\left(M_{i},\varphi_{i}\right)$ $(i=0,1)$ are diffeomorphic if and only if there is a bordism $(W,\psi)$ between them, such that
	\begin{compactenum}
		\item $W\xrightarrow{\psi}\mathcal{B}$ is a $4$-equivalence;
		\item $\sigma\left(W,\partial W\right)=0$;
		\item $\Lambda(W)=0\in\mathbb{Z}$.
	\end{compactenum}
	As a consequence, $M_{0}$ and $M_{1}$ are diffeomorphic, thereby proving that invariant $\lambda$ is a diffeomorphism invariant of $\mathcal{G}^{3}(\mathrm{Wu})$-like manifolds.
\end{pf}

\subsection{Classify $\mathcal{G}^{3}_{p}\left(S^{5}\right)$-like manifolds}\label{Subsection: Classify G_3^p S^5 like mfds}
Throughout this subsection, $\mathcal{B}=BSpin\times K\left(\mathbb{Z}\big/p,2\right)$ and $h^{q}$, $h_{q}$ denote the mod $p$ (co)homology. The arrangement of this subsection is parallel to Subsection \ref{Subsection: Classify G_3 Wu like mfds}, in which most conclusions and proofs are similar and the only difference is the coefficient ring. Hence we omit the common arguments and exhibit what is different in the proof.

\begin{lem}\label{Lemma: further identify H_4(B), spin}
	$H_{4}(\mathcal{B})$ admits the unique generator set $\left(\overline{p_{1}}\spcheck,\alpha\right)$ such that
	\begin{compactenum}
		\item $H_{4}(\mathcal{B})=\mathbb{Z}\left\{\overline{p_{1}}\spcheck\right\}\oplus\mathbb{Z}\big/p\left\{\alpha\right\}$;
		\item $\left<\overline{p_{1}},\overline{p_{1}}\spcheck\right>=1$;
		\item $\rho_{p}\left(\overline{p_{1}}\spcheck\right)=\left(\rho_{p}\overline{p_{1}}\right)\spcheck$, $\rho_{p}(\alpha)=\left(\iota_{p}^{2}\right)\spcheck$.
	\end{compactenum}
\end{lem}
To prove this lemma, we notice that $h^{4}(\mathcal{B})=\mathbb{Z}\big/p\left\{\rho_{p}\overline{p_{1}},\iota_{p}^{2}\right\}\cong\left(\mathbb{Z}\big/p\right)^{2}$, and the argument in the proof of Lemma \ref{Lemma: further identify H_4(B)} still applies here. The only two differences are bases of $h^{4}(\mathcal{B})$ and the coefficient rings. As a straightforward corollary, we have
\begin{cor}
	Denote $x_{\psi}:=\widehat{\psi}^{*}\iota_{p}$. Then $H_{4}(\psi)$ is given by
	\begin{equation}\label{Equation: effect of H_4(psi), spin}
		(H_{4}(\psi))(x)=\left<\overline{p_{1}}(W),x\right>\overline{p_{1}}\spcheck+\left<x_{\psi}^{2},\rho_{p}(x)\right>\alpha,\quad\forall x\in H_{4}(W).
	\end{equation}
\end{cor}

\begin{lem}
	$\left(F\left(H_{4}(W)\right),\beta\right)$ is odd.
\end{lem}

\begin{pf}\label{Lemma: (FH_4W, beta) is odd, spin}
	We still apply the idea in the proof of Lemma \ref{Lemma: (FH_4W, beta) is odd}. This time $W$ is spin, and the relative Wu class is given by $\mathrm{U}^{4}(W)=w_{4}(W)$. Hence we obtain
	\begin{align*}
		\left<x\cup x,[W,\partial W]\right>\ \mathrm{mod}\ 2&=\left<\left(\rho_{2}(x)\right)\cup\left(\rho_{2}(x)\right),[W,\partial W]_{2}\right>\\
		&=\left<\left(\rho_{2}(x)\right)\cup\left(w_{4}(W)\right),[W,\partial W]_{2}\right>\\
		&=\left<w_{4}(W),\rho_{2}\left(x\cap [W,\partial W]\right)\right>,\ \forall x\in H^{4}(W,\partial W).
	\end{align*}
	Therefore, $\check{\beta}$ is even if and only if 
	$$\left<w_{4}(W),\rho_{2}(x)\right>=0,\ \forall x\in H_{4}(W).$$
	The long exact sequence of homology groups of $W$ induced by the short exact sequence of coefficient rings $\mathbb{Z}\rightarrowtail\mathbb{Z}\twoheadrightarrow\mathbb{Z}\big/2$ implies that $H_{4}(W)\xrightarrow{\rho_{2}}H_{4}\left(W;\mathbb{Z}\big/2\right)$ is epic, hence $\check{\beta}$ is even if and only if 
	$$\left<w_{4}(W),x\right>=0,\ \forall x\in H_{4}\left(W;\mathbb{Z}\big/2\right).$$
	Meanwhile, by assumption $\mathcal{B}\xrightarrow{\psi}W$ is a $4$-equivalence, hence the long exact sequence of mod $2$ cohomology groups associated to $(\mathcal{B},W)$ implies that $H_{4}\left(\mathcal{B};\mathbb{Z}\big/2\right)\xrightarrow{\psi^{*}}H_{4}\left(W;\mathbb{Z}\big/2\right)$ is monic. Since $p$ is an odd prime, we have $H_{4}\left(\mathcal{B};\mathbb{Z}\big/2\right)\cong H_{4}\left(BSpin;\mathbb{Z}\big/2\right)=\mathbb{Z}\big/2\left\{w_{4}\right\}$, where $w_{4}$ is the universal 4th Stiefel-Whitney class and $w_{4}=\rho_{2}\overline{p_{1}}$. Hence $w_{4}(W)\neq0$ and $\check{\beta}$, $\beta$ must be odd.
\end{pf}

\begin{prop}\label{Propositioin: theta_0 elementary, equivalent condition, spin}
	$\theta_{0}(W,\psi)$ is elementary if and only if the following conditions hold:
	\begin{compactenum}
		\item $\left(\mathcal{V},\beta\right)$ has signature $\sigma(\mathcal{V},\beta)=0$;
		\item $\beta(v,v)=0$;
		\item $g(v)=0$;
		\item $\Xi_{odd}\left(\mathcal{V},\beta,g\right)=0$.
	\end{compactenum}
\end{prop}
Here $\sigma(\mathcal{V},\beta)=0$ forces $\mathrm{rank}\ \mathcal{V}=2r$ to be even, and by Proposition \ref{Proposition: identify surgery obstruction, G_3^p S^5} we have $r\geqslant1$. $\Xi_{odd}$ is defined for any tripple $\left(\mathcal{V},\beta,g\right)$, such that $\mathcal{V}\cong\mathbb{Z}^{2n}$ for some $n$, $\beta$ is an odd symmetric unimodular bilinear form over $\mathcal{V}$ with vanishing signature, and $\mathcal{V}\xrightarrow{g}\mathbb{Z}\big/p$ is a homomorphism. It can be shown that $\Xi_{odd}\left(\mathcal{V},\beta,g\right)=0$ if and only if $\left(\mathcal{V},\beta\right)$ admits a Lagrangian contained in $\ker g$. See Section \ref{Section: Arf inv} for a complete discussion.

The proof is exactly the same as Proposition \ref{Propositioin: theta_0 elementary, equivalent condition}, except that $\mathbb{Z}\big/2$ is replaced by $\mathbb{Z}\big/p$ and the Arf type invariant is replaced by the mod $p$ version. These abstract criterions for $\theta_{0}(W,\psi)$ being elementary can also be transformed into computable characteristic numbers of $W$ as before.

\begin{lem}\label{Lemma: transform abstract charaterization of obstruction into computable char numbers, spin}
	We have $\beta(v,v)=\mathrm{M}(W)$ and $g(v)=\left(x_{\psi}^{2}\rho_{p}\left(\overline{p_{1}}\right)\right)(W)$.
\end{lem}

The result for Arf type invariant $\Xi_{odd}(\mathcal{V},\beta,g)$ is different to Lemma \ref{Lemma: Arf_odd=p_1^2 mod 2}.

\begin{lem}\label{Lemma: Arf_odd=x^2^2, spin}
	$\Xi_{odd}(\mathcal{V},\beta,g)=\left(x_{\psi}^{2}\right)^{2}(W)$.
\end{lem}

\begin{pf}
	Recall that $\mathcal{V}=F\left(H_{4}(W)\right)$ and $\beta$ is the homological intersection form on $\mathcal{V}$ such that $\beta$ is unimodular and $\sigma\left(\mathcal{V},\beta\right)=0$. Let $\left\{e_{i},f_{i}:1\leqslant i\leqslant r\right\}$ be a standard orthonormal basis of $(\mathcal{V},\beta)$, and there are unique classes $\widetilde{e_{i}},\widetilde{f_{i}}\in F\left(H^{4}(W,\partial W)\right)$ such that $e_{i}=\widetilde{e_{i}}\cap[W,\partial W]$, $f_{i}=\widetilde{f_{i}}\cap[W,\partial W]$. Then $\rho_{p}e_{i}=\left(\rho_{p}\widetilde{e_{i}}\right)\cap[W,\partial W]_{p}$ and $\rho_{p}f_{i}=\left(\rho_{p}\widetilde{f_{i}}\right)\cap[W,\partial W]_{p}$. We have computed in the proof of Proposition \ref{Proposition: identify surgery obstruction, G_3^p S^5} that $H_{3}(W)=H_{5}(W)=0$ and $H_{4}(W)\cong\mathbb{Z}^{2r}\oplus\mathbb{Z}\big/p$. Hence by the long exact sequence of homology groups of $W$ associated to the short exact sequence of coefficients $\mathbb{Z}\rightarrowtail\mathbb{Z}\twoheadrightarrow\mathbb{Z}\big/p$ we conclude that $h_{4}(W)\cong\left(\mathbb{Z}\big/p\right)^{2r}$, $\left\{\rho_{p}e_{i},\rho_{p}f_{i}:1\leqslant i\leqslant r\right\}$ is a standard orthogonal basis of $\left(h_{4}(W),\overline{\widetilde{\beta}}\right)$ and $\left\{\rho_{p}e_{i},\rho_{p}f_{i}:1\leqslant i\leqslant r\right\}$ is a standard orthogonal basis of $\left(\overline{V},\overline{\beta}\right)$. Here overline means mod $p$ reduction of groups or bilinear forms.
	
	Set $a_{i}=g\left(e_{i}\right)=\left<x_{\psi}^{2},\rho_{p}e_{i}\right>$ and $b_{i}=g\left(f_{i}\right)=\left<x_{\psi}^{2},\rho_{p}f_{i}\right>$. Then we have
	\begin{align*}
		x_{\psi}^{2}&=j^{*}\left(\sum_{i=1}^{r}a_{i}\rho_{p}\widetilde{e_{i}}+\sum_{i=1}^{r}b_{i}\rho_{p}\widetilde{f_{i}}\right),\\
		\left(x_{\psi}^{2}\right)^{2}(W)&=\sum_{i=1}^{r}a_{i}^{2}-\sum_{i=1}^{r}b_{i}^{2}
	\end{align*}
	and by definition
	\begin{align*}
		\Xi_{odd}(\mathcal{V},\beta,g)&=\sum_{i=1}^{r}g\left(e_{i}\right)^{2}-\sum_{i=1}^{r}g\left(f_{i}\right)^{2}\\
		&=\sum_{i=1}^{r}\left<x_{\psi}^{2}\cup\rho_{p}\left(\widetilde{e_{i}}\right),[W,\partial W]_{p}\right>^{2}-\sum_{i=1}^{r}\left<x_{\psi}^{2}\cup\rho_{p}\left(\widetilde{f_{i}}\right),[W,\partial W]_{p}\right>^{2}\\
		&=\sum_{i=1}^{r}a_{i}^{2}-\sum_{i=1}^{r}b_{i}^{2},
	\end{align*}
	in which the last identity is a direct substitution. Therefore, we have $\Xi_{odd}(\mathcal{V},\beta,g)=\left(x_{\psi}^{2}\right)^{2}(W)$.
\end{pf}

It turns out that $\left(x_{\psi}^{2}\rho_{p}\left(\overline{p_{1}}\right)\right)(W)$ and  $\left(x_{\psi}^{2}\right)^{2}(W)$ are identical when $p=3$. We have the following lemma.
\begin{lem}\label{Lemma: x^2 rho_3 p_1=x^4}
	If $X$ is an oriented closed $8$-manifold or a compact oriented $8$-manifold with $\left.\rho_{3}\left(\overline{p_{1}}(X)\right)\right|_{\partial X}=0$, and $X$ admits a cohomology class $x\in h^{2}(X)=H^{2}\left(X;\mathbb{Z}\big/3\right)$ such that $\left.x^{2}\right|_{\partial X}=0$,
	then $\left(x^{2}\rho_{3}\left(\overline{p_{1}}\right)\right)(X)=\left(x^{2}\right)^{2}(X)$.
\end{lem}

\begin{pf}
	First we discuss the mod $3$ Steenrod power and its dual on manifold computation.
	Let $P^{r}$ denote the mod $3$ Steenrod power. By \cite[Section 1.1, Formula (1)]{Hirzebruch54}, given a closed smooth $m$-manifold $M$, there is a unique family of mod $3$ cohomology classes $s^{r}\in h^{4r}(M)$ such that
	\[P^{r}(v)=s^{r}\cup v,\ \forall v\in h^{m-4r}(M),\]
	due to which we may also refer to $s^{r}$ as the mod $3$ Wu class. And by \cite[Section 1.1, Formula (2)]{Hirzebruch54}, $s^{r}$ is related to the $L$-polynomial of $M$ by
	\[s^{r}=\rho_{3}\left(3^{r}L_{r}(M)\right).\]
	
	Then suppose $W$ is a smooth compact orientable $m$-manifold with non-empty boundary.
	Following the proof of \cite[Lemma 7.3]{Kervaire57}, we can also establish the relative mod $3$ Wu class. There is a unique family of cohomology classes $R^{r}\in h^{4r}(W)$ such that
	\[P^{r}(v)=R^{r}\cup v,\ \forall v \in h^{m-4r}(W,\partial W),\]
	and $R^{r}$ can be constructed explicitly as follows: Form the closed orientable manifold $Y=W\cup_{\partial W}(-W)$, denote by $W\xrightarrow{i}Y$ the inclusion, and $R^{r}(W)=i^{*}s^{r}(Y)$.
	
	Now we return to the $8$-manifold $X$ and begin with the closed case. Take $v=x^{2}$. By Cartan's formula $P^{1}\left(x^{2}\right)=2x^{4}$, and by definition of first spin Pontryagin class $p_{1}(X)=2\overline{p_{1}}(X)$. Hence $x^{4}=\left(x^{2}\right)^{2}=x^{2}\rho_{3}\left(\overline{p_{1}}(X)\right)\in h^{8}(X)$ and the corresponding characteristic numbers are equal.
	
	Next suppose $X$ has non-empty boundary. Since $\left.x^{2}\right|_{\partial X}=0$, we can lift $x^{2}$ to a relative class $\widetilde{x^{2}}\in h^{4}(X,\partial X)$ and $P^{1}\left(\widetilde{x^{2}}\right)=2\left(\widetilde{x^{2}}\cup \widetilde{x^{2}}\right)$. By construction $R^{1}(W)=i^{*}s^{1}(Y)=i^{*}\rho_{3}\left(p_{1}(Y)\right)=2\rho_{3}\left(\overline{p_{1}}(W)\right)$, in which $Y=W\cup_{\partial W}(-W)$. Therefore, we obtain $\left(\widetilde{x^{2}}\cup \widetilde{x^{2}}\right)=\widetilde{x^{2}}\cup\rho_{3}\left(\overline{p_{1}}(W)\right)$, and the corresponding characteristic numbers are equal.
\end{pf}

\begin{pf}[of Theorem \ref{Theorem: main, spin}, Statement 1]
	We have explained in the introduction that $\mu$-invariant is well-defined for $\mathcal{G}^{3}_{p}\left(S^{5}\right)$-like manifolds. The necessity of Theorem \ref{Theorem: main, spin}, Statement 1 can be proved similarly following the conterpart of the proof of Theorem \ref{Theorem: main}, Statement 1. Hence
	it remains to show that two $\mathcal{G}^{3}_{p}\left(S^{5}\right)$-like manifolds are diffeomorphic provided that they have the same $\mu$-invariant.
	
	By Proposition \ref{Propositioin: theta_0 elementary, equivalent condition, spin} and Lemmata \ref{Lemma: transform abstract charaterization of obstruction into computable char numbers, spin}, \ref{Lemma: Arf_odd=x^2^2, spin}, two $\mathcal{G}^{3}_{p}\left(S^{5}\right)$-like manifolds $M_{0}$, $M_{1}$ are diffeomorphic, if and only if there are non-zero classes $x_{i}\in h^{2}\left(M_{i}\right)$, an oriented spin bordism $W$ between $M_{i}$ and a class $x_{W}\in h^{2}(W)$, such that $\sigma(W,\partial W)$, $\mathrm{M}(W)$, $\left(x_{W}^{2}\rho_{p}\left(\overline{p_{1}}\right)\right)(W)$ and $\left(x_{W}^{2}\right)^{2}(W)$ all vanish. Therefore, we define an invariant $s(M,x)$ for a $\mathcal{G}^{3}_{p}\left(S^{5}\right)$-like manifold and its non-zero class $x\in h^{2}(M)$ as follows:
	\begin{compactenum}
		\item Take any $(\mathcal{B},\mathcal{F})$-coboundary $\left(V,\widehat{x}\right)$ of $(M,x)$, namely a spin coboundary $V$ of $M$ and a class $\widehat{x}\in h^{2}(V)$ such that $\left.\widehat{x}\right|_{M}=x$. We further require that $(V,M)$ has vanishing signature
		\item Compute 
		\[S\left(V,\widehat{x}\right):=\left(\mathrm{M}(V),\left(\widehat{x}^{2}\rho_{p}\overline{p_{1}}\right)(V),\left(\widehat{x}^{2}\right)^{2}(V)\right)\in \mathbb{Q}\times\mathbb{Z}\big/p\times\mathbb{Z}\big/3\]
		when $p\geqslant5$, and
		\[S\left(V,\widehat{x}\right):=\left(\mathrm{M}(V),\left(\widehat{x}^{2}\right)^{2}(V)\right)\in \mathbb{Q}\times\mathbb{Z}\big/p\]
		when $p=3$.
		\item Note that $S$ is also defined for closed $8$-dimensional $(\mathcal{B},\mathcal{F})$-manifolds and defines a homomorphism 
		\[S:\Omega_{8}\left(\mathbb{Z}\big/p,2\right)\to\mathbb{Z}\times\mathbb{Z}\big/p\times\mathbb{Z}\big/p\]
		when $p\geqslant5$, and 
		\[S:\Omega_{8}\left(\mathbb{Z}\big/3,2\right)\to\mathbb{Z}\times\mathbb{Z}\big/3\]
		when $p=3$. Let $\Omega_{8}\left(\mathbb{Z}\big/p,2\right)_{0}$ denote the subgroup of closed $8$-manifolds with vanishing signature. Set
		\[s(M,x):=S\left(V,\widehat{x}\right)+S\left(\Omega_{8}\left(\mathbb{Z}\big/p,2\right)_{0}\right)\in\left(\mathbb{Q}\times \mathbb{Z}\big/p\times\mathbb{Z}\big/p\right)\big/S\left(\Omega_{8}\left(\mathbb{Z}\big/p,2\right)_{0}\right)\]
		for $p\geqslant5$ and 
		\[s(M,x):=S\left(V,\widehat{x}\right)+S\left(\Omega_{8}\left(\mathbb{Z}\big/3,2\right)_{0}\right)\in\left(\mathbb{Q}\times \mathbb{Z}\big/3\right)\big/S\left(\Omega_{8}\left(\mathbb{Z}\big/3,2\right)_{0}\right)\]
		when $p=3$.
		Following the argument of proof of Theorem \ref{Theorem: main}, Statement 1, it can be shown that $s(M,x)$ does not depend on choices of $(\mathcal{B},\mathcal{F})$-coboundary, that $s(M,x)$ is a well-defined invariant for the pair $(M,x)$, and that two $\mathcal{G}^{3}_{p}\left(S^{5}\right)$-like manifolds $M_{0}$, $M_{1}$ are diffeomorphic, if and only if there are non-zero classes $x_{i}\in h^{2}\left(M_{i}\right)$ such that $s\left(M_{0},x_{0}\right)=s\left(M_{1},x_{1}\right)$.
	\end{compactenum}
	
	It remains to compute the subgroup $S\left(\Omega_{8}\left(\mathbb{Z}\big/p,2\right)_{0}\right)$ and derive explicit expressions of invariant $s$. From the complete identification results of $\Omega_{8}\left(\mathbb{Z}\big/p,2\right)$ (Propositions \ref{Proposition: Omega_q^Spin(Z/p,2), p>=5} and \ref{Proposition: Omega_q^Spin(Z/3,2)}, Tables \ref{Table: Omega_8^Spin(Z/p,2) to Z^2+(Z/p)^2} and \ref{Table: Omega_8^Spin(Z/3,2) to Z^2+Z/3}) we can easily derive Tables \ref{Table: S(Omega_8^Spin(Z/p,2))_0, p geq 5} and \ref{Table: S(Omega_8^Spin(Z/3,2))_0}, which follows from a straightforward computation of characteristic numbers.
	\begin{table}[H]
		\centering
		\caption{The subgroup $S\left(\Omega_{8}\left(\mathbb{Z}\big/p,2\right)_{0}\right)$, $p\geqslant5$}
		\label{Table: S(Omega_8^Spin(Z/p,2))_0, p geq 5}
		\begin{tabular}{cccc}
			\toprule[1.5pt]
			Basis of $\Omega^{Spin}_{8}\left(\mathbb{Z}\big/p,2\right)_{0}$  & $\mathrm{M}$ & $x^{2}\rho_{p}\overline{p_{1}}$ & $\left(x^{2}\right)^{2}$\\
			\midrule[1pt]
			$Bott-224\mathbb{H}P^{2}$ & $-224$ & $0$ & $0$ \\
			$[V_{5}(2),\rho_{p}\left(\widehat{z_{5}}\right)]-2\mathbb{H}P^{2}$  & $0$ & $2$ & $2$ \\
			$\left[\left(S^{2}\right)^{4},\rho_{p}\left(\Delta\right)\right]$ & $0$ & $0$ & $24$\\
			\bottomrule[1.5pt]
		\end{tabular}
	\end{table}
	\begin{table}[H]
		\centering
		\caption{The subgroup $S\left(\Omega_{8}\left(\mathbb{Z}\big/3,2\right)_{0}\right)$}
		\label{Table: S(Omega_8^Spin(Z/3,2))_0}
		\begin{tabular}{ccc}
			\toprule[1.5pt]
			Basis of $\Omega^{Spin}_{8}\left(\mathbb{Z}\big/3,2\right)_{0}$ & $\mathrm{M}$ & $\left(x^{2}\right)^{2}$ \\
			\midrule[1pt]
			$Bott-2\mathbb{H}P^{2}$ & $-224$ & $0$ \\
			$[V_{5}(2),\rho_{3}\left(\widehat{z_{5}}\right)]-2\mathbb{H}P^{2}$ & $0$ & $2$ \\
			\bottomrule[1.5pt]
		\end{tabular}
	\end{table}
	It follows from Tables \ref{Table: S(Omega_8^Spin(Z/p,2))_0, p geq 5} and \ref{Table: S(Omega_8^Spin(Z/3,2))_0} that 
	\[S\left(\Omega_{8}\left(\mathbb{Z}\big/p,2\right)_{0}\right)=\begin{cases}
		224\mathbb{Z}\times\mathbb{Z}\big/p\times\mathbb{Z}\big/p,\ &p\geqslant5;\\
		224\mathbb{Z}\times\mathbb{Z}\big/3,\ &p=3.
	\end{cases}\]
	Therefore, in both cases the value of $s$-invariant for the $7$-dimensional $(\mathcal{B},\mathcal{F})$-manifold $(M,x)$ is $\mathrm{M}(V)\ \mathrm{mod}\ 224$, is equal to $\mu(M)$ and does not depend on choices of non-zero cohomology classes $x\in h^{2}(M)$. We can simply denote $s(M):=s(M,x)$
	
	It is not sufficient to conclude directly that two $\mathcal{G}^{3}_{p}\left(S^{5}\right)$-like manifolds are diffeomorphic if and only if they have the same $\mu$-invariant, and the reason is as follows. From our previous argument, the invariant $s$ can be computed only if we can find a spin cobundary such that the boundary inclusion induces an epimorphism on the second mod $p$ cohomology groups. While this cohomological condition is not required in the computation of invariant $\mu$. In practice it would be harder to directly construct a spin coboundary with boundary inclusion inducing an epimorphism on the second mod $p$ cohomology groups, comparing with simply constructing a spin coboundary. Hence we must justify that if $V'$ is a coboundary of $\mathcal{G}^{3}_{p}\left(S^{5}\right)$-like manifold $M$ such that $\mu(M)$ can be computed from $V'$, then $s(M)=\mu(M)$. Only then can we conclude that $\mu$-invariant is the complete diffeomorphism invariant of $\mathcal{G}^{3}_{p}\left(S^{5}\right)$-like manifolds.
	
	Let $V'$ be a spin coboundary of $\mathcal{G}^{3}_{p}\left(S^{5}\right)$-like manifold $M$ with $\sigma\left(V',M\right)=0$. Take any non-zero class $x\in h^{2}(M)$ and let $\left(V,\widehat{x}\right)$ be a $(\mathcal{B},\mathcal{F})$-coboundary of $(M,x)$ with $\sigma(V,M)=0$. Form the closed oriented spin $8$-manifold $X=V'\cup_{M}(-V)$, and $\mathrm{M}\left(V'\right)-\mathrm{M}(V)=\mathrm{M}(X)\in224\mathbb{Z}$. Therefore, $s(M)=\mathrm{M}\left(V'\right)\ \mathrm{mod}\ 224=\mathrm{M}\left(V\right)\ \mathrm{mod}\ 224=\mu(M)$. Now we conclude that two $\mathcal{G}^{3}_{p}\left(S^{5}\right)$-like manifolds are diffeomorphic if and only if they have the same $\mu$-invariant.
\end{pf}
	
	\section{Compute the $\lambda$-invariants and $\mu$-invariants}\label{Section: Compute s inv}
In this section we compute the $\lambda$-invariant of $\mathcal{G}^{3}(\mathrm{Wu})$-like manifolds and $\mu$-invariant of $\mathcal{G}^{3}_{p}\left(S^{5}\right)$-like manifolds, thereby completing the proofs of Theorems \ref{Theorem: main} and \ref{Theorem: main, spin}. We consider $\mathcal{G}^{3}(\mathrm{Wu})$, $\mathcal{G}^{3}_{p}\left(S^{5}\right)$ and $M\#\Sigma$, in which $M$ is any $\mathcal{G}^{3}(\mathrm{Wu})$-like or $\mathcal{G}^{3}_{p}\left(S^{5}\right)$-like manifold and $\Sigma$ is a homotopy $7$-sphere. 
With these computational results we completely classify $\mathcal{G}^{3}(\mathrm{Wu})$-like and $\mathcal{G}^{3}_{p}\left(S^{5}\right)$-like manifolds. We also compute the inertia groups of $\mathcal{G}^{3}(\mathrm{Wu})$-like and $\mathcal{G}^{3}_{p}\left(S^{5}\right)$-like manifolds.

We first introduce some notations. 
Let $\mathrm{Tr}$ denote the trace of surgery on $S^{2}\times\mathrm{Wu}$ that produces $\mathcal{G}_{3}\left(\mathrm{Wu}\right)$, and $\mathrm{Tr}$ is a bordism from $S^{2}\times\mathrm{Wu}$ to $\mathcal{G}_{3}\left(\mathrm{Wu}\right)$. Then we glue $D^{3}\times\mathrm{Wu}$ and $\mathrm{Tr}$ along their common boundary component $S^{2}\times\mathrm{Wu}$ via identity, obtaining a coboundary $V$ for $\mathcal{G}_{3}\left(\mathrm{Wu}\right)$. 

Similarly, let $\mathrm{Tr}'$ denote the trace of surgery on $S^{2}\times S^{5}$ that produces $\mathcal{G}^{3}_{p}\left(S^{5}\right)$, and $\mathrm{Tr}'$ is a bordism from $S^{2}\times S^{5}$ to $\mathcal{G}^{3}_{p}\left(S^{5}\right)$. Then we glue $D^{3}\times S^{5}$ and $\mathrm{Tr}'$ along their common boundary component $S^{2}\times S^{5}$ via identity, obtaining a coboundary $V'$ for $\mathcal{G}^{3}_{p}\left(S^{5}\right)$.

We also need information about homotopy $7$-spheres. Up to orientation-preserving diffeomorphism there are exactly $28$ homotopy spheres and can be parametrized as $\Sigma_{r}$ $(0\leqslant r\leqslant27)$ such that $\Sigma_{0}=S^{7}$ is the standard Euclidean $7$-sphere and $\Sigma_{r}$ admits a $3$-connected parallelizable coboundary $W_{r}$ with signature $-8r$. See \cite{KervaireMilnor63}, \cite[Section 6]{EKinv} for original definitions and \cite[Section 5.1]{FarrellSu} for explicit constructions. See also \cite[Remarks 3, 5]{Xu25} for a synthesis.

\begin{lem}\label{Lemma: s-inv of G_3(Wu)}
	\begin{compactenum}
		\item[\ ]
		\item $\lambda\left(\mathcal{G}^{3}(\mathrm{Wu})\right)=0\ \mathrm{mod}\ 7$.
		\item Let $M$ be a $\mathcal{G}^{3}(\mathrm{Wu})$-like manifold. Then for $0\leqslant r\leqslant27$ we have $\lambda\left(M\#\Sigma_{r}\right)=\lambda(M)-3r\ \mathrm{mod}\ 7.$
	\end{compactenum}
\end{lem}

\begin{lem}\label{Lemma: s-inv of G_3^p S^5}
	\begin{compactenum}
		\item[\ ]
		\item $\mu\left(\mathcal{G}^{3}_{p}\left(S^{5}\right)\right)=0\ \mathrm{mod}\ 224$.
		\item Let $M$ be a $\mathcal{G}^{3}_{p}\left(S^{5}\right)$-like manifold. Then for $0\leqslant r\leqslant27$ we have $\mu\left(M\#\Sigma_{r}\right)=\mu(M)+8r\ \mathrm{mod}\ 224.$
	\end{compactenum}
\end{lem}

As direct corollaries we have
\begin{cor}
	\begin{compactenum}
		\item[\ ]
		\item Let $M$ be a $\mathcal{G}^{3}(\mathrm{Wu})$-like manifold. Then its inertia group $I(M)$ is isomorphic to $\mathbb{Z}\big/4$.
		\item A $\mathcal{G}^{3}(\mathrm{Wu})$-like manifold admits an orientation-reversing self-diffeomorphism if and only if it is oriented diffeomorphic to $\mathcal{G}^{3}(\mathrm{Wu})$.
		\item All $\mathcal{G}^{3}(\mathrm{Wu})$-like manifolds are homeomorphic, and $\mathcal{G}^{3}(\mathrm{Wu})$ admits $7$ distinct smooth structure.
	\end{compactenum}
\end{cor}

\begin{cor}
	\begin{compactenum}
		\item[\ ]
		\item Let $M$ be a $\mathcal{G}^{3}_{p}\left(S^{5}\right)$-like manifold. Then its inertia group $I(M)$ is trivial.
		\item A $\mathcal{G}^{3}_{p}\left(S^{5}\right)$-like manifold admits an orientation-reversing self-diffeomorphism if and only if it is diffeomorphic to $\mathcal{G}^{3}_{p}\left(S^{5}\right)$ (not requiring orientation-preserving).
		\item All $\mathcal{G}^{3}_{p}\left(S^{5}\right)$-like manifolds are homeomorphic, and $\mathcal{G}^{3}_{p}\left(S^{5}\right)$ admits $28$ distinct smooth structure.
	\end{compactenum}
\end{cor}

\begin{pf}[of Lemma \ref{Lemma: s-inv of G_3(Wu)}]
	We begin with Statement 1. Recall that
	$$V:=\mathrm{Tr}\cup_{S^{2}\times\mathrm{Wu}\times\{0\}}\left(D^{3}\times\mathrm{Wu}\right)\cong \left(D^{3}\times\mathrm{Wu}\right)\cup_{S^{2}\times D^{5}}\left(D^{3}\times D^{5}\right)$$
	is a coboundary of $\mathcal{G}^{3}(\mathrm{Wu})$. It follows from a standard Mayer-Vietoris argument that
	\[H_{q}(V)=\left\{
	\begin{alignedat}{2}
		\mathbb{Z},\ &q=0, 3, 5;\\
		\mathbb{Z}\big/2,\ &q=2;\\
		0,\ &\text{otherwise}.
	\end{alignedat}
	\right.\]
	Hence $H^{4}\left(V,\mathcal{G}^{3}(\mathrm{Wu})\right)=0$ by Poincar\'{e}-Lefschetz duality and $H^{4}(V)=0$ by universal coefficient theorem. Consequentially, $\sigma\left(V,\mathcal{G}^{3}(\mathrm{Wu})\right)=\Lambda(V)=0$ and $\lambda\left(\mathcal{G}^{3}(\mathrm{Wu})\right)=0\in\mathbb{Z}\big/7$.

	Next we prove Statement 2. Let $V$ be a coboundary of $M$ such that $\sigma(V,M)=0$ and $\lambda(M)=\Lambda(V)\ \mathrm{mod}\ 7$. Form the boundary connected sum $V\natural W_{r}$ of $V$ and $W_{r}$, then construct its connected sum with $8r$ copies of $\mathbb{H}P^{2}$, and we obtain a coboundary
	$$V_{r}:=V\natural W_{r}\#\left(8r\mathbb{H}P^{2}\right)$$
	of $M\#\Sigma_{r}$. Since signature is additive over connected sum with closed manifolds and over boundary connected sum with manifolds with boundary, we have
	\[\sigma\left(V_{r},M\#\Sigma_{r}\right)=\sigma\left(V,M\right)+\sigma\left(W_{r},\Sigma_{r}\right)+\sigma\left(8r\mathbb{H}P^{2}\right)=0-8r+8r=0.\]
	Hence $\lambda\left(M\#\Sigma_{r}\right)=\Lambda(V_{r})\ \mathrm{mod}\ 7$. By Lemma \ref{Lemma: char nums additive wrt gluing via boundary}, Statement 1 we have
	\[\Lambda\left(V_{r}\right)=\Lambda(V)+8r p_{1}^{2}\left(\mathbb{H}P^{2}\right)=\Lambda(V)+32r.\]
	Hence immediately we obtain $\lambda\left(M\#\Sigma_{r}\right)=\Lambda\left(V_{r}\right)\ \mathrm{mod}\ 7=\lambda(M)-3r\ \mathrm{mod}\ 7$.
\end{pf}

\begin{pf}[of Theorem \ref{Theorem: main}, Statements 2 and 3]
	By Lemma \ref{Lemma: s-inv of G_3(Wu)} we have 
	$$\lambda\left(\mathcal{G}^{3}(\mathrm{Wu})\#\Sigma_{r}\right)=-3r\ \mathrm{\mod}\ 7.$$ 
	Since $-3$ and $7$ are coprime, $\lambda\left(\mathcal{G}^{3}(\mathrm{Wu})\#\Sigma_{r}\right)$ can take every value in $\mathbb{Z}\big/7$ as $r$ takes every value from $0$ to $27$. 
	
	Let $\mathcal{M}$ be the set of oriented diffeomorphism classes of $\mathcal{G}^{3}(\mathrm{Wu})$-like manifolds. Theorem \ref{Theorem: main}, Statement 1 implies that $\mathcal{M}\xrightarrow{\lambda}\mathbb{Z}\big/7$ is injective, and the argument above shows that $\mathcal{M}\xrightarrow{\lambda}\mathbb{Z}\big/7$ is also surjective, thereby bijective. In particular, any two $\mathcal{G}^{3}(\mathrm{Wu})$-like manifolds are differed by a connected sum with certain homotopy $7$-sphere, and they must be homeomorphic. Now we complete the proof of Theorem \ref{Theorem: main}.
\end{pf}

\begin{pf}[of Lemma \ref{Lemma: s-inv of G_3^p S^5}]
	The proof is parallet to that of Lemma \ref{Lemma: s-inv of G_3(Wu)}. We begin with Statement 1. Recall that
	$$V':=\mathrm{Tr}'\cup_{S^{2}\times S^{5}\times\{0\}}\left(D^{3}\times S^{5}\right)\cong \left(D^{3}\times S^{5}\right)\cup_{S^{2}\times D^{5}}\left(D^{3}\times D^{5}\right)$$
	is a coboundary of $\mathcal{G}^{3}_{p}\left(S^{5}\right)$. It follows from a standard Mayer-Vietoris argument that
	\[H_{q}\left(V'\right)=\left\{
	\begin{alignedat}{2}
		\mathbb{Z},\ &q=0, 3, 5;\\
		0,\ &\text{otherwise}.
	\end{alignedat}
	\right.\]
	Hence $H^{4}\left(V',\mathcal{G}^{3}_{p}\left(S^{5}\right)\right)=0$ and $H^{4}(V)=0$. As a consequence, $\sigma\left(V',\mathcal{G}^{3}_{p}\left(S^{5}\right)\right)=\mathrm{M}(V)=0$ and $\mu\left(\mathcal{G}^{3}_{p}\left(S^{5}\right)\right)=0\in\mathbb{Z}\big/224$.
	
	Next we prove Statement 2. Let $V$ be a coboundary of $M$ such that $\sigma(V,M)=0$ and $\mu(M)=\mathrm{M}(V)\ \mathrm{mod}\ 224$. Form the  coboundary
	$$V_{r}:=V\natural W_{r}\#\left(8r\mathbb{H}P^{2}\right)$$
	of $M\#\Sigma_{r}$ as before, and we have
	\(\sigma\left(V_{r},M\#\Sigma_{r}\right)=0.\)
	Hence $\mu\left(M\#\Sigma_{r}\right)=\mathrm{M}(V_{r})\ \mathrm{mod}\ 224$. By Lemma \ref{Lemma: char nums additive wrt gluing via boundary}, Statement 1 we have
	\[\mathrm{M}\left(V_{r}\right)=\mathrm{M}(V)+8r \overline{p_{1}}^{2}\left(\mathbb{H}P^{2}\right)=\mathrm{M}(V)+8r.\]
	Hence immediately we obtain $\mu\left(M\#\Sigma_{r}\right)=\mathrm{M}\left(V_{r}\right)\ \mathrm{mod}\ 224=\mu(M)+8r\ \mathrm{mod}\ 224$.
\end{pf}

\begin{pf}[of Theorem \ref{Theorem: main, spin}, Statements 2 and 3]
	By Lemmata \ref{Lemma: s-inv of G_3^p S^5} we have 
	$$\mu\left(\mathcal{G}^{3}_{p}\left(S^{5}\right)\#\Sigma_{r}\right)=8r\ \mathrm{\mod}\ 224.$$ 
	Hence $\mu\left(\mathcal{G}^{3}_{p}\left(S^{5}\right)\#\Sigma_{r}\right)$ can take every value in $8\mathbb{Z}\big/224$ as $r$ takes every value from $0$ to $27$. 
	
	Let $\mathcal{M}'$ be the set of oriented diffeomorphism classes of $\mathcal{G}^{3}_{p}\left(S^{5}\right)$-like manifolds. Theorem \ref{Theorem: main, spin}, Statement 1 implies that $\mathcal{M}'\xrightarrow{\mu}\mathbb{Z}\big/224$ is injective, and the argument above shows that $\mathcal{M}'\xrightarrow{\mu}\mathbb{Z}\big/224$ has the image $8\mathbb{Z}\big/224\cong\mathbb{Z}\big/28$. Hence $\mu$-invariant induces a bijection $\mathcal{M}'\xrightarrow{\cong}\mathbb{Z}\big/28$. In particular, any two $\mathcal{G}^{3}_{p}\left(S^{5}\right)$-like manifolds are differed by a connected sum with certain homotopy $7$-sphere, and they must be homeomorphic. Now we complete the proof of Theorem \ref{Theorem: main, spin}.
\end{pf}
	
	\section{Compute the bordism groups}\label{Section: Compute bordism gps}
In this section we compute the bordism groups $\Omega_{q}^{Spin}\left(\mathbb{Z}\big/p,2\right)$ for odd primes $p$ and $q=7,8$. Recall in the proof of Proposition \ref{Proposition: identify surgery obstruction, G_3^p S^5} that we adapt the following convention as in \cite{Zhubr75}. If $E$ is a (co)homology theory, $G$ is an abelian group and $n$ is a positive integer, then $E\left(K\left(G,n\right)\right)$ would be abbreviated as $E(G,n)$.

First we state our results of bordism groups. For this purpose we introduce some characteristic numbers and closed spin $8$-manifolds with distinguished cohomology classes.
Recall that each element in $\Omega_{q}^{Spin}\left(\mathbb{Z}\big/p,2\right)$ can be represented by a pair $(X,x)$, where $M$ is a closed smooth spin $q$-manifold and $x\in H^{2}\left(X;\mathbb{Z}\big/p\right)$ is a cohomology class. For a closed smooth spin $8$-manifold $X$, its signature $\sigma(X)$ and $\widehat{A}$-genus $\widehat{A}(X)$ are integer spin bordism invariant. If $X$ is assigned with a mod $p$ cohomology class $x\in H^{2}\left(X;\mathbb{Z}\big/p\right)$, we introduce the characteristic numbers
\begin{align*}
	\left(\left(x^{2}\right)^{2}\right)(X,x)&:=\left<\left(x^{2}\right)^{2},[X]_{p}\right>\in\mathbb{Z}\big/p,\\
	\left(x^{2}\rho_{p}\overline{p_{1}}\right)(X,x)&:=\left<x^{2}\rho_{p}\left(\overline{p_{1}}(X)\right),[X]_{p}\right>\in\mathbb{Z}\big/p.
\end{align*}
It is straightforward to verify that these two characteristic numbers are also bordism invariant and are additive over disjoint union, thereby defining homomorphisms from $\Omega_{8}^{Spin}\left(\mathbb{Z}\big/p,2\right)$ to \(\mathbb{Z}\big/p\).

Next we introduce some closed spin $8$-manifolds and their distinguished cohomology classes.
\begin{compactenum}
	\item Let $Bott$ be the following manifold. Plumb 8 copies of the total space unit disc bundle associated to the tangent bundle of \( S^4 \), and we obtain a smooth compact 8-manifold \( W(E_8) \) whose boundary represents a generator of the group of homotopy 7-spheres \( \Theta_7 \cong \mathbb{Z}/28 \). Form the boundary connected sum of 28 copies of \( W(E_8) \), and we obtain a smooth compact 8-manifold whose boundary is now diffeomorphic to \( S^7 \) so that we attach an 8-dimensional unit disc \( D^8 \) along the boundary via a standard diffeomorphism, obtaining the smooth closed manifold Bott. In particular $Bott$ is 3-connected.
	
	\item Let \( z_n \in H^2(\mathbb{C}P^n) \) be the Poincaré dual of the embedding hyperplane \( \mathbb{C}P^{n-1} \). Let \( V_{5}(2) \) be the hypersurface in \( \mathbb{C}P^5 \) of degree 2 and let \( \widehat{z}_5 \) denote the restriction of \( z_5 \) to \( V(2) \).
	
	\item Consider \( (S^2)^4 \), the product of 4 copies of \( S^2 \). Let \( x_i \) be the second cohomology class corresponding to the \( i \)-th component \( S^2 \) and denote \( \Delta = x_1 + x_2 + x_3 + x_4 \).
\end{compactenum}
\begin{prop}\label{Proposition: Omega_q^Spin(Z/p,2), p>=5}
	Let $p\geqslant5$ be an odd prime. We have $\Omega_{7}^{Spin}\left(\mathbb{Z}\big/p,2\right)=0$ and $\Omega_{8}^{Spin}\left(\mathbb{Z}\big/p,2\right)\cong\mathbb{Z}^{2}\oplus\left(\mathbb{Z}\big/p\right)^{2}$. The homomorphism
	\begin{eqnarray}\label{Formula: homom Omega_8^Spin(Z/p,2) to Z^2+(Z/p)^2}
		\begin{aligned}
			\Omega_{8}^{Spin}\left(\mathbb{Z}\big/p,2\right)&\to \mathbb{Z}^{2}\oplus\left(\mathbb{Z}\big/p\right)^{2},\\
			[X,x]&\mapsto\left(\sigma(X),\widehat{A}(X),\left(x^{2}\right)^{2},x^{2}\rho_{p}\overline{p_{1}}\right)
		\end{aligned}
	\end{eqnarray}
	is an isomorphism. A generating set of $\Omega_{8}^{Spin}\left(\mathbb{Z}\big/p,2\right)$ is
	\[\left[\mathbb{H}P^{2},0\right],\left[Bott,0\right],\left[V_{5}(2),\rho_{p}\left(\widehat{z_{5}}\right)\right],\left[\left(S^{2}\right)^{4},\rho_{p}(\Delta)\right].\]
\end{prop}
\begin{prop}\label{Proposition: Omega_q^Spin(Z/3,2)}
	We have $\Omega_{7}^{Spin}\left(\mathbb{Z}\big/3,2\right)=0$ and $\Omega_{8}^{Spin}\left(\mathbb{Z}\big/3,2\right)\cong\mathbb{Z}^{2}\oplus\mathbb{Z}\big/3$. The homomorphism
	\begin{eqnarray}\label{Formula: homom Omega_8^Spin(Z/3,2) to Z^2+Z/3}
		\begin{aligned}
			\Omega_{8}^{Spin}\left(\mathbb{Z}\big/3,2\right)&\to \mathbb{Z}^{2}\oplus\mathbb{Z}\big/3,\\
			[X,x]&\mapsto\left(\sigma(X),\widehat{A}(X),\left(x^{2}\right)^{2}\right)
		\end{aligned}
	\end{eqnarray}
	is an isomorphism. A generating set of $\Omega_{8}^{Spin}\left(\mathbb{Z}\big/3,2\right)$ is
	\[\left[\mathbb{H}P^{2},0\right],\left[Bott,0\right],\left[V_{5}(2),\rho_{p}\left(\widehat{z_{5}}\right)\right].\]
\end{prop}

We will prove these propositions as follows. First we recall some results of $H_{*}(\mathbb{Z}\big/p,2)$. Then we apply Atiyah-Hirzebruch spectral sequence (abbreviated as AHSS). When $p\geqslant5$, AHSS converges at $E^{2}$-page in the range we concern. We can read directly from the $E^{2}$-page that $\Omega_{7}^{Spin}\left(\mathbb{Z}\big/p,2\right)$, and there is an extension problem associated to $\Omega_{8}^{Spin}\left(\mathbb{Z}\big/p,2\right)$, which can be shown as trivial by constructing bordism invariants directly. When $p=3$, there are possibly non-trivial differentials in the $E^{5}$-page of AHSS. Hence we further apply Adams spectral sequence (abbreviated as ASS) to determine the isomorphism type. Finally we directly compute the characteristic numbers mentioned before and give the generating sets.

We begin with partial results of $H_{*}(\mathbb{Z}\big/p,2)$. Their proofs are postponed to the end of this section.
\begin{lem}\label{Lemma: H_{q}(Z/p,2)}
	When $p$ is an odd prime and $q\leqslant9$, the integral homology groups $H_{q}\left(\mathbb{Z}\big/p,2\right)$ are given as follows.
	\begin{table}[H]
		\centering
		\caption{Integral homology of $K\left(\mathbb{Z}\big/p,2\right)$}
		\begin{tabular}{ccccccccccc}
			\toprule[1pt]
			$q$ & $0$ & $1$ & $2$ & $3$ & $4$ & $5$ & $6$ & $7$ & $8$ & $9$\\
			\midrule[0.3pt]
			$H_{q}\left(\mathbb{Z}\big/3,2\right)$ & $\mathbb{Z}$ & $0$ & $\mathbb{Z}\big/3$ & $0$ & $\mathbb{Z}\big/3$ & $0$ & $\mathbb{Z}\big/9$ & $\mathbb{Z}\big/3$ & $\mathbb{Z}\big/3$ & $\mathbb{Z}\big/3$\\
			$H_{q}\left(\mathbb{Z}\big/p,2\right)(p\geqslant5)$ & $\mathbb{Z}$ & $0$ & $\mathbb{Z}\big/p$ & $0$ & $\mathbb{Z}\big/p$ & $0$ & $\mathbb{Z}\big/p$ & $0$ & $\mathbb{Z}\big/p$ & $0$\\
			\bottomrule[1pt]
		\end{tabular}
		\label{Table: H_q(Z/p,2)}
	\end{table}
\end{lem}

\begin{pf}[of Proposition \ref{Proposition: Omega_q^Spin(Z/p,2), p>=5}]
	First we apply AHSS. The $E^{2}$-page is given as in Figure \ref{Figure: E^{2} page of AHSS for Omega^Spin_q(Z/p,2), p geq5}. The terms $E^{2}_{a,b}$ vanish in the line $a+b=7$, hence $\Omega^{Spin}_{7}\left(\mathbb{Z}/p,2\right)=0$. The differentials starting from or ending at the line $a+b=8$ are trivial, hence $E^{2}_{a,b}=E^{\infty}_{a,b}$ for $a+b=8$. Since $\Omega_{q}^{Spin}(-)$ is a generalized homology theory, we obtain 
	\[\Omega^{Spin}_{8}\left(\mathbb{Z}/p,2\right)=\Omega^{Spin}_{8}\oplus\widetilde{\Omega}^{Spin}_{8}\left(\mathbb{Z}/p,2\right),\]
	where $\Omega^{Spin}_{8}=E^{\infty}_{0,8}=E^{2}_{0,8}$ and $\widetilde{\Omega}^{Spin}_{8}\left(\mathbb{Z}/p,2\right)$ fits into the following extension problem:
	\[\mathbb{Z}\big/p\rightarrowtail \widetilde{\Omega}^{Spin}_{8}\left(\mathbb{Z}/p,2\right)\twoheadrightarrow \mathbb{Z}\big/p.\]
	\begin{figure}[h]
		\begin{center}
			\begin{tikzpicture}
				\matrix (m) [matrix of math nodes,
				nodes in empty cells,nodes={minimum width=5ex,
					minimum height=5ex,inner sep=3pt, outer sep=0pt},
				column sep=1ex,row sep=-1ex]{
					b & & & & & & & & & & & \\
					8 & \mathbb{Z}^{2} & 0 & & & & & & & & & \\
					7 & 0 & 0 & 0 & & & & & & & &\\
					6 & 0 & 0 & 0 & 0 & & & & & & &\\
					5 & 0 & 0 & 0 & 0 & 0 & & & & & &\\
					4 & \mathbb{Z} & 0 & \mathbb{Z}\big/p & 0 & \mathbb{Z}\big/p & 0 & & & & &\\
					3 & 0 & 0 & 0 & 0 & 0 & 0 & 0 & & & &\\
					2 & \mathbb{Z}\big/2 & 0 & 0 & 0 & 0 & 0 & 0 & 0 & & & \\
					1 & \mathbb{Z}\big/2 & 0 & 0 & 0 & 0 & 0 & 0 & 0 & 0 & & \\
					0 & \mathbb{Z} & 0 & \mathbb{Z}\big/p & 0 & \mathbb{Z}\big/p & 0 & \mathbb{Z}\big/p & 0 & \mathbb{Z}\big/p & 0 &\\
					\quad\strut & 0 & 1 & 2 & 3 & 4 & 5 & 6 & 7 & 8 & 9 & a\strut \\};
				\draw[thick,<-] (m-1-1.east) -- (m-11-1.east) ;
				\draw[thick,->] (m-11-1.north) -- (m-11-12.north) ;
				\draw [opacity=.4,line width=6mm,line cap=round, color=yellow] 
				(m-2-2.center) -- (m-10-10.center);
				\draw[opacity=.4,line width=6mm,line cap=round, color=orange]
				(m-2-3.center) -- (m-10-11.center);
				\draw[opacity=.4,line width=6mm,line cap=round, color=orange]
				(m-3-2.center) -- (m-10-9.center);
			\end{tikzpicture}
		\end{center}
		\caption{The $E^{2}$-page of Atiyah-Hirzebruch spectral sequence for $\Omega^{Spin}_{q}\left(\mathbb{Z}/p,2\right)$, $p\geqslant5$}
		\label{Figure: E^{2} page of AHSS for Omega^Spin_q(Z/p,2), p geq5}
	\end{figure}
	
	Now we show that $\Omega^{Spin}_{8}\left(\mathbb{Z}/p,2\right)\cong\mathbb{Z}^{2}\oplus\left(\mathbb{Z}\big/p\right)^{2}$ by constructing explicit isomorphisms. It is straightforward to compute the characteristic numbers and obtain Table \ref{Table: Omega_8^Spin(Z/p,2) to Z^2+(Z/p)^2}.
	\begin{table}[H]
		\centering
		\caption{Homomorphism $\Omega^{Spin}_{8}\left(\mathbb{Z}\big/p,2\right)\to\mathbb{Z}^{2}\oplus\left(\mathbb{Z}\big/p\right)^{2}$}
		\label{Table: Omega_8^Spin(Z/p,2) to Z^2+(Z/p)^2}
		\begin{tabular}{ccccc}
			\toprule[1.5pt]
			Basis $[X,x]$ of $\Omega^{Spin}_{8}\left(\mathbb{Z}\big/p,2\right)$ & $\sigma$ & $\widehat{A}$ & $\left(x^{2}\right)^{2}$ & $x^{2}\rho_{p}\overline{p_{1}}$\\
			\midrule[1pt]
			$\mathbb{H}P^{2}$ & $1$ & $0$ & $0$ & $0$ \\
			$Bott$ & $224$ & $-1$ & $0$ & $0$ \\
			$[V_{5}(2),\rho_{p}\left(\widehat{z_{5}}\right)]$ & $2$ & $0$ & $2$ & $2$ \\
			$\left[\left(S^{2}\right)^{4},\rho_{p}\left(\Delta\right)\right]$ & $0$ & $0$ & $24$ & $0$ \\
			\bottomrule[1.5pt]
		\end{tabular}
	\end{table}
	Since $p$ is a prime greater than $3$, we conclude from Table \ref{Table: Omega_8^Spin(Z/p,2) to Z^2+(Z/p)^2} that the homomorphism given in Formula \eqref{Formula: homom Omega_8^Spin(Z/p,2) to Z^2+(Z/p)^2} is an isomorphism. This completes the proof of Proposition \ref{Proposition: Omega_q^Spin(Z/p,2), p>=5}.
\end{pf}

\begin{pf}[of Proposition \ref{Proposition: Omega_q^Spin(Z/3,2)}]
	Again we begin with AHSS. The $E^{2}$-page is given in Figure \ref{Figure: E^{2} page of AHSS for Omega^Spin_q(Z/3,2)}, from which we directly see that until $E^{5}$-page the spectral sequence has the same terms as the $E^{2}$-page, at least for the shown range. There are two possibly non-trivial differentials $d^{5}_{7,0}$ and $d^{5}_{9,0}$, and we consider other approaches to further determine $\Omega^{Spin}_{q}\left(\mathbb{Z}/3,2\right)$ for $q=7,8$.
	\begin{figure}[h]
		\begin{center}
			\begin{tikzpicture}
				\matrix (m) [matrix of math nodes,
				nodes in empty cells,nodes={minimum width=5ex,
					minimum height=5ex,inner sep=3pt, outer sep=0pt},
				column sep=1ex,row sep=-1ex]{
					b & & & & & & & & & & & \\
					8 & \mathbb{Z}^{2} & 0 & & & & & & & & & \\
					7 & 0 & 0 & 0 & & & & & & & &\\
					6 & 0 & 0 & 0 & 0 & & & & & & &\\
					5 & 0 & 0 & 0 & 0 & 0 & & & & & &\\
					4 & \mathbb{Z} & 0 & \mathbb{Z}\big/3 & 0 & \mathbb{Z}\big/3 & 0 & & & & &\\
					3 & 0 & 0 & 0 & 0 & 0 & 0 & 0 & & & &\\
					2 & \mathbb{Z}\big/2 & 0 & 0 & 0 & 0 & 0 & 0 & 0 & & & \\
					1 & \mathbb{Z}\big/2 & 0 & 0 & 0 & 0 & 0 & 0 & 0 & 0 & & \\
					0 & \mathbb{Z} & 0 & \mathbb{Z}\big/3 & 0 & \mathbb{Z}\big/3 & 0 & \mathbb{Z}\big/9 & \mathbb{Z}\big/3 & \mathbb{Z}\big/3 & \mathbb{Z}\big/3 &\\
					\quad\strut & 0 & 1 & 2 & 3 & 4 & 5 & 6 & 7 & 8 & 9 & a\strut \\};
				\draw[thick,<-] (m-1-1.east) -- (m-11-1.east) ;
				\draw[thick,->] (m-11-1.north) -- (m-11-12.north) ;
				\draw [opacity=.4,line width=6mm,line cap=round, color=yellow] 
				(m-2-2.center) -- (m-10-10.center);
				\draw [opacity=.4,line width=6mm,line cap=round, color=yellow] 
				(m-4-2.center) -- (m-10-8.center);
				\draw[opacity=.4,line width=6mm,line cap=round, color=orange]
				(m-2-3.center) -- (m-10-11.center);
				\draw[opacity=.4,line width=6mm,line cap=round, color=orange]
				(m-3-2.center) -- (m-10-9.center);
			\end{tikzpicture}
		\end{center}
		\caption{The $E^{2}$-page of Atiyah-Hirzebruch spectral sequence for $\Omega^{Spin}_{q}\left(\mathbb{Z}/3,2\right)$}
		\label{Figure: E^{2} page of AHSS for Omega^Spin_q(Z/3,2)}
	\end{figure}
	
	First we claim that \(\Omega_{6}^{Spin}\left(\mathbb{Z}\big/3,2\right)\cong\mathbb{Z}\big/9\), $d^{5}_{7,0}$ is an isomorphism and consequentially \(\Omega_{7}^{Spin}\left(\mathbb{Z}\big/3,2\right)=0\). By \cite[Theorem 5]{Zhubr75} we have an injection 
	\[\Omega_{6}^{Spin}\left(\mathbb{Z}\big/3,2\right)\rightarrowtail
	\mathbb{Z}\big/3\oplus H_{6}\left(\mathbb{Z}\big/3,2\right)\cong
	\mathbb{Z}\big/3\oplus\mathbb{Z}\big/9,\]
	and its image coincides with the kernel of homomorphism
	\begin{align*}
		\chi:\mathbb{Z}\big/3\oplus E^{3}_{6,0}&\to \mathbb{Z}\big/3\otimes\mathbb{Z}\big/6,\\
		\left(g,g'\right)&\mapsto\rho_{6}(g)-\beta\left(g'\right).
	\end{align*}
	Here $\rho_{6}$ is given by
	\[\mathbb{Z}\big/3\cong\mathbb{Z}\big/3\otimes\mathbb{Z}\xrightarrow{\mathrm{id}\otimes\rho_{6}}\mathbb{Z}\big/3\otimes\mathbb{Z}\big/6\cong\mathbb{Z}\big/3\]
	and is an isomorphism. $E^{3}_{6,0}=E^{2}_{6,0}=H_{6}\left(\mathbb{Z}\big/3,2\right)\cong\mathbb{Z}\big/9$ and $\mathbb{Z}\big/9\xrightarrow{\beta}\mathbb{Z}\big/3$ is either trivial or epic. 
	If \(\beta=0\), then \(\ker\chi=E^{3}_{6,0}\cong\mathbb{Z}\big/9\). Suppose \(\beta\) is epic. Let $g$, $h$ and $g'$ be the generators of domain \(\mathbb{Z}\big/3\), domain \(\mathbb{Z}\big/9\) and target \(\mathbb{Z}\big/9\) respectively such that$\rho_{6}(g)=g'$. Then $\beta(h)=\varepsilon g'$ for some $\varepsilon=\pm1$, and it is straightforward to verify that \(\ker\chi\) is a cyclic subgroup of order $9$ generated by \((g,\varepsilon h)\). Therefore, under either hypothesis we have \(\Omega_{6}^{Spin}\left(\mathbb{Z}\big/3,2\right)\cong\ker\chi\cong\mathbb{Z}\big/9\), and $d^{5}_{7,0}$ is an isomorphism. As a consequence, $E^{6}_{7,0}=0$ and $\Omega_{7}^{Spin}\left(\mathbb{Z}\big/3,2\right)=0$.
	
	Next we apply ASS to compute $\Omega_{8}^{Spin}\left(\mathbb{Z}\big/3,2\right)$. It follows from the result of AHSS that $\Omega_{8}^{Spin}\left(\mathbb{Z}\big/3,2\right)$ contains only $3$-torsion and possibly $9$-torsion. Hence it suffices to consider the $3$-local part, and the ASS reads
	\[E^{s,t}_{2}=\mathrm{Ext}^{s,t}_{\mathcal{A}_{3}}\left(H^{*}\left(MSpin\wedge\left(K\left(\mathbb{Z}\big/3,2\right)\right)_{+};\mathbb{Z}\big/3\right),\mathbb{Z}\big/3\right)\Longrightarrow\left(\Omega^{Spin}_{t-s}\left(\mathbb{Z}\big/3,2\right)\right)\sphat_{3}.\]
	Here \(\mathcal{A}_{3}\) is the mod $3$ Steenrod algebra. By K\"{u}nneth formula we have 
	\[H^{*}\left(MSpin\wedge\left(K\left(\mathbb{Z}\big/3,2\right)\right)_{+};\mathbb{Z}\big/3\right)\cong H^{*}\left(MSpin;\mathbb{Z}\big/3\right)\otimes_{\mathbb{Z}/3}H^{*}\left(\mathbb{Z}\big/3,2;\mathbb{Z}\big/3\right).\]
	
	We begin with the first tensor component.
	When localized at $3$ we have the decomposition of spectra
	\[MSpin_{(3)}\simeq MSO_{(3)}\simeq BP\bigvee \Sigma^{8}BP\bigvee\cdots.\]
	Here $BP$ is the Brown-Peterson spectrum and its mod $3$ cohomology is given by
	\[H^{*}\left(BP;\mathbb{Z}\big/3\right)\cong\mathcal{A}_{3}\big/(\beta),\]
	where $\beta$ is the mod $3$ Bockstein and $(\beta)$ is the two-sided ideal of $\mathcal{A}_{3}$ generated by $\beta$ (\cite[p.~147]{WanWang19ASS}). 
	
	For further computation of $\mathrm{Ext}$-groups we also give the minimal free resolution of \(\mathcal{A}_{3}\big/(\beta)\) up to the desired range, which is shown in Figure \ref{Figure: Free resolution ofA_3/(beta)}.
	\begin{figure}[H]
		\begin{center}
			\[\xymatrix@C=1.5ex@R=4ex{
				\vdots\ar[d] & & & & & & \\
				P_{4}\ar[d]\ar@{=}[r] & \Sigma^{4}\mathcal{A}_{3}\ar[d]^{\beta} & + & \Sigma^{12}\mathcal{A}_{3}\ar[d]^{\beta} & + & \Sigma^{8}\mathcal{A}_{3}\ar[d]^{\beta} & + P_{4}^{\geqslant14}\\
				P_{3}\ar[d]\ar@{=}[r] & \Sigma^{3}\mathcal{A}_{3}\ar[d]^{\beta} & + & \Sigma^{11}\mathcal{A}_{3}\ar[d]^{\beta} & + & \Sigma^{7}\mathcal{A}_{3}\ar[d]^{\beta} & + P_{3}^{\geqslant14}\\
				P_{2}\ar[d]\ar@{=}[r] & \Sigma^{2}\mathcal{A}_{3}\ar[rd]^{\beta} & + & \Sigma^{10}\mathcal{A}_{3}\ar[ld]^{\beta P^{2}}\ar[rd]^{-\beta P^{1}} & + & \Sigma^{6}\mathcal{A}_{3}\ar[ld]^{\beta} & + P_{2}^{\geqslant14}\\
				P_{1}\ar[d]\ar@{=}[r] & & \Sigma\mathcal{A}_{3}\ar[rd]^{\beta} & +  & \Sigma^{5}\mathcal{A}_{3}\ar[ld]^{\beta P^{1}} & & + P_{1}^{\geqslant14}\\
				P_{0}\ar@{=}[r] & & & \mathcal{A}_{3} & & & + P_{0}^{\geqslant14}
			}\]
		\end{center}
		\caption{The minimal free resolution of $\mathcal{A}_{3}\big/(\beta)$}
		\label{Figure: Free resolution ofA_3/(beta)}
	\end{figure}
	
	Next we consider the second tensor component. Recall from \cite[\S4.L, p.~500]{HatcherAT} that for an odd prime $p$,
	\[H^{*}\left(\mathbb{Z}\big/p,2;\mathbb{Z}\big/p\right)=\mathbb{Z}\big/p\left[x_{2},x_{2p^{n}+2}:n\geqslant1\right]\otimes_{\mathbb{Z}/p}\Lambda_{\mathbb{Z}/p}\left(y_{3},y_{2p^{n}+1}:n\geqslant1\right).\]
	Here
	\begin{compactenum}
		\item subscrpits indicate the degrees;
		\item $x_{2}\in H^{2}\left(\mathbb{Z}\big/p,2;\mathbb{Z}\big/p\right)\cong\left[K\left(\mathbb{Z}\big/p,2\right),K\left(\mathbb{Z}\big/p,2\right)\right]$ corresponds to the identity map of $K\left(\mathbb{Z}\big/p,2\right)$;
		\item $y_{3}=\beta x_{2}$, where $H^{q}\left(-;\mathbb{Z}\big/p\right)\xrightarrow{\beta}H^{q+1}\left(-;\mathbb{Z}\big/p\right)$ denotes the mod $p$ Bockstein homomorphism;
		\item for $n\geqslant1$, $y_{2p^{n}+1}=P^{p^{n-1}}P^{p^{n-2}}\cdots P^{p}P^{1}\beta x_{2}$, where $H^{q}\left(-;\mathbb{Z}\big/p\right)\xrightarrow{P^{k}}H^{q+2k(p-1)}\left(-;\mathbb{Z}\big/p\right)$ is the mod $p$ Steenrod power;
		\item for $n\geqslant1$, $x_{2p^{n}+2}=\beta y_{2p^{n}+1}$.
	\end{compactenum}
	For the moment we set $p=3$, and the explicit expressions of generators indicate how $\mathcal{A}_{3}$ acts on them.

	Now the Adams chart of $\mathrm{Ext}^{s,t}_{\mathcal{A}_{3}}\left(\mathcal{A}_{3}\big/(\beta)\otimes_{\mathbb{Z}/3}H^{*}\left(\mathbb{Z}\big/3,2;\mathbb{Z}\big/3\right),\mathbb{Z}\big/3\right)$ up to the range $t-s\leqslant9$ is computed as in Figure \ref{Figure: ASS, A_3/(beta) otimes L, E_2 page}. Here a single dot means a $\mathbb{Z}\big/3$-summand, line segments imply non-trivial extensions, and an infinite ray indicates a $\mathbb{Z}_{(3)}$-summand.
	\begin{figure}[H]
		\begin{center}
			\begin{tikzpicture}[xscale=0.8, yscale=0.5]
				\draw[thick,->] (-1.5,-1) -- (10,-1) node[right] {\small $t - s$};
				\draw[thick,->] (-1,-1.5) -- (-1,9) node[above] {\small $s$};
				
				\draw[thick] (-0,0) -- (-0,8);
				
				\draw[thick] (4,1) -- (4,8);
				\filldraw[black] (4,0) circle (0.6pt);
				
				\draw[thick] (6,0) -- (6,8);
				
				\draw[thick] (7,0) -- (7,8);
				
				\draw[thick] (8,2) -- (8,8);
				\filldraw[black] (8,0) circle (0.6pt);				
				
				\foreach \y in {0,1,2,3,4,5,6,7,8} {
					\draw (-1.2,\y) -- (-1,\y);  
					\node[left] at (-1.3, \y) {\small \y}; 
				}
				
				\foreach \x in {0,1,2,3,4,5,6,7,8,9} {
					\draw (\x,-1.2) -- (\x,-1);  
					\node[below] at (\x, -1.3) {\small \x}; 
				}					
			\end{tikzpicture}
		\end{center}
		\caption{$E_{2}$-page of Adams chart for $\mathrm{Ext}^{s,t}_{\mathcal{A}_{3}}\left(\mathcal{A}_{3}\big/(\beta)\otimes_{\mathbb{Z}/3}H^{*}\left(\mathbb{Z}\big/3,2;\mathbb{Z}\big/3\right)\mathbb{Z}\big/3\right)$, $t-s\leqslant9$}
		\label{Figure: ASS, A_3/(beta) otimes L, E_2 page}
	\end{figure}
	Accordingly, the Adams chart of $\mathrm{Ext}^{s,t}_{\mathcal{A}_{3}}\left(MSpin\wedge\left(K\left(\mathbb{Z}\big/3,2\right)_{+}\right),\mathbb{Z}\big/3\right)$ up to the range $t-s\leqslant9$ is given as in Figure \ref{Figure: ASS, compute Omega^Spin_8(Z/3,2)}. 
	\begin{figure}[H]
		\begin{center}
			\begin{tikzpicture}[xscale=0.8, yscale=0.5]
				\draw[thick,->] (-1.5,-1) -- (10,-1) node[right] {\small $t - s$};
				\draw[thick,->] (-1,-1.5) -- (-1,9) node[above] {\small $s$};
				
				\draw[thick] (-0,0) -- (-0,8);
				
				\draw[thick] (4,1) -- (4,8);
				\filldraw[black] (4,0) circle (0.6pt);
				
				\draw[thick] (6,0) -- (6,8);
				
				\draw[thick] (7,0) -- (7,8);
				
				\draw[thick] (7.9,2) -- (7.9,8);
				\filldraw[black] (7.9,0) circle (0.6pt);	
				\draw[thick] (8.1,0) -- (8.1,8);			
				
				\foreach \y in {0,1,2,3,4,5,6,7,8} {
					\draw (-1.2,\y) -- (-1,\y);  
					\node[left] at (-1.3, \y) {\small \y}; 
				}
				
				\foreach \x in {0,1,2,3,4,5,6,7,8,9} {
					\draw (\x,-1.2) -- (\x,-1);  
					\node[below] at (\x, -1.3) {\small \x}; 
				}					
				
				\draw[-stealth,color=red,line width=0.2mm] (7,0) -- (6,2);
				\draw[-stealth,color=red,line width=0.2mm] (7,1) -- (6,3);
				\draw[-stealth,color=red,line width=0.2mm] (7,2) -- (6,4);
				\draw[-stealth,color=red,line width=0.2mm] (7,3) -- (6,5);
				\draw[-stealth,color=red,line width=0.2mm] (7,4) -- (6,6);
				\draw[-stealth,color=red,line width=0.2mm] (7,5) -- (6,7);
				\draw[-stealth,color=red,line width=0.2mm] (7,6) -- (6,8);
			\end{tikzpicture}
		\end{center}
		\caption{Adams chart of $\mathrm{Ext}^{s,t}_{\mathcal{A}_{3}}\left(\mathcal{A}_{3}\big/(\beta)\otimes_{\mathbb{Z}/3}H^{*}\left(\mathbb{Z}\big/3,2;\mathbb{Z}\big/3\right)\mathbb{Z}\big/3\right)$, $t-s\leqslant9$}
		\label{Figure: ASS, compute Omega^Spin_8(Z/3,2)}
	\end{figure}
	Here we have non-trivial differentials exhibited as the red arrows, which follows from the known results that $\Omega_{6}^{Spin}\left(\mathbb{Z}\big/3,2\right)\cong\mathbb{Z}\big/9$ and $\Omega_{7}^{Spin}\left(\mathbb{Z}\big/3,2\right)=0$. Hence we obtain
	\(\left(\Omega_{8}^{Spin}\left(\mathbb{Z}\big/3,2\right)\right)\sphat_{3}\cong\left(\mathbb{Z}_{(3)}\right)^{2}\oplus\mathbb{Z}\big/3\),
	and \(\Omega_{8}^{Spin}\left(\mathbb{Z}\big/3,2\right)\cong\mathbb{Z}^{2}\oplus\mathbb{Z}\big/3\).
	
	Finally, after a straightforward computation of characteristic numbers we obtain Table \ref{Table: Omega_8^Spin(Z/3,2) to Z^2+Z/3}.
	\begin{table}[H]
		\centering
		\caption{Homomorphism $\Omega^{Spin}_{8}\left(\mathbb{Z}\big/3,2\right)\to\mathbb{Z}^{2}\oplus\mathbb{Z}\big/3$}
		\label{Table: Omega_8^Spin(Z/3,2) to Z^2+Z/3}
		\begin{tabular}{cccc}
			\toprule[1.5pt]
			Basis $[X,x]$ of $\Omega^{Spin}_{8}\left(\mathbb{Z}\big/3,2\right)$ & $\sigma$ & $\widehat{A}$ & $\left(x^{2}\right)^{2}$ \\
			\midrule[1pt]
			$\mathbb{H}P^{2}$ & $1$ & $0$ & $0$ \\
			$Bott$ & $224$ & $-1$ & $0$ \\
			$[V_{5}(2),\rho_{3}\left(\widehat{z_{5}}\right)]$ & $2$ & $0$ & $2$ \\
			\bottomrule[1.5pt]
		\end{tabular}
	\end{table}
	Table \ref{Table: Omega_8^Spin(Z/3,2) to Z^2+Z/3} implies that these characteristic numbers establish an explicit isomorphism $\Omega^{Spin}_{8}\left(\mathbb{Z}\big/3,2\right)\cong\mathbb{Z}^{2}\oplus\mathbb{Z}\big/3$ and detect a generating set of $\Omega^{Spin}_{8}\left(\mathbb{Z}\big/3,2\right)$. This completes the proof of Proposition \ref{Proposition: Omega_q^Spin(Z/3,2)}.
\end{pf}

\begin{pf}[of Lemma \ref{Lemma: H_{q}(Z/p,2)}]
	When $q\leqslant6$ the results already follows from \cite[\S 21,22]{EM54}, and it remains to compute the groups for $q=7, 8, 9$. 
	The arguments are different when $p=3$ and $p\geqslant5$.
	
	\noindent\textbf{The case $p\geqslant5$.} We have \(H^{q}\left(\mathbb{Z}\big/p,2;\mathbb{Z}\big/p\right)\cong\mathbb{Z}\big/p\) for \(2\leqslant q\leqslant 9\), and it can be deduced from universal coefficient theorem that \(H_{q}\left(\mathbb{Z}\big/p,2\right)=0\) for \(q=7,9\) and \(H^{9}\left(\mathbb{Z}\big/p,2\right)\cong \mathrm{Ext}\left(H_{8}\left(\mathbb{Z}\big/p,2\right),\mathbb{Z}\right)\cong\mathbb{Z}\big/p^{e}\) for some positive integer \(e\).
	
	Next we determine the value of $e$. The short exact sequence $\mathbb{Z}\rightarrowtail\mathbb{Z}\twoheadrightarrow\mathbb{Z}\big/p$ induces a fibration \(K(\mathbb{Z},2)\to K\left(\mathbb{Z}\big/p,2\right)\to K(\mathbb{Z},3)\). Let us consider the associated cohomological Leray-Serre spectral sequence (abbreviated as LSSS). For this purpose we need the cohomology groups of \(K(\mathbb{Z},3)\). They are given as in Table \ref{Table: H_q(Z,3)}, which contains the results for $q\geqslant3$ since \(K(\mathbb{Z},3)\) is $2$-connected. The first row can be found in \cite[\S4.3, p.~404]{HatcherAT} and the second row is the result of universal coefficient theorem. Accordingly, we can compute the $E_{2}$-page of LSSS. Since \(H^{*}(\mathbb{Z},2)=\mathbb{Z}[z]\) with \(\deg z=2\), the differentials in $E_{2}$-page either start from or end at zero groups, and $E_{3}$-page admits identical terms as $E_{2}$-page but with possibly non-trivial differentials. See Figure \ref{Figure: E_{3} page of LSSS ass to K(Z,2) to K(Z/p,2) to K(Z,3)}.
	\begin{table}[h]
		\centering
		\caption{Integral homology and cohomology of $K\left(\mathbb{Z},3\right)$}
		\begin{tabular}{ccccccccc}
			\toprule[1pt]
			$q$ & $3$ & $4$ & $5$ & $6$ & $7$ & $8$ & $9$ & $10$ \\
			\midrule[0.3pt]
			$H_{q}\left(\mathbb{Z},3\right)$ & $\mathbb{Z}$ & $0$ & $\mathbb{Z}\big/2$ & $0$ & $\mathbb{Z}\big/3$ & $\mathbb{Z}\big/2$ & $\mathbb{Z}\big/2$ & $\mathbb{Z}\big/3$ \\
			$H^{q}\left(\mathbb{Z},3\right)$ & $\mathbb{Z}$ & $0$ & $0$ & $\mathbb{Z}\big/2$ & $0$ & $\mathbb{Z}\big/3$ & $\mathbb{Z}\big/2$ & $\mathbb{Z}\big/2$ \\
			\bottomrule[1pt]
		\end{tabular}
		\label{Table: H_q(Z,3)}
	\end{table}
	
	\begin{figure}[H]
		\begin{center}
			\begin{tikzpicture}
				\matrix (m) [matrix of math nodes,
				nodes in empty cells,nodes={minimum width=7ex,
					minimum height=5ex,inner sep=3pt, outer sep=0pt},
				column sep=1ex,row sep=-1ex]{
					b & & & & & & & & & & & &\\
					9 & 0 & 0 & & & & & & & & & &\\
					8 & \mathbb{Z}\left\{z^{4}\right\} & 0 & 0 & & & & & & & & &\\
					7 & 0 & 0 & 0 & 0 & & & & & & & &\\
					6 & \mathbb{Z}\left\{z^{3}\right\} & 0 & 0 & \mathbb{Z}\left\{z^{3}\iota\right\} & 0 & & & & & & &\\
					5 & 0 & 0 & 0 & 0 & 0 & 0 & & & & & &\\
					4 & \mathbb{Z}\left\{z^{2}\right\} & 0 & 0 & \mathbb{Z}\left\{z^{2}\iota\right\} & 0 & 0 & \mathbb{Z}\big/2 & & & & &\\
					3 & 0 & 0 & 0 & 0 & 0 & 0 & 0 & 0 & & & &\\
					2 & \mathbb{Z}\left\{z\right\} & 0 & 0 & \mathbb{Z}\left\{z\iota\right\} & 0 & 0 & \mathbb{Z}\big/2 & 0 & \mathbb{Z}\big/3 & & &\\
					1 & 0 & 0 & 0 & 0 & 0 & 0 & 0 & 0 & 0 & 0 & &\\
					0 & \mathbb{Z} & 0 & 0 & \mathbb{Z}\left\{\iota\right\} & 0 & 0 & \mathbb{Z}\big/2 & 0 & \mathbb{Z}\big/3 & \mathbb{Z}\big/2 & \mathbb{Z}\big/2 &\\
					\quad\strut & 0 & 1 & 2 & 3 & 4 & 5 & 6 & 7 & 8 & 9 & 10 & a\strut \\};
				\draw[-stealth,color=blue,line width=0.3mm] (m-9-2) -- (m-11-5)
				node[midway,above,yshift=-4pt]{\scriptsize \textcolor{blue}{$d_{3}$}};
				\draw[-stealth,color=blue,line width=0.3mm] (m-7-2) -- (m-9-5)
				node[midway,above,yshift=-4pt]{\scriptsize \textcolor{blue}{$d_{3}$}};
				\draw[-stealth,color=blue,line width=0.3mm] (m-5-2) -- (m-7-5)
				node[midway,above,yshift=-4pt]{\scriptsize \textcolor{blue}{$d_{3}$}};
				\draw[-stealth,color=blue,line width=0.3mm] (m-3-2) -- (m-5-5)
				node[midway,above,yshift=-4pt]{\scriptsize \textcolor{blue}{$d_{3}$}};
				
				\draw[-stealth,color=blue,line width=0.3mm] (m-9-5) -- (m-11-8)
				node[midway,above,yshift=-4pt]{\scriptsize \textcolor{blue}{$d_{3}$}};
				\draw[-stealth,color=blue,line width=0.3mm] (m-7-5) -- (m-9-8)
				node[midway,above,yshift=-4pt]{\scriptsize \textcolor{blue}{$d_{3}$}};
				\draw[-stealth,color=blue,line width=0.3mm] (m-5-5) -- (m-7-8)
				node[midway,above,yshift=-4pt]{\scriptsize \textcolor{blue}{$d_{3}$}};
				
				\draw[-stealth,color=blue,line width=0.3mm] (m-9-8) -- (m-11-11)
				node[midway,above,yshift=-4pt]{\scriptsize \textcolor{blue}{$d_{3}$}};
				
				\draw[thick,<-] (m-1-1.east) -- (m-12-1.east) ;
				\draw[thick,->] (m-12-1.north) -- (m-12-13.north) ;
				\draw [opacity=.4,line width=6mm,line cap=round, color=yellow] 
				(m-2-2.center) -- (m-11-11.center);
				\draw[opacity=.4,line width=6mm,line cap=round, color=orange]
				(m-2-3.center) -- (m-11-12.center);
				\draw[opacity=.4,line width=6mm,line cap=round, color=orange]
				(m-3-2.center) -- (m-11-10.center);
			\end{tikzpicture}
		\end{center}
		\caption{The $E_{3}$-page of Leray-Serre spectral sequence associated to \(K(\mathbb{Z},2)\to K\left(\mathbb{Z}\big/p,2\right)\to K(\mathbb{Z},3)\)}
		\label{Figure: E_{3} page of LSSS ass to K(Z,2) to K(Z/p,2) to K(Z,3)}
	\end{figure}
	
	Here \(H^{3}(\mathbb{Z},3)\cong\mathbb{Z}\) admits a distinguished generator $\iota$ corresponding to the identity map of \(K(\mathbb{Z},3)\). Since there are no differentials starting from or ending at \(E_{r}^{3,0}\), we must have \(E_{4}^{3,0}=E_{\infty}^{3,0}\) for such pairs. While for all remaining terms \(E_{3}^{a,b}=0\) when $a+b=3$, we obtain further that \(E_{4}^{3,0}=H^{3}\left(\mathbb{Z}\big/p,2\right)\cong\mathbb{Z}\big/p\).
	Therefore, we can change the sign of $z$ if necessary and obtain $d_{3}z=p\iota$, $d_{3}\left(z^{k}\right)=kz^{k-1}\iota$. Since $p$ is odd, we have \(\mathrm{im}d_{3}^{0,8}=4\mathbb{Z}\left\{z^{3}\iota\right\}\). The differential $d_{3}^{3,6}$ has domain $\mathbb{Z}\left\{z^{3}\iota\right\}$ codomain $\mathbb{Z}\big/2$, hence its kernel is generated by either $z^{3}\iota$ or $2z^{3}\iota$. As a consequence, $E_{4}^{3,6}\cong\mathbb{Z}\big/4p$ or $\mathbb{Z}\big/2p$. 
	
	There is another possibly non-vanishing term $E_{4}^{9,0}$ among the terms $E_{4}^{a,b}$ with $a+b=9$. It follows from the $E_{3}$-page that $E_{4}^{9,0}$ is a quotient of $\mathbb{Z}\big/2$. Now we conclude from LSSS that $H^{9}\left(\mathbb{Z}\big/p,2\right)$ is an extension of sub-quotients of $\mathbb{Z}\big/4p$ and $\mathbb{Z}\big/2$. However, we have seen $H^{9}\left(\mathbb{Z}\big/p,2\right)\cong\mathbb{Z}\big/p^{e}$ for some positive integer $e$. Therefore, we must have $e=1$.
	
	\noindent\textbf{The case $p=3$.} Now we read from the ring structure that
	\[H^{q}\left(\mathbb{Z}\big/3,2;\mathbb{Z}\big/3\right)\cong\begin{cases}
		\mathbb{Z}\big/3,\ &2\leqslant q\leqslant6;\\
		\left(\mathbb{Z}\big/3\right)^{2},\ &7\leqslant q\leqslant9.
	\end{cases}\]
	Hence by universal coefficient theorem \(H_{q}\left(\mathbb{Z}\big/3,2\right)\cong\mathbb{Z}\big/3^{e_{q}}\) for some positive integers \(e_{q}\) when \(7\leqslant q\leqslant9\). 
	
	Just as the case $p\geqslant5$, we apply cohomological LSSS associated to the fibration \(K(\mathbb{Z},2)\to K\left(\mathbb{Z}\big/3,2\right)\to K(\mathbb{Z},3)\) to determine \(e_{q}\) for \(7\leqslant q\leqslant9\). Notice that Figure \ref{Figure: E_{3} page of LSSS ass to K(Z,2) to K(Z/p,2) to K(Z,3)} still applies to the case $p=3$. After analyzing the $E_{4}$-page of LSSS we obtain that
	\begin{compactenum}
		\item $H^{8}\left(\mathbb{Z}\big/3,2\right)$ is an extension of sub-quotients of $\mathbb{Z}\big/2$ and $\mathbb{Z}\big/3$;
		\item $H^{9}\left(\mathbb{Z}\big/3,2\right)$ is an extension of sub-quotients of $\mathbb{Z}\big/12$ (or $\mathbb{Z}\big/6$) and $\mathbb{Z}\big/2$;
		\item $H^{10}\left(\mathbb{Z}\big/3,2\right)$ is an extension of sub-quotients of $\mathbb{Z}\big/2$, $\mathbb{Z}\big/3$ and $\mathbb{Z}\big/2$.
	\end{compactenum}
	Meanwhile, it follows from universal coefficient theorem that $H^{q+1}\left(\mathbb{Z}\big/3,2\right)\cong \mathrm{Ext}\left(H_{q}\left(\mathbb{Z}\big/p,2\right),\mathbb{Z}\right)\cong\mathbb{Z}\big/3^{e_{q}}$ for $7\leqslant q\leqslant9$. Therefore, $e_{7}=e_{8}=e_{9}=1$.
\end{pf}
	
	\section{The Arf type invariants}\label{Section: Arf inv}

In this section we introduce the Arf type invariants $\Xi_{odd}$ applied in Section \ref{Section: Prove of main thm}. We will show that they serve as the obstructions to existence of certain Lagrangian. We also define the Arf type invariant $\Xi_{even}$ for even forms and establish certain properties of these Arf type invariants. This section is of independent algebraic interest. 

Let $\mathcal{V}\cong\mathbb{Z}^{2n}$ be a free abelian group of rank $2n$. Let $\beta$ be a symmetric unimodular bilinear form on $\mathcal{V}$ with vanishing signature (no prescribed parity). Let $p$ be a prime and let $\mathcal{V}\xrightarrow{g}\mathbb{Z}\big/p$ be a homomorphism. We are interested in whether $\ker g$ contains a Lagrangian of $(\mathcal{V},\beta)$, and we expect to express this obstruction im terms of certain numerical invariant of the triple $(\mathcal{V},\beta,g)$. For this purpose we introduce certain Arf type invariants on $(\mathcal{V},\beta,g)$. We will show that these invariants are well-defined and that $\ker g$ contains a Lagrangian if and only if such an Arf invariant vanishes (Propositions \ref{Proposition: sym unimodular odd bilinear (V, lambda) with vanishing signature admits Lagrangian contained in ker of epim onto Z/2}, \ref{Proposition: sym unimodular even bilinear (V, lambda) with vanishing signature admits Lagrangian contained in ker of epim onto Z/2}). We will also establish properties that concern orthogonal direct sums (Proposition \ref{Proposition: Arf inv of direct sum}).

\begin{prop}\label{Proposition: sym unimodular odd bilinear (V, lambda) with vanishing signature admits Lagrangian contained in ker of epim onto Z/2}
	Let $(\mathcal{V},\beta,g)$ be given as above and suppose $\beta$ is odd. Then $(\mathcal{V},\beta)\cong{D}^{n}$ admits a standard orthogonal basis $\left\{e_{i},f_{i}:1\leqslant i\leqslant n\right\}$, namely $\beta\left(e_{i},e_{j}\right)=-\beta\left(f_{i},f_{j}\right)=\delta_{i,j}$ and $\beta\left(e_{i},f_{j}\right)=0$. Define
		$$\Xi_{odd}\left(\mathcal{V},\beta,g\right):=\sum_{i=1}^{n}g\left(e_{i}\right)^{2}-\sum_{i=1}^{n}g\left(f_{i}\right)^{2}.$$
	For simplicity we also denote by $\Xi_{odd}\left({g}\right)$. Then we have
	\begin{compactenum}
		\item $\Xi_{odd}\left({g}\right)$ does not depend on choices of standard orthogonal bases for $(\mathcal{V},\beta)$.
		\item $(\mathcal{V},\beta)$ admits a Lagrangian contained in $\ker g$ if and only if $\Xi_{odd}\left({g}\right)=0$.
	\end{compactenum}
\end{prop}

\begin{prop}\label{Proposition: sym unimodular even bilinear (V, lambda) with vanishing signature admits Lagrangian contained in ker of epim onto Z/2}
	Let $(\mathcal{V},\beta,g)$ be given as above and suppose $\beta$ is even. Then $(\mathcal{V},\beta)\cong{H}^{n}$ admits a hyperbolic basis $\left\{e_{i},f_{i}:1\leqslant i\leqslant n\right\}$, namely $\beta\left(e_{i},e_{j}\right)=\beta\left(f_{i},f_{j}\right)=0$ and $\beta\left(e_{i},f_{j}\right)=\delta_{i,j}$. Define
		$$\Xi_{even}\left(\mathcal{V},\beta,g\right):=\begin{cases}
			\sum_{i=1}^{n}g\left(e_{i}\right)g\left(f_{i}\right),\ &p=2,\\
			2\sum_{i=1}^{n}g\left(e_{i}\right)g\left(f_{i}\right),\ &\text{otherwise}.
		\end{cases}$$
		For simplicity we also denote by $\Xi_{even}\left({g}\right)$. Then we have
		\begin{compactenum}
			\item $\Xi_{even}\left({g}\right)$ does not depend on choices of standard orthogonal bases for $(\mathcal{V},\beta)$.
			\item $(\mathcal{V},\beta)$ admits a Lagrangian contained in $\ker g$ if and only if $\Xi_{even}\left({g}\right)=0$.
		\end{compactenum}
\end{prop}

\begin{rmk}\label{Remark: on Arf type inv}
	When $p=2$, the Arf type invariant for odd forms can be simplified as
		$$\Xi_{odd}\left(\mathcal{V},\beta,g\right)=\sum_{i=1}^{n}g\left(e_{i}\right)+\sum_{i=1}^{n}g\left(f_{i}\right).$$
	The definitions of Arf type invariants for even forms are slightly different, with an extra factor $2$ added when $p$ is an odd prime. Clearly whether the invertible element $2\in\mathbb{Z}\big/p$ is multiplied does not affect whether the invariant vanishes, and this extra factor enables us to better describe the Arf type invariants of directed sum, as the following Lemma implies.
\end{rmk}

\begin{prop}\label{Proposition: Arf inv of direct sum}
	Let $(\mathcal{V},\beta,g)$ and $\left(\mathcal{V}',\beta',g'\right)$ be two such triples given as above. Then we have
	\begin{align*}
		\Xi_{odd}\left(g\oplus g'\right)&=\begin{cases}
			\Xi_{odd}(g),\ &\text{if $\left(\mathcal{V},\beta\right)$ is odd, $\left(\mathcal{V}',\beta'\right)$ is even and $p=2$,}\\
			\Xi_{odd}(g)+\Xi_{even}\left(g'\right),\ &\text{if $\left(\mathcal{V},\beta\right)$ is odd, $\left(\mathcal{V}',\beta'\right)$ is even and $p\geqslant3$,}\\
			\Xi_{odd}(g)+\Xi_{odd}\left(g'\right),\ &\text{if $\left(\mathcal{V},\beta\right)$ and $\left(\mathcal{V}',\beta'\right)$ are odd;}
			\end{cases}\\
		\Xi_{even}\left(g\oplus g'\right)&=\Xi_{even}(g)+\Xi_{even}\left(g'\right),\quad\text{if $\left(\mathcal{V},\beta\right)$ and $\left(\mathcal{V}',\beta'\right)$ are even.}
	\end{align*}
\end{prop}

The proofs of Propositions \ref{Proposition: sym unimodular odd bilinear (V, lambda) with vanishing signature admits Lagrangian contained in ker of epim onto Z/2}$\sim$\ref{Proposition: Arf inv of direct sum} are organized as follows. First we prove Statement 1 of Propositions \ref{Proposition: sym unimodular odd bilinear (V, lambda) with vanishing signature admits Lagrangian contained in ker of epim onto Z/2} and \ref{Proposition: sym unimodular even bilinear (V, lambda) with vanishing signature admits Lagrangian contained in ker of epim onto Z/2}, which can be treated uniformly. Then we prove Statement 2 of Propositions \ref{Proposition: sym unimodular odd bilinear (V, lambda) with vanishing signature admits Lagrangian contained in ker of epim onto Z/2} and \ref{Proposition: sym unimodular even bilinear (V, lambda) with vanishing signature admits Lagrangian contained in ker of epim onto Z/2} for the case $p=2$. Next we prove Proposition \ref{Proposition: Arf inv of direct sum}. Finally we deal with Statement 2 of Propositions \ref{Proposition: sym unimodular odd bilinear (V, lambda) with vanishing signature admits Lagrangian contained in ker of epim onto Z/2} and \ref{Proposition: sym unimodular even bilinear (V, lambda) with vanishing signature admits Lagrangian contained in ker of epim onto Z/2} for the case $p\geqslant3$.

\begin{pf}[of Proposition \ref{Proposition: sym unimodular odd bilinear (V, lambda) with vanishing signature admits Lagrangian contained in ker of epim onto Z/2}, Statement 1]
	Suppose $(\mathcal{V},\beta)$ is odd. It follows from \cite[Chapter I, Theorem 5.3]{HosemullerMilnor13} that if $(\mathcal{V},\beta)$ is odd, then $(\mathcal{V},\beta)$ admits a standard orthogonal basis $\left\{e_{i},f_{i}:1\leqslant i\leqslant n\right\}$. Denote
	$$g\begin{pmatrix}
		e_{1} & \cdots & e_{n} & f_{1} & \cdots & f_{n}
	\end{pmatrix}=\begin{pmatrix}
		a & b
	\end{pmatrix},$$
	where $a,b\in\left(\mathbb{Z}\big/p\right)^{n}$ are viewed as row vectors, and we have
	\begin{align*}
		\sum_{i=1}^{n}g\left(e_{i}\right)^{2}-\sum_{i=1}^{n}g\left(f_{i}\right)^{2}
		=aa^{T}-bb^{T}.
	\end{align*}
	
	The orthogonal group $O(\mathcal{V},\beta)$ acts transitively on hyperbolic bases of $O(\mathcal{V},\beta)$, and under the given hyperbolic basis $O(\mathcal{V},\beta)$ can be described explicitly as matrix group
	$$O(\mathcal{V},\beta)\cong O_{n,n}(\mathbb{Z})=\left\{\begin{pmatrix}
		A & C \\ B & D
	\end{pmatrix}\in GL(2n,\mathbb{Z})\middle|
	\left.
	\begin{alignedat}{2}
		AA^{T}-CC^{T}=BB^{T}-DD^{t}&=I_{n},\\
		AB^{T}-CD^{T}&=O.
	\end{alignedat}
	\right. \right\}$$
	Take any $m\in O(\mathcal{V},\beta)$, let $\begin{pmatrix}
		A & C \\ B & D
	\end{pmatrix}$ be the matrix representation of $m$ and now we compute
	\begin{align*}
		\sum_{i=1}^{n}g\left(me_{i}\right)^{2}-\sum_{i=1}^{n}g\left(mf_{i}\right)^{2}&=\begin{pmatrix}
			a & b
		\end{pmatrix}\begin{pmatrix}
			A & C \\ B & D
		\end{pmatrix}\left(\begin{pmatrix}
			a & b
		\end{pmatrix}\begin{pmatrix}
			A & C \\ B & D
		\end{pmatrix}\right)^{T}\\
		&=\begin{pmatrix}
			a & b
		\end{pmatrix}\begin{pmatrix}
			A & C \\ B & D
		\end{pmatrix}\begin{pmatrix}
			A & C \\ B & D
		\end{pmatrix}^{T}\begin{pmatrix}
			a &  b
		\end{pmatrix}^{T}\\
		&=\begin{pmatrix}
			a & b
		\end{pmatrix}\begin{pmatrix}
			a & b
		\end{pmatrix}^{T}\\
		&=aa^{T}-bb^{T}=\sum_{i=1}^{n}g\left(e_{i}\right)^{2}-\sum_{i=1}^{n}g\left(f_{i}\right)^{2}.
	\end{align*}
	Therefore, $\Xi_{odd}(g)$ is well-defined.
\end{pf}

\begin{pf}[of Proposition \ref{Proposition: sym unimodular even bilinear (V, lambda) with vanishing signature admits Lagrangian contained in ker of epim onto Z/2}, Statement 1]
	Suppose $(\mathcal{V},\beta)$ is even. Again it follows from \cite[Chapter I, Theorem 5.3]{HosemullerMilnor13} that $(\mathcal{V},\beta)$ is hyperbolic and admits a hyperbolic basis $\left\{e_{i},f_{i}:1\leqslant i\leqslant n\right\}$. Denote
	$$g\begin{pmatrix}
		e_{1} & \cdots & e_{n} & f_{1} & \cdots & f_{n}
	\end{pmatrix}=\begin{pmatrix}
		a & b
	\end{pmatrix},$$
	where $a,b\in\left(\mathbb{Z}\big/2\right)^{n}$ are viewed as row vectors, and we have $\sum_{i=1}^{n}g\left(e_{i}\right)g\left(f_{i}\right)=ab^{T}$.
	
	The orthogonal group $O(\mathcal{V},\beta)$ acts transitively on the hyperbolic bases of $O(\mathcal{V},\beta)$, and under the given hyperbolic basis $O(\mathcal{V},\beta)$ can be described explicitly as matrix group
	$$O(\mathcal{V},\beta)\cong Sp_{2n}(\mathbb{Z})=\left\{\begin{pmatrix}
		A & C \\ B & D
	\end{pmatrix}\in GL(2n,\mathbb{Z})\middle|
	\left.
	\begin{alignedat}{2}
		AC^{T}+CA^{T}=BD^{T}+DB^{t}&=O,\\
		AD^{T}+CB^{T}&=I_{n}.
	\end{alignedat}
	\right. \right\}$$
	Take any $m\in O(\mathcal{V},\beta)$, let $\begin{pmatrix}
		A & C \\ B & D
	\end{pmatrix}$ be the matrix representation of $m$ and now we compute $\sum_{i=1}^{n}g\left(me_{i}\right)g\left(mf_{i}\right)$:
	\begin{align*}
		\sum_{i=1}^{n}g\left(me_{i}\right)g\left(mf_{i}\right)&=(aA+bB)(aC+bD)^{T}\\
		&=aAC^{T}a^{T}+aAD^{T}b^{T}+bBC^{T}a^{T}+bBD^{T}b^{T}\\
		&=aAD^{T}b^{T}+bBC^{T}a^{T}&\left(AC^{T}\text{ and }BD^{T}\text{ are skew-symmetric}\right)\\
		&=aAD^{T}b^{T}+aCB^{T}b^{T}&\left(\text{the transpose of a number is itself}\right)\\
		&=ab^{T}=\sum_{i=1}^{n}g\left(e_{i}\right)g\left(f_{i}\right).
	\end{align*}
	Therefore, $\Xi_{even}(g)$ is well-defined.
\end{pf}

\begin{pf}[of Proposition \ref{Proposition: sym unimodular odd bilinear (V, lambda) with vanishing signature admits Lagrangian contained in ker of epim onto Z/2}, Statement 2, the case $p=2$]
	Now we prove that $(\mathcal{V},\beta)$ admits a Lagrangian contained in $\ker g$ if and only if $\Xi_{odd}(g)=0$ for the case $p=2$. We begin with the ``only if'' part. Suppose $L$ is a Lagrangian of $(\mathcal{V},\beta)$ contained in $\ker g$. Then $\mathcal{V}$ admits a basis $\left\{e_{i},f_{i}:1\leqslant i\leqslant n\right\}$ such that $L=\left<e_{i}:1\leqslant i\leqslant n\right>$, $g\left(e_{i}\right)=0$ and $\beta\left(e_{i},f_{j}\right)=\beta\left(f_{i},f_{j}\right)=\delta_{i,j}$. Then we set $e_{i}'=f_{i}$, $f_{i}'=e_{i}-f_{i}$, and it is straightforward to verify that $\left\{e_{i}',f_{i}':1\leqslant i\leqslant n\right\}$ is a standard orthogonal basis for $(\mathcal{V},\beta)$. Hence
	\begin{align*}
		\Xi_{odd}(g)&=\sum_{i=1}^{n}g\left(e_{i}'\right)^{2}-\sum_{i=1}^{n}g\left(f_{i}'\right)^{2}\\
		&=\sum_{i=1}^{n}g\left(f_{i}\right)^{2}-\sum_{i=1}^{n}\left(g\left(e_{i}\right)-g\left(f_{i}\right)\right)^{2}=0.
	\end{align*}
	
	To show the ``if'' part, we directly construct a Lagrangian contained in $\ker g$ when $\Xi_{odd}(g)=0$. Since $\Xi_{odd}(g)=0$, $(\mathcal{V},\beta)$ admits a standard orthogonal basis $\left\{e_{i},f_{i}:1\leqslant i\leqslant n\right\}$ with $\sum_{i=1}^{n}g\left(e_{i}\right)+\sum_{i=1}^{n}g\left(f_{i}\right)=0$, and we may permute them suitably into a new basis $\left\{e_{i}':1\leqslant i\leqslant 2n\right\}$ such that
	\begin{compactenum}
		\item $\beta\left(e_{i}',e_{i}'\right)=\varepsilon_{i}=\pm1$ and $\beta\left(e_{i}',e_{j}'\right)=0$ for $i\neq j$; 
		\item $\sum_{i=1}^{2n}\varepsilon_{i}=0$;
		\item there exists $n_{0}$ such that $g\left(e_{i}'\right)=1$ for all $1\leqslant i\leqslant 2n_{0}$ and $g\left(e_{j}'\right)=0$ for all $2n_{0}+1\leqslant j\leqslant 2n$.
	\end{compactenum}
	Then we set $e_{i}''=e_{2i-1}'-\varepsilon_{2i-1}\varepsilon_{2i}e_{2i}'$ for $1\leqslant i\leqslant n$, and it is straightforward to verify that $g\left(e_{i}''\right)=0$, that $e_{i}''$ is primitive in $\left<e_{2i-1}',e_{2i}'\right>$ for all $1\leqslant i\leqslant n$ and $L:=\left<e_{i}'':1\leqslant i\leqslant n\right>$ is a Lagrangian of $(\mathcal{V},\beta)$ contained in $\ker g$. This completes the proof of Proposition \ref{Proposition: sym unimodular odd bilinear (V, lambda) with vanishing signature admits Lagrangian contained in ker of epim onto Z/2}.
\end{pf}

\begin{pf}[of Proposition \ref{Proposition: sym unimodular even bilinear (V, lambda) with vanishing signature admits Lagrangian contained in ker of epim onto Z/2}, Statement 2, the case $p=2$]
	Now we prove that $(\mathcal{V},\beta)$ admits a Lagrangian contained in $\ker g$ if and only if $\Xi_{even}(g)=0$ for the case $p=2$. We begin with the ``only if'' part. Let $L$ be a Lagrangian of $(\mathcal{V},\beta)$ contained in $\ker g$ and let $\left\{e_{1},\cdots,e_{n}\right\}$ be a basis of $L$. Then $g\left(e_{i}\right)=0$. Extend the basis of $L$ to a hyperbolic basis $\left\{e_{1},f_{1},\cdots,e_{n},f_{n}\right\}$ of $(\mathcal{V},\beta)$. Then it follows from $g\left(e_{i}\right)=0$ that $\Xi_{even}(g)=0$.
	
	To show the ``if'' part, we directly construct a Lagrangian contained in $\ker g$ if $\Xi_{even}(g)=0$. Suppose $\Xi_{even}(g)=0$, and after a suitable permutation we may assume that  $\left\{e_{1},f_{1},\cdots,e_{n},f_{n}\right\}$ is a hyperbolic basis of $(\mathcal{V},\beta)$ such that
	\begin{compactenum}
		\item $g\left(e_{i}\right)g\left(f_{i}\right)=1$, $1\leqslant i\leqslant 2l$, $4l\leqslant n$;
		\item $g\left(e_{j}\right)g\left(f_{j}\right)=0$, $2l+1\leqslant j\leqslant n$.
	\end{compactenum}
	For each $1\leqslant i\leqslant 2l$ we have $g\left(e_{i}\right)=g\left(f_{i}\right)=1$. We set $e_{2i-1}'=e_{2i-1}+e_{2i}$ and $e_{2i}'=f_{2i-1}-f_{2i}$, and it can be directly verified that $\left<e_{2i-1}',e_{2i}'\right>$ is a direct summand of the hyperbolic summand $\left<e_{2i-1},f_{2i-1},e_{2i},f_{2i}\right>$ when $1\leqslant i\leqslant l$, and $g\left(e_{i}'\right)=0$, $\beta\left(e_{i}',e_{i}'\right)=0$ when $1\leqslant i\leqslant 2l$. For each $2l+1\leqslant j\leqslant n$ we have $g\left(e_{j}\right)g\left(f_{j}\right)=0$, and there exists $e_{j}'\in\left\{e_{j},f_{j}\right\}$ such that $g\left(e_{j}'\right)=0$. In particular $\left<e_{j}'\right>$ is a direct summand of the hyperbolic summand $\left<e_{j},f_{j}\right>$ and $g\left(e_{j}'\right)=0$, $\beta\left(e_{j}',e_{j}'\right)=0$ when $2l+1\leqslant j\leqslant n$. Then $\left<e_{i}':1\leqslant i\leqslant n\right>$ is a Lagrangian of $(\mathcal{V},\beta)$ contained in $\ker g$. This completes the proof of Proposition \ref{Proposition: sym unimodular even bilinear (V, lambda) with vanishing signature admits Lagrangian contained in ker of epim onto Z/2}.
\end{pf}

\begin{pf}[of Proposition \ref{Proposition: Arf inv of direct sum}]
	The proof is clear when $\left(\mathcal{V},\beta\right)$ and $\left(\mathcal{V}',\beta'\right)$ are both odd or both even, and we only need to prove the case where $\left(\mathcal{V},\beta\right)$ is odd and $\left(\mathcal{V}',\beta'\right)$ is even. According to the classification of symmetric unimodular bilinear forms with vanishing signatures over $\mathbb{Z}$ (\cite[Chapter I, Theorem 5.3]{HosemullerMilnor13}), it suffices to further reduce to the case that $\left(\mathcal{V},\beta\right)$ and $\left(\mathcal{V}',\beta'\right)$ both have rank $2$.
	
	Let $\left\{e_{1},f_{1}\right\}$ be a standard orthogonal basis for $\left(\mathcal{V},\beta\right)$ and let $\left\{e_{2},f_{2}\right\}$ be a hyperbolic basis for $\left(\mathcal{V}',\beta'\right)$. Let $\mathcal{V}^{1}\xrightarrow{g}\mathbb{Z}\big/p$ and $\mathcal{V}'\xrightarrow{g'}\mathbb{Z}\big/p$ be homomorphisms. We shall give a standard orthogonal basis for $\left(\mathcal{V}\oplus\mathcal{V}',\beta\oplus\beta'\right)$ in terms of $e_{i}$, $f_{i}$,  compute $\Xi_{odd}\left(g\oplus g'\right)$ and express the result in terms of $\Xi_{odd}(g)$, $\Xi_{even}\left(g'\right)$. It is straightforward to verify that
	\begin{align*}
		u_{1}&=e_{1},\\
		u_{2}&=f_{1}+e_{2}+f_{2},\\
		v_{1}&=f_{1}+e_{2},\\
		v_{2}&=f_{1}+f_{2}
	\end{align*}
	is a standard orthogonal basis for $\left(\mathcal{V}\oplus\mathcal{V}',\beta\oplus\beta'\right)$. Hence by definition, when $p=2$ we have
	\begin{align*}
		\Xi_{odd}\left(g\oplus g'\right)&=\left(g\oplus g'\right)\left(u_{1}\right)+\left(g\oplus g'\right)\left(u_{2}\right)+\left(g\oplus g'\right)\left(v_{1}\right)+\left(g\oplus g'\right)\left(v_{2}\right)\\
		&=g\left(e_{1}\right)+\left(g\left(f_{1}\right)+g'\left(e_{2}\right)+g'\left(f_{2}\right)\right)+\left(g\left(f_{1}\right)+g'\left(e_{2}\right)\right)+\left(g\left(f_{1}\right)+g'\left(f_{2}\right)\right)\\
		&=g\left(e_{1}\right)+g\left(f_{1}\right)=\Xi_{odd}(g),
	\end{align*}
	and when $p\geqslant3$ we have
	\begin{align*}
		\Xi_{odd}\left(g\oplus g'\right)&=\left(g\oplus g'\right)\left(u_{1}\right)^{2}+\left(g\oplus g'\right)\left(u_{2}\right)^{2}-\left(g\oplus g'\right)\left(v_{1}\right)^{2}-\left(g\oplus g'\right)\left(v_{2}\right)^{2}\\
		&=g\left(e_{1}\right)^{2}+\left(g\left(f_{1}\right)+g'\left(e_{2}\right)+g'\left(f_{2}\right)\right)^{2}-\left(g\left(f_{1}\right)+g'\left(e_{2}\right)\right)^{2}-\left(g\left(f_{1}\right)+g'\left(f_{2}\right)\right)^{2}\\
		&=g\left(e_{1}\right)^{2}-g\left(f_{1}\right)^{2}+2g'\left(e_{2}\right)g'\left(f_{2}\right)\\
		&=\Xi_{odd}(g)+\Xi_{even}\left(g'\right).
	\end{align*}
	This completes the proof of Proposition \ref{Proposition: Arf inv of direct sum}.
\end{pf}

It remains to prove Statement 2 of Propositions \ref{Proposition: sym unimodular odd bilinear (V, lambda) with vanishing signature admits Lagrangian contained in ker of epim onto Z/2} and \ref{Proposition: sym unimodular even bilinear (V, lambda) with vanishing signature admits Lagrangian contained in ker of epim onto Z/2} when $p\geqslant3$. In this more general case we require the following lemma concerning the existence of specific primitive element in a lattice, whose proof is postponed to the end of this section.

\begin{lem}\label{Lemma: Existence of integral primitive isotropic elt}
	Suppose $\mathcal{V}\cong\mathbb{Z}^{2n}$ and $\beta$ is a symmetric unimodular bilinear form on $V$ with vanishing signature. Let $p$ be an odd prime, denote $\overline{\mathcal{V}}=\mathcal{V}\big /p\mathcal{V}$, $\mathcal{V}\xrightarrow{\overline{\cdot}}\overline{\mathcal{V}}$ the quotient homomorphism and $\overline{\mathcal{V}}\times \overline{\mathcal{V}}\xrightarrow{\overline{\beta}}\mathbb{Z}\big/p$ the induced symmetric bilinear form. Let $\xi\in \left(\overline{\mathcal{V}},\overline{\beta}\right)$ be an isotropic element, i.e. $\overline{\beta}(\xi,\xi)=0\in\mathbb{Z}\big/p$.
	
	Then $(\mathcal{V},\beta)$ admits a primitive isotropic element $x$ such that $\overline{\beta}\left(\xi,\overline{x}\right)=0$.
\end{lem}

\begin{pf}[of Proposition \ref{Proposition: sym unimodular even bilinear (V, lambda) with vanishing signature admits Lagrangian contained in ker of epim onto Z/2}, Statement 2, the case $p\geqslant3$]
	Now we prove that $(\mathcal{V},\beta)$ admits a Lagrangian contained in $\ker g$ if and only if $\Xi_{even}(g)=0$ for the case $p\geqslant3$. The ``only if'' part is exactly the same as the case $p=2$, and we focus on the proof of the ``if'' part.
	
	We prove by induction on $n=\frac{1}{2}\mathrm{rank}\ \mathcal{V}$. When $n=1$, let $\{e,f\}$ be a hyperbolic basis of $(\mathcal{V},\beta)$. Then $\Xi_{even}(\mathcal{V},\beta,g)=0$ implies $g(e)g(f)=0\in\mathbb{Z}\big/p$. Since $p$ is prime, either $g(e)=0$ or $g(f)=0$. Without losing of generality we suppose $g(e)=0$, and $L=\left<e\right>$ is the desired Lagrangian.
	
	Suppose the sufficiency is true for all such tripples $(\mathcal{V},\beta,g)$ with $\beta$ even and $n=\frac{1}{2}\mathrm{rank}\ \mathcal{V}$ not exceeding $k$ for some $k\geqslant1$. Given $\mathcal{V}\cong\mathbb{Z}^{2k+2}$, a symmetric unimodular bilinear even form $\beta$ on $\mathcal{V}$ and a homomorphism $\mathcal{V}\xrightarrow{g}\mathbb{Z}\big/p$ with $\Xi_{even}(\mathcal{V},\beta,g)=0$. Let us apply the notations of Lemma \ref{Lemma: Existence of integral primitive isotropic elt} that overline implies mod $p$ reductions. Since $\left(\overline{\mathcal{V}},\overline{\beta}\right)$ is again unimodular, there is a unique element $\xi\in\overline{\mathcal{V}}$ such that $g\left(v\right)=\overline{\beta}\left(\xi,v\right)$ for all $v\in\overline{\mathcal{V}}$, and it is routine to verify that $\Xi_{even}(\mathcal{V},\beta,g)=\overline{\beta}(\xi,\xi)=0$ by computing under a hyperbolic basis. Then by Lemma \ref{Lemma: Existence of integral primitive isotropic elt} $\left(\mathcal{V},\beta\right)$ admits a primitive isotropic element $x$ contained in $\ker g$. Since $(\mathcal{V},\beta)$ is even, symmetric and unimodular with vanishing signature and $x$ is primitive, we can find another vector $y$ such that $\mathcal{H}_{0}:=\left<x,y\right>$ is a hyperbolic summand of $(\mathcal{V},\beta)$. 
	
	Denote $h_{0}=\beta|_{\mathcal{H}_{0}}$, $\mathcal{V}_{1}=\left(\mathcal{H}_{0}\right)^{\perp}$, $\beta_{1}=\beta|_{\mathcal{V}_{1}}$, $g_{0}=g|_{\mathcal{H}_{0}}$ and $g_{1}=g|_{\mathcal{V}_{1}}$. Then we have the decompositions
	\begin{align*}
		(\mathcal{V},\beta)&=\left(\mathcal{H}_{0},h_{0}\right)\oplus\left(\mathcal{V}_{1},\beta_{1}\right),\\
		g&=g_{0}\oplus g_{1},
	\end{align*}
	and by Proposition \ref{Proposition: Arf inv of direct sum} we have 
	\[\Xi_{even}\left(\mathcal{V},\beta,g\right)=\Xi_{even}\left(\mathcal{H}_{0},h_{0},g_{0}\right)+\Xi_{even}\left(\mathcal{V}_{1},\beta_{1},g_{1}\right).\]
	By construction $\{x,y\}$ is a hyperbolic basis of $\left(\mathcal{H}_{0},h_{0}\right)$ with $g_{0}(x)=0$. Hence $\Xi_{even}\left(\mathcal{H}_{0},h_{0},g_{0}\right)=0$ and $\Xi_{even}\left(\mathcal{V}_{1},\beta_{1},g_{1}\right)=\Xi_{even}\left(\mathcal{V},\beta,g\right)=0$. Now $\mathrm{rank}\ \mathcal{V}$ admits a Lagrangian $L_{1}$ contained in $\ker g_{1}$. Hence $L=\left<x\right>\oplus L_{1}$ is a Lagrangian of $\left(\mathcal{V},\beta\right)$ contained in $\ker\left(g_{0}\oplus g_{1}\right)=\ker g$. By induction the sufficiency is true for all such tripples $(\mathcal{V},\beta,g)$ with $\beta$ even and $\frac{1}{2}\mathrm{rank}\ \mathcal{V}$ being any positive integer.
\end{pf}

\begin{pf}[of Proposition \ref{Proposition: sym unimodular odd bilinear (V, lambda) with vanishing signature admits Lagrangian contained in ker of epim onto Z/2}, Statement 2, the case $p\geqslant3$]
	Now we prove that $(\mathcal{V},\beta)$ admits a Lagrangian contained in $\ker g$ if and only if $\Xi_{odd}(g)=0$ for the case $p\geqslant3$. The ``only if'' part is exactly the same as the case $p=2$, and we focus on the proof of the ``if'' part.
	
	We argue by induction on $n=\frac{1}{2}\mathrm{rank}\ \mathcal{V}$. When $n=1$, let $\{e,f\}$ be standard orthogonal basis of $(\mathcal{V},\beta)$. Then $\Xi_{odd}(\mathcal{V},\beta,g)=0$ implies $g(e)^{2}-g(f)^{2}=0\in\mathbb{Z}\big/p$. Since $p$ is prime, $g(e)=\varepsilon g(f)$ for some $\varepsilon=\pm1$, and $L=\left<e-\varepsilon f\right>$ is the desired Lagrangian.
	
	Suppose the sufficiency is true for all such tripples $(\mathcal{V},\beta,g)$ with $\beta$ odd and $n=\frac{1}{2}\mathrm{rank}\ \mathcal{V}$ not exceeding $k$ for some $k\geqslant1$. Given $\mathcal{V}\cong\mathbb{Z}^{2k+2}$, a symmetric unimodular bilinear odd form $\beta$ on $\mathcal{V}$ and a homomorphism $\mathcal{V}\xrightarrow{g}\mathbb{Z}\big/p$ with $\Xi_{odd}(\mathcal{V},\beta,g)=0$. Let us apply the notations of Lemma \ref{Lemma: Existence of integral primitive isotropic elt} for mod $p$ reductions. Since $\left(\overline{\mathcal{V}},\overline{\beta}\right)$ is again unimodular, there is a unique element $\xi\in\overline{\mathcal{V}}$ such that $g\left(v\right)=\overline{\beta}\left(\xi,v\right)$ for all $v\in\overline{\mathcal{V}}$, and it is routine to verify that $\Xi_{odd}(\mathcal{V},\beta,g)=\overline{\beta}(\xi,\xi)=0$ by computing under a standard orthogonal basis. Then by Lemma \ref{Lemma: Existence of integral primitive isotropic elt} $\left(\mathcal{V},\beta\right)$ admits a primitive isotropic element $x$ contained in $\ker g$. Since $(\mathcal{V},\beta)$ is symmetric and unimodular with vanishing signature and $x$ is primitive, we can find another vector $y$ such that $\beta(x,y)=1$ and $\mathcal{V}_{0}:=\left<x,y\right>$ is a direct summand of $(\mathcal{V},\beta)$. We shall mention that the parity of $\beta(y,y)$ cannot be determined, and the best we can do is to add certain multiple of $x$ to $y$ and assume that $\beta(y,y)=0$ or $1$.
	
	Denote $\beta_{0}=\beta|_{\mathcal{V}_{0}}$, $\mathcal{V}_{1}=\left(\mathcal{V}_{0}\right)^{\perp}$, $\beta_{1}=\beta|_{\mathcal{V}_{1}}$, $g_{0}=g|_{\mathcal{V}_{0}}$ and $g_{1}=g|_{\mathcal{V}_{1}}$. Then we have the decompositions
	\begin{align*}
		(\mathcal{V},\beta)&=\left(\mathcal{V}_{0},\beta_{0}\right)\oplus\left(\mathcal{V}_{1},\beta_{1}\right),\\
		g&=g_{0}\oplus g_{1}.
	\end{align*}
	If $\beta(y,y)=0$ so that $\left(\mathcal{V}_{0},\beta_{0}\right)$ is even, then $\left(\mathcal{V}_{1},\beta_{1}\right)$ must be odd,
	and by Proposition \ref{Proposition: Arf inv of direct sum} we have 
	\[\Xi_{odd}\left(\mathcal{V},\beta,g\right)=\Xi_{even}\left(\mathcal{H}_{0},h_{0},g_{0}\right)+\Xi_{odd}\left(\mathcal{V}_{1},\beta_{1},g_{1}\right).\]
	By construction $\{x,y\}$ is a hyperbolic basis of $\left(\mathcal{H}_{0},h_{0}\right)$ with $g_{0}(x)=0$. Hence $\Xi_{even}\left(\mathcal{V}_{0},\beta_{0},g_{0}\right)=0$ and $\Xi_{odd}\left(\mathcal{V}_{1},\beta_{1},g_{1}\right)=\Xi_{odd}\left(\mathcal{V},\beta,g\right)=0$. Now $\mathcal{V}$ admits a Lagrangian $L_{1}$ contained in $\ker g_{1}$. 
	
	If $\beta(y,y)=1$ so that $\left(\mathcal{V}_{0},\beta_{0}\right)$ is odd, then $\{x,y\}$ is a standard orthogonal basis with $g_{0}(x)=0$ and thus $\Xi_{odd}\left(\mathcal{V}_{0},\beta_{0},g_{0}\right)=0$, and $\left(\mathcal{V}_{1},\beta_{1}\right)$ can be either even or odd. When $\left(\mathcal{V}_{1},\beta_{1}\right)$ is even, we have $\Xi_{even}\left(\mathcal{V}_{1},\beta_{1},g_{1}\right)=\Xi_{odd}\left(\mathcal{V},\beta,g\right)=0$, and by the case $p\geqslant3$ of Proposition \ref{Proposition: sym unimodular even bilinear (V, lambda) with vanishing signature admits Lagrangian contained in ker of epim onto Z/2}, Statement 2 \(\left(\mathcal{V}_{1},\beta_{1}\right)\) admits a Lagrangian $L_{1}$ contained in $\ker g_{1}$. When $\left(\mathcal{V}_{1},\beta_{1}\right)$ is odd, we have $\Xi_{odd}\left(\mathcal{V}_{1},\beta_{1},g_{1}\right)=\Xi_{odd}\left(\mathcal{V},\beta,g\right)=0$, and by induction hypothesis \(\left(\mathcal{V}_{1},\beta_{1}\right)\) admits a Lagrangian $L_{1}$ contained in $\ker g_{1}$. 
	
	Hence in either case, $L=\left<x\right>\oplus L_{1}$ is a Lagrangian of $\left(\mathcal{V},\beta\right)$ contained in $\ker\left(g_{0}\oplus g_{1}\right)=\ker g$. By induction the sufficiency is true for all such tripples $(\mathcal{V},\beta,g)$ with $\beta$ even and $\frac{1}{2}\mathrm{rank}\ \mathcal{V}$ being any positive integer.
\end{pf}

\begin{pf}[of Lemma \ref{Lemma: Existence of integral primitive isotropic elt}]
	When $\xi=0$, any primitive isotropic element $x$ of $(\mathcal{V},\beta)$ meets the requirement, and since $(V,\beta)$ is symmetric and unimorular with vanishing signature, it always admits primitive isotropic elements. Hence in the following we require that $\xi\neq0\in \overline{\mathcal{V}}$. 
	
	The treatment differs due to the parity of $(\mathcal{V},\beta)$, and we begin with the even case. Let $\left\{e_{1},f_{1},\cdots,e_{n},f_{n}\right\}$ be a hyperbolic basis. For simplicity we denote $\overline{v}=\rho_{p}v$ for $v\in \mathcal{V}$. Set $\xi=\sum_{i=1}^{n}\left(\overline{c_{i}}\overline{e_{i}}+\overline{d_{i}}\overline{f_{i}}\right)$ with $\overline{c_{i}},\overline{d_{i}}\in\mathbb{Z}\big/p$. Then $\overline{\beta}(\xi,\xi)=0$ implies that $2\sum_{i=1}^{n}\overline{c_{i}}\overline{d_{i}}=0$. Since $\xi\neq0$, $\overline{c_{i}},\overline{d_{i}}$ cannot be simultaneously zero, and we may assume $\overline{c_{1}}\neq0$. Since $\overline{\beta}$ is bilinear, whether $\overline{\beta}(\xi,\xi)$ and $\overline{\beta}\left(\xi,\overline{x}\right)$ vanish is not affected by multiplying a non-zero scalar to $\xi$, and we may further assume that $\overline{c_{1}}=1$.
	
	Now we construct the desired element $x=\sum_{i=1}^{n}\left(u_{i}e_{i}+v_{i}f_{i}\right)\in \mathcal{V}$ directly as follows. Set $u_{1}=1$. For each $i\geqslant2$, let $u_{i}$ and $v_{i}$ be any integral liftings of $\overline{c_{i}}$ and $\overline{d_{i}}$ respectively. Finally we set $v_{1}=-\sum_{i=2}^{n}u_{i}v_{i}$. Then $x$ is primitive since $u_{1}=1$, and it is straightforward to verify that $x$ is isotropic and $\overline{\beta}\left(\xi,\overline{x}\right)=0$.
	
	Next we consider the odd case. Let $\left\{e_{1},f_{1},\cdots,e_{n},f_{n}\right\}$ be a standard orthogonal basis of $(\mathcal{V},\beta)$ and set $\xi=\sum_{i=1}^{n}\left(\overline{c_{i}}\overline{e_{i}}+\overline{d_{i}}\overline{f_{i}}\right)$ with $\overline{c_{i}},\overline{d_{i}}\in\mathbb{Z}\big/p$. Then $\overline{\beta}(\xi,\xi)=0$ implies that $\sum_{i=1}^{n}\left(\overline{c_{i}}^{2}-\overline{d_{i}}^{2}\right)=0$.
	
	If there is an index $k$ such that $\overline{c_{k}}^{2}-\overline{d_{k}}^{2}=0$, then we must have $\overline{c_{k}}=\varepsilon\overline{d_{k}}$ for some $\varepsilon=\pm1$. Set $x=e_{k}+\varepsilon f_{k}$. Then $x$ is primitive, and it is straightforward to verify that $x$ is isotropic with $\overline{\beta}\left(\xi,\overline{x}\right)=0$.
	
	Suppose $\overline{c_{i}}^{2}-\overline{d_{i}}^{2}\neq0$ for all indices $i$. In particular $\overline{c_{i}}^{2}-\overline{d_{i}}^{2}\neq0$ and $\overline{c_{i}}\pm\overline{d_{i}}\neq0$ for $i=1,2$. Let $u_{i}$, $v_{i}$ be any integral liftings of $\overline{c_{i}}$, $\overline{d_{i}}$ respectively. Set
	\[x=\left(u_{2}-v_{2}\right)\left(e_{1}-f_{1}\right)-\left(u_{1}+v_{1}\right)\left(e_{2}+f_{2}\right).\]
	Since $\overline{c_{i}}\pm\overline{d_{i}}\neq0$, we have $\overline{x}\neq0\in\overline{V}$ and thus $x\neq0\in \mathcal{V}$. It is straightforward to verify that $x$ is isotropic with $\overline{\beta}\left(\xi,\overline{x}\right)=0$. When $x$ is not primitive, denote $g=\gcd\left(u_{1}+v_{1},u_{2}-v_{2}\right)$, $x'=\frac{1}{g}\cdot x$, and $x'$ is a primitive isotropic element in $(\mathcal{V},\beta)$ with $\overline{\beta}\left(\xi,\overline{x'}\right)=0$. This completes the proof of Lemma \ref{Lemma: Existence of integral primitive isotropic elt}.
\end{pf}
	
	\bibliography{refs}
	
\end{document}